\documentclass[10pt]{amsart}
\usepackage{kantlipsum}
\usepackage{microtype}

\usepackage{amsmath}
\usepackage{amsmath}
\usepackage{amssymb}
\usepackage{tikz}

\usepackage{color}
\newcommand{\e}{\mathsf e}

\DeclareRobustCommand{\subtitle}[1]{\\#1}
\newtheorem{fact}{Fact}
\newtheorem{proposition}{Proposition}
\newtheorem{corollary}{Corollary}
\newtheorem{theorem}{Theorem}
\newtheorem{lemma}{Lemma}
\newtheorem{definition}{Definition}
\newtheorem{example}{Example}
\newtheorem{remark}{Remark}

\newcommand\blfootnote[1]{%
  \begingroup
  \renewcommand\thefootnote{}\footnote{#1}%
  \addtocounter{footnote}{-1}%
  \endgroup
}

\title{An algebraic theory of clones\subtitle{\small with an application to a question of Birkhoff and Maltsev}}

\author{Antonio Bucciarelli}
\address{Institut de Recherche en Informatique Fondamentale,
Universit\'e de Paris, 8 Place Aur\'elie Nemours, 75205 Paris Cedex 13, France}
\email{buccia@irif.fr}
\urladdr{www.irif.fr/$\sim$buccia} 

\author{Antonino Salibra}

\address{Department of Environmental Sciences, Informatics and Statistics,
Universit\`a Ca'Foscari Venezia, Via Torino 155, 30173 Venezia, Italia}
\email{salibra@unive.it}
\urladdr{www.dsi.unive.it/$\sim$salibra}

\begin{document}


\begin{abstract}

The functional composition $f(g_1,\ldots,g_k)$,
the substitution $t[t_1/v_1,\ldots,t_k/v_k]$ of the terms $t_j$'s for the variables  $v_j$'s in $t$, the statement  {\tt  case}$(x,y_1,\ldots,y_k)$ returning one of  $y_j$'s depending on the value of $x$, are all instances of a unique $(k+1)$-ary operation, $q_k(x,y_1,\ldots,y_k)$,  equipped with a set  of 
$k$ constants, representing  respectively projections, variables  and (generalised) truth-values.
Needless to say, these are very basic operations, largely used in computer science and algebra.
This observation is at the root of some recent developments,
at the frontier between universal algebra and computer science.
They  concern in particular the $k$-dimensional generalisations of
Boolean algebras ($k\geq 1$,  the case $k=1$ giving rise to the skew Boolean algebras) 
and the one-sorted, purely algebraic presentation of the notion of clone introduced in this paper.  

\emph{Clone algebras} ($\mathsf{CA}$) are defined by true identities and thus form a variety in the sense of universal algebra.
 The most natural $\mathsf{CA}$s, the ones the axioms are intended to characterise, are algebras of  functions, called \emph{functional clone algebras} ($\mathsf{FCA}$).  
The universe of a $\mathsf{FCA}$, called  \emph{$\omega$-clone}, is a set of infinitary operations from $A^\omega$ into $A$, for a given set $A$, containing the projection $p_i$ and closed under finitary compositions. We show that there exists a bijective correspondence between clones (of finitary operations) and a suitable subclass of $\mathsf{FCA}$s, called \emph{block algebras}.  Given a clone, the corresponding block algebra is obtained by extending the operations of the clone by countably many dummy arguments.
  
One of the main results of this paper is the general representation theorem, where it is shown that every $\mathsf{CA}$ is isomorphic to a $\mathsf{FCA}$. 
 In another result of the paper we prove that  the variety of $\mathsf{CA}$s is generated by the class of block algebras. This implies that every $\omega$-clone is algebraically generated by a suitable family of clones by using direct products, subalgebras and homomorphic images.

We conclude the paper with two applications. In the first one, we use clone algebras  to answer 
a classical question about the lattices of equational theories. 
The second application is to the study of the category $\mathcal{VAR}$ of all varieties. 
We introduce the category $\mathcal{CA}$ of all clone algebras (of arbitrary similarity type) with pure homomorphisms (i.e., preserving only the nullary operators $\e_i$ and the operators $q_n$) as arrows.
We show that the category $\mathcal{VAR}$ is categorically isomorphic to a full subcategory  of $\mathcal{CA}$. We use this result to provide a generalisation of a classical theorem on independent varieties.

\end{abstract}

\maketitle

\renewcommand{\subtitle}[1]{}
 
%
%
%
%

\section{Introduction}\label{sec:intro}

Clones are sets of finitary operations on a given set that contain all the projections and are closed under composition. They play an important role in universal algebra due to the fact that the set of all term operations of an algebra, always forms a clone. Moreover, important properties, like whether a given subset forms a subalgebra, or whether a given map is a homomorphism, do not depend on the specific fundamental operations of the considered algebra, but rather on the clone of its term operations. Hence, comparing clones of algebras is much more suitable than comparing their signatures, in order to classify them according to essentially different behaviours (see \cite{SZ86,T93}).


Some attempts  have been made to encode clones into algebras. A particularly important one led to  the concept of abstract clones  \cite{Co65,T93}, which are many-sorted algebras axiomatising composition of finitary functions and projections. Every abstract clone has a concrete representation as an isomorphic clone of finitary operations. Modulo a caveat about nullary operations, we remark that abstract clones may be recasted as a reformulation of the concept of Lawvere's algebraic theories \cite{Law63}.
The latter constitutes a common category theoretic means to capture equational theories independently of their presentation (i.e. of the chosen similarity type).


Somehow unexpectedly, some recent work at the frontier of theoretical computer science and universal algebra provides  tools for
giving an alternative algebraic account of clones.
There is a thriving literature on abstract treatments of the if-then-else construct of computer science, starting with McCarthy's seminal investigations \cite{MC}. On the algebraic side, one of the most influential approaches originated with Dicker's axiomatisation of Boolean algebras in the language with the if-then-else as primitive \cite{D63}. Accordingly, this construct was treated as a proper algebraic operation $q_2^\mathbf A$ of arity three on algebras $\mathbf A$ whose type contains, besides the ternary term $q_2$, two constants $0$ and $1$, and having the property that for every $a, b\in A$, $q_2^\mathbf A(1^\mathbf A, a, b) = a$ and $q_2^\mathbf A(0^\mathbf A, a, b) = b$. Such algebras, called Church algebras of dimension $2$ in \cite{Bucciasali}, will be termed here $2$-Church algebras.
This approach was generalised in \cite{BLPS18} (see also \cite{Bucciasali,SBLP20}) to algebras $\mathbf A$ having $n$ designated elements $\e_1,\dots, \e_n$ ($n\geq  2$) and a $(n + 1)$-ary operation $q_n$ (a sort of ``generalised if-then-else'') satisfying the identities $q_n(\e_i,a_1,\dots,a_n)=a_i$. These algebras will be called here $n$-Church algebras.

At the root of the most important results in the theory of Boolean algebras (including Stone's representation theorem) there is the simple observation that every element $c\neq 0, 1$ of a Boolean algebra $B$ decomposes $B$ as a Cartesian product $[0, c]\times [c, 1]$ of two nontrivial Boolean algebras. In the more general context of $n$-Church algebras, we say that an element $c$ of an $n$-Church algebra $\mathbf A$ is $n$-central if $\mathbf A$ can be decomposed as the product $\mathbf A/\theta(c, \e_1)\times\dots\times \mathbf A/\theta(c, \e_n)$, where $\theta(c, \e_i)$ is the smallest congruence on $\mathbf A$ that collapses $c$ and $\e_i$. An $n$-Church algebra where every element is $n$-central, called Boolean-like algebra of dimension $n$ in \cite{BLPS18}, will be termed here $n$-Boolean-like algebra ($n$BA, for short). 
Varieties of nBAs share many remarkable properties with the variety of Boolean algebras. In particular, any variety of nBAs is generated by the nBAs of finite cardinality $n$. In the pure case (i.e., when the type includes just the generalised if-then-else $q_n$ and the $n$ constants), the variety is generated by a unique algebra $\mathbf n$ of universe $\{\e_1,\dots, \e_n \}$, so that any pure nBA is, up to isomorphism, a subalgebra of $\mathbf n^X$, for a suitable set $X$. The variety of all $2$BAs in the type $(q_2,0,1)$ is term-equivalent to the variety of Boolean algebras.

In the framework of  $n$-Church and $n$-Boolean like  algebras, the constants 
$\e_i$ and the $n+1$-ary operation $q_n$ represent the generalised truth-values and the generalised conditional operation, respectively.
More generally, these constants and operation allow to express neatly other fundamental algebraic concepts as one-sorted,  purely algebraic theories.
These include in particular: (i) variables and term-for-variable substitution in free algebras on one side, and (ii) projections and functional composition in clones on the other.

Building up on this observation, we introduce  in this paper an algebraic theory of clones. 
%
%
%
Indeed, the variety of clone algebras  ($\mathsf{CA}$) introduced here  constitutes a purely one-sorted algebraic theory of clones in the same spirit as  Boolean algebras constitute an algebraic theory of classical propositional logic. 
Clone algebras of a given similarity type $\tau$ ($\mathsf{CA}_\tau$s) are defined by universally quantified equations and thus form a variety in the universal algebraic sense. 
The operators of type $\tau$ are taken as fundamental operations in $\mathsf{CA}_\tau$s.
 A crucial feature of our approach is connected with the role played by variables in algebras (resp. by projections in clones)  as placeholders. In clone algebras this is abstracted out, and takes the form of a system of fundamental elements (nullary operations) $\e_1,\e_2,\dots,\e_n,\dots$ of the algebra. This  important feature is borrowed from algebraic logic, namely cylindric and polyadic algebras and from lambda abstraction algebras (see \cite{HMT,SG99}).
One important consequence of the abstraction of variables is the abstraction of term-for-variable substitution (or functional composition) in $\mathsf{CA}_\tau$s, obtained by introducing an $n+1$-ary operator $q_n$ for every $n\geq 0$.
Roughly speaking, $q_n(a,b_1,\dots,b_n)$  represents
the substitution of $b_i$ for $\e_i$ into $a$ for $1\leq i\leq n$ (or the composition of  $a$ with  $b_1,\dots,b_n$). 
Every clone algebra is an $n$-Church algebra, for every $n$.

 The most natural $\mathsf{CA}$s, the ones the axioms are intended to characterise, are algebras of functions, called \emph{functional clone algebras}. 
 The elements of a functional clone algebra are infinitary operations from $A^\omega$ into $A$, for a given set $A$. In this framework  $q_n(f,g_1,\dots,g_n)$ represents the $n$-ary composition of $f$ with $g_1,\dots,g_n$, acting on the first $n$ coordinates:
 $$q_n(f,g_1,\dots,g_n)(s)= f(g_1(s),\dots,g_n(s),s_{n+1},s_{n+2},\dots),\ \text{for every $s\in A^\omega$}$$
 and the nullary operators are the projections $p_i$ defined by $p_i(s)=s_i$ for every $s\in A^\omega$.
Hence, the universe of a functional clone algebra is a set of infinitary operations containing the projection $p_i$ and closed under finitary compositions, called hereafter \emph{$\omega$-clone}. We show that there exists a bijective correspondence between clones (of finitary operations) and a suitable subclass of functional clone algebras, called \emph{block algebras}.  Given a clone $C$, the corresponding block algebra is obtained by extending the operations of the clone by countably many dummy arguments.
If $f\in C$ has arity $k$, then the top expansion  of $f$ is an infinitary operation $f^\top:A^\omega\to A$:
$$f^\top(s_1,s_2,\dots,s_k,s_{k+1},\dots)=f(s_1,\dots,s_k),\quad \text{for every $(s_1,s_2,\dots,s_k,s_{k+1}\dots)\in S^\omega$}.$$ 
By collecting all these top expansions in a set $C^\top=\{f^\top: f\in C\}$, we get a functional clone algebra, called block algebra.
In the first representation theorem of the paper we show that the ``concrete'' notion of block algebra coincides, up to isomorphism, with  the abstract notion of finite-dimensional clone algebra, where a clone algebra is finite-dimensional if each of its elements can be assigned a finite dimension, generalising the notion of arity to infinitary functions.


The axiomatisation of functional clone algebras is a central issue in the algebraic approach to clones.
 We say that a clone algebra is functionally representable if it is isomorphic to a functional clone algebra. One of the main results of this paper is the general representation theorem, where it is shown that every $\mathsf{CA}$ is functionally representable. 
 Therefore, the  clone algebras are the full algebraic counterpart of $\omega$-clones, while the block algebras are the algebraic counterpart of clones.
 In another result of the paper we prove that  the variety of clone algebras is generated by the class of block algebras. This implies that every $\omega$-clone is algebraically generated by a suitable family of clones by using direct products, subalgebras and homomorphic images.

We conclude the paper with two applications. The first one is to the lattice of equational theories problem stated by Birkhoff \cite{bir2} in 1946: Find an algebraic characterisation of those lattices which can be isomorphic to  a lattice of equational theories. Maltsev \cite{mal} was instrumental in attracting attention to this problem, which is sometimes referred to as Maltsev's Problem. This problem is still open, but work on it has led to many results described in \cite[Section 4]{nulty}.

The problem of characterising the lattices of equational theories as the congruence lattices of a class of algebras was tackled by Newrly \cite{newrly} and Nurakunov \cite{nur}. 
In this paper we propose an alternative  answer to the lattice of equational theories problem. We prove that a lattice is isomorphic to a lattice of equational theories if and only if it is isomorphic to the lattice of all congruences of a finite dimensional clone algebra. Unlike in Newrly's and Nurakunov's approaches, we are able to provide  the equational axiomatisation of the variety whose congruence lattices are exactly the lattices of equational theories, up to isomorphisms.
We also show that a lattice is isomorphic to a lattice of subclones if and only if it is isomorphic to the lattice of subalgebras of a finite dimensional clone algebra.

The second application is to the study of the category $\mathcal{VAR}$ of all varieties. 
We say that a clone algebra is \emph{pure} if it is an algebra in the type of  the nullary operators $\e_1,\e_2,\dots$ and of the operators $q_n$ ($n\geq 0$). The \emph{pure reduct} of a clone algebra of type $\tau$ is a pure clone algebra. It is worth mentioning that  
important properties of a variety depend on the pure reduct of the clone algebra associated with its free algebra.  
After characterising central elements in clone algebras, we introduce the concept of a minimal clone algebra. We show that  a clone algebra $\mathbf C$ of type $\tau$ is minimal if and only if the $\tau$-reduct $\mathbf C_\tau$ of $\mathbf C$ is the free algebra over a countable set of generators in the variety generated by $\mathbf C_\tau$.
We introduce the category $\mathcal{CA}$ of all clone algebras (of arbitrary similarity type) with pure homomorphisms (i.e., preserving only the nullary operators $\e_i$ and the operators $q_n$) as arrows and show that $\mathcal{CA}$ is equivalent both to the full subcategory $\mathcal{MCA}$ of minimal clone algebras and, more to the point, to the variety $\mathsf{CA}_0$ of pure clone algebras.
Moreover, we show that $\mathcal{MCA}$ is isomorphic to $\mathcal{VAR}$ as a category. This result allows us to directly use $\mathcal{MCA}$ to study  the category $\mathcal{VAR}$. We conclude the paper by showing  that the category $\mathcal{MCA}$ is closed under categorical product and utilise this result and central elements to provide a generalisation of the theorem on independent varieties presented by Gr\"atzer et al. in \cite{GLP}.

\subsection{Plan of the work}
In Section \ref{sec:prelim} we present some preliminary notions, 
including those of factor congruence and decomposition operator, and  those of 
Church and Boolean-like algebra, less well known; we also expose the Birkhoff and Maltsev's 
problem and sketch some related work. In Section \ref{sec:clo} we introduce the notion of a clone with nullary operations;
we also recall abstract clones.
Section \ref{sec:clonAlg} introduces the clone algebras that we propose as an algebraic one-sorted
counterpart of clones. In Sections \ref{sec:dueuno} and \ref{sec:clones} we present two prototypical classes 
of clone algebras: functional clone algebras, whose carriers are named $\omega$-clones, and block algebras. Those are algebras of infinitary operations. The former are uncostrained, and in particular they may be sensible to countably many arguments, whereas the latter are finite dimensional, since they are obtained by suitable extensions, called top extensions, of finitary operations.
We show that there is a bijection between clones and block algebras.
In Section \ref{sec:rep} we introduce the representable (finitary) operations
inside a clone algebra, which are those operations whose behaviour is univocally determined by an element of the algebra, via the operators $q_n$.
The representable operations of $\mathbf C$ turn out to be a clone and the top extension of this clone is a block algebra, 
isomorphic to a finite dimensional subalgebra of $\mathbf C$. 
This subalgebra coincides with $\mathbf C$ whenever $\mathbf C$ is finite dimensional. 
Since all the basic operation of a clone algebra are representable, there is no loss of information in replacing each of them 
with the corresponding element: we show in Section \ref{sec:nullary} that the variety 
of clone $\tau$-algebras and that of clone algebras with $\tau$-constants are term equivalent.
In Section \ref{sec:GRT} we prove the main representation theorem, 
indicating the pertinence of our approach to the theory of clones.
It can be summarized as follows:
the variety of clone algebras is the algebraic counterpart of $\omega$-clones, 
the class of block algebras is the algebraic counterpart of clones, and the $\omega$-clones
are algebraically generated by clones through direct products, subalgebras and 
homomorphic images. In other words, the variety of clone algebras is generated by the class of block algebras.
Section \ref{sec:eqth} presents an application of clone algebras to the Birkhoff and Maltsev's 
problem: we prove that a lattice is isomorphic to a lattice of equational theories if and only if it is isomorphic 
to the lattice of all congruences of a finite dimensional clone algebra.
The last section of the paper is devoted to some applications of the theory of clone algebras to the study of the category of all varieties. In the conclusions we present some directions for future work.

\section{Preliminaries}\label{sec:prelim}

The notation and terminology in this paper are pretty standard. For
concepts, notations and results not covered hereafter, the reader is
referred to \cite{BS,Co65,mac} for universal algebra and to \cite{L06,SZ86,T93} for the theory of clones. 


In this paper $\omega=\{1,2,\dots\}$ denotes the set of positive natural numbers.

By an \emph{operation} on a set $A$ we will always mean a finitary operation (i.e., a function $f:A^n\to A$ for some $n\geq 0$), and by an \emph{infinitary operation} on $A$ we mean a function from $A^\omega$ into $A$.
As a matter of notation, operations will be denoted by the letters $f,g,h,\dots$ and infinitary operations by the greek letters $\varphi,\psi,\chi,\dots$.

We denote by $\mathcal{O}_A$  the set of all operations on a set $A$, and by $\mathcal{O}_A^{(\omega)}$ the set of all infinitary operations on $A$. If 
$F\subseteq\mathcal{O}_A $, then $F^{(n)}= \{f:A^n\to A\ |\ f\in  F\}$. 

In the following we fix a countable infinite set $I=\{v_1, v_2,\dots,  v_n,\dots\}$ of \emph{indeterminates or variables} that we assume totally ordered: $v_1 < v_2 <\dots < v_n<\dots$.


\subsection{Algebras}\label{sec:alg}
If $\tau$ is an algebraic type, an algebra $\mathbf{A}$ of type $\tau $ is
called \emph{a }$\tau $\emph{-algebra}, or simply an algebra when $\tau $ is
clear from the context. An algebra is \emph{trivial} if its carrier set is a singleton set.

Superscripts that mark the difference between operations and operation symbols will be dropped whenever the context is sufficient for a disambiguation. 

If $t$ is a $\tau$-term, then we write $t=t(v_1,\dots,v_n)$ if $t$ can be built up starting from variables $v_1,\dots,v_n$.
Not all variables $v_1,\dots,v_n$ may occur in $t$. If $t=t(v_1,\dots,v_n)$, then $t=t(v_1,\dots,v_m)$ for every $m\geq n$.
A term is \emph{ground} if no variable occurs in it. 

We denote by $T_{\tau}(\omega)$ the set of  $\tau$-terms over the countable infinite set $I$ of variables.

$\mathrm{Con}\,\mathbf{A}$ is the lattice of all
congruences on an algebra $\mathbf{A}$, whose bottom and top elements are,
respectively, $\Delta =\{(a,a):a\in A\}$ and $\nabla =A\times A$. 
Given $a,b\in A$, we write $\theta(a,b)$ for the smallest congruence $\theta$
such that $( a,b) \in \theta $.


Closure under  homomorphic images, direct products, subalgebras and  isomorphic images is denoted by $\mathbb{H}$, $\mathbb{P}$, $\mathbb{S}$ and $\mathbb{I}$ respectively.
We denote by $\mathbb{U}_p$ the closure under ultraproducts.

A class $\mathcal{V}$ of $\tau$-algebras is a \emph{variety} if it is closed
under subalgebras, direct products and homomorphic images, i.e., $\mathcal{V}= \mathbb{HSP}(\mathcal{V})$.
The variety $\mathrm{Var}(K)$ generated by a class $K$ of $\tau$-algebras is the smallest variety including $K$: $\mathrm{Var}(K)= \mathbb{HSP}(K)$.
 If $K=\{\mathbf A\}$ we write $\mathrm{Var}(\mathbf A)$ for $\mathrm{Var}(\{\mathbf A\})$.

If $\mathcal V$ is a variety, then we denote by $\mathbf F_\mathcal V$ its free algebra over the countable infinite set
$I$ of generators.

Recall that $n$ subvarieties $\mathcal V_1,\dots,\mathcal V_n$ of a variety $\mathcal V$ of type $\tau$ are said to be
 \emph{independent}, if there exists a term $t(v_1,\dots,v_n)$ of type $\tau$, containing at most the
indicated variables, such that $\mathcal V_i\models t(v_1,\dots,v_n)=v_i$ ($i=1,\dots,n$).
Moreover,  the \emph{product of similar varieties} $\mathcal V_1,\dots,\mathcal V_n$ is
defined as $\mathcal V_1\times\dots\times \mathcal V_n =\mathbb{I}\,\{\mathbf A_1\times\dots\times\mathbf A_n  :\mathbf A_i \in \mathcal V_i\}$. We have $\mathcal V_1\times\dots\times \mathcal V_n\subseteq \mathcal V_1\lor\dots\lor \mathcal V_n$.



We recall from \cite[Page 245]{mac} that an \emph{interpretation} of a variety $\mathcal V$ of type $\tau$ into a variety $\mathcal W$ of type $\nu$  is a mapping $f$ with domain $\tau$ satisfying:
\begin{itemize}
\item If $\sigma\in\tau$ has arity $n>0$, then $f(\sigma)$ is an $n$-ary $\nu$-term;
\item If $\sigma\in\tau$ has arity $0$, then $f(\sigma)=t$ is a unary $\nu$-term such that the equation $t(v_1)=t(v_2)$ is valid in $\mathcal W$;
\item For every algebra $\mathbf A\in \mathcal W$, the algebra $\mathbf A^f=(A,f(\sigma)^{\mathbf A,k})_{\sigma\in\tau}$ belongs to $\mathcal V$, where $f(\sigma)^{\mathbf A,k}$ ($\sigma$ of arity $k$) is the $k$-ary term operation defined in Section \ref{sec:to}.
\end{itemize}

\subsection{Factor Congruences and Decomposition}\label{sec:fcd}
Directly indecomposable algebras play an important role in the characterisation of the structure of a variety of
algebras.
In this section we summarise the basic ingredients of factorisation:
tuples of complementary factor congruences
and  decomposition operators (see  \cite{mac}).

\begin{definition}\label{def:cong} 
A sequence $(\theta_{1},\dots,\theta_n)$ of congruences on a $\tau$-algebra $\mathbf{A}$ is an $n$-tuple of complementary factor congruences exactly when:
\begin{enumerate}
\item $\bigcap_{1\leq i\leq n}\theta_{i}=\Delta$;

\item $\forall (a_1,\dots,a_n)\in A^n$, there is a unique $u\in A$ such that $a_i\theta_{i}\,u$,
for all $1\leq i\leq n$.
\end{enumerate}
\end{definition}

If $(\theta_{1},\dots,\theta_n)$  is an $n$-tuple of complementary factor congruences on $\mathbf{A}
$, then the function
$f:\mathbf{A}\rightarrow \prod\limits_{i=1}^n\mathbf{A}/\theta _{i}$,
defined by $f(a) =(a/\theta _{1},\dots,a/\theta _{n})$, is an
isomorphism. Moreover, every factorisation of $\mathbf A$ in $n$ factors univocally determines an $n$-tuple of complementary factor congruences.

A pair  $(\theta_{1},\theta_{2})$ of congruences is a pair of complementary factor
congruences if and only if $\theta_{1}\cap\theta_{2}=\Delta$ and $\theta_{1}\circ\theta_{2}=\nabla$. 
A \emph{factor congruence} is any congruence which belongs to a pair of
complementary factor congruences. 
Notice that, if $(\theta_{1},\dots,\theta_n)$ is an $n$-tuple of complementary factor congruences, then $\theta_i$ is a factor congruence for each $1\leq i\leq n$, because
the pair $(\theta_i, \bigcap_{j\neq i} \theta_j)$ is a pair of complementary factor congruences. 

It is possible to characterise $n$-tuples of complementary factor congruences in terms of certain algebra
homomorphisms called \emph{decomposition operators} (see \cite[Def.~4.32]{mac} for additional details). 

\begin{definition}
\label{def:decomposition} 
An \emph{$n$-ary decomposition operator} on a $\tau$-algebra $\mathbf{A}$ is a function $f:A^{n}\rightarrow A$ satisfying the following conditions: 
\begin{description}
\item[D1] $f( x,x,\dots,x) =x$;
\item[D2] $f( f( x_{11},x_{12},\dots,x_{1n}),\dots,f(x_{n1},x_{n2},\dots,x_{nn}))=f(x_{11},\dots,x_{nn})$;
\item[D3] $f$ is a homomorphism from $\mathbf{A}^n$ onto $ \mathbf{A}$.
\end{description}
\end{definition}

There is a bijective correspondence between $n$-tuples of complementary factor
congruences and $n$-ary decomposition operators, and thus, between $n$-ary decomposition
operators and factorisations of an algebra in $n$ factors.

If $f:A^n\to A$ is a function, then we denote by $f_i:A^2\to A$ the binary function defined as follows:
$$f_i(x,y)= f(y,\dots,y,x,y,\dots,y)\quad\text{$x$ at position $i$}.$$

\begin{theorem}
\label{prop:pairfactor} Any $n$-ary decomposition operator $f:A^{n}\rightarrow A$ on an algebra $\mathbf{A}$ induces an $n$-tuple of
complementary factor congruences $\theta _{1},\dots,\theta _{n}$, where each $\theta _{i}\subseteq A\times A$ is defined by: 
\begin{equation*}
a\ \theta _{i}\ b\ \ \text{iff}\ \ f_i(b,a)=a.
\end{equation*}
Moreover, $f(x_1,\dots,x_n)$ is the unique element such that $x_i \theta_i f(x_1,\dots,x_n)$ for all $i$.
Conversely, any $n$-tuple $\theta _{1},\dots,\theta _{n}$ of complementary factor
congruences induces a decomposition operator $f$ on $\mathbf{A}$: 
$f(a_1,\dots,a_n)=u$ iff $a_{i}\,\theta _{i}\,u$ for all $i$. 
\end{theorem}

\subsection{Church algebras}\label{dobbiaco}
In this section we recall from \cite{BLPS18} the notion of an $n$-Church algebra. These algebras have $n$ nullary operations $\e_1,\dots, \e_n$ ($n \geq 2$) and an operation $q_n$ of arity $n + 1$ (a sort of ``generalised if-then-else'') satisfying the identities $q_n(\e_i, x_1,\dots, x_n) = x_i$. The operator $q_n$ induces, through the so-called $n$-central elements, a decomposition of the algebra into $n$ factors. 

\begin{definition}
Algebras of type $\tau$, equipped with at least $n$ nullary operations $\e_1,\dots,\e_n$ and a term operation $q_n$ of arity $n+1$ satisfying $q_n(\e_i,x_1,\dots,x_n)=x_i$, are called \emph{$n$-Church algebras} ($n\mathrm{CH}$, for short); 
 $n\mathrm{CH}$s admitting only the $(n+1)$-ary $q_n$ operator and the $n$ constants $\e_{1},\dots ,\e_{n}$ are called \emph{pure} $n\mathrm{CH}$s.
\end{definition}


$2$CHs were introduced as Church algebras in \cite{MS08} and studied in \cite{first}.  Examples of $2$CHs  are Boolean algebras (with $q_2(x,y,z) =(x\wedge y)\vee (\lnot x\wedge z)$) or rings with unit (with $q_2( x,y,z) =xy+z-xz$). 
%


In \cite{vaggione}, Vaggione introduced the notion
of \emph{central element} to study algebras whose complementary factor
congruences can be replaced by certain elements of their universes. 
Central elements coincide with central idempotents in rings with
unit and with members of the centre in ortholattices. 

\begin{theorem} \cite{BLPS18}
\label{thm:centrale} If $\mathbf{A}$ is an $n\mathrm{CH}$ of type $\tau $
and $c\in A$, then the following conditions are equivalent:

\begin{enumerate}

\item the sequence of congruences $\theta (c,\e_{1}),\dots
,\theta (c,\e_{n})$ is an $n$-tuple of complementary factor congruences of $\mathbf{A}$;

\item for all $a_{1},\dots ,a_{n}\in A$, $q_n(c,a_{1},\dots ,a_{n})$ is the
unique element such that $$a_{i}\ \theta (c,\e_{i})\ q(c,a_{1},\dots ,a_{n}),\ \text{for all $1\leq i\leq n$;}$$

\item The function $f_{c}$, defined by $f_{c}(a_{1},\dots,a_{n})=q_n(c,a_{1},\dots ,a_{n})$ for all $a_1,\dots,a_n\in A$, is an $n$-ary decomposition operator on $\mathbf{A}$ such that 
$f_{c}(\e_{1},\dots ,\e_{n})=c.$
\end{enumerate}
\end{theorem}

\begin{definition}
\label{def:ncentral} If $\mathbf{A}$ is an $n\mathrm{CH}$, then $c\in A$
is called \emph{$n$-central} if it satisfies one of the equivalent conditions of Theorem \ref{thm:centrale}.
An $n$-central element $c$ is \emph{nontrivial} if $c\notin\{\e_{1},\dots ,\e_{n}\}$.
\end{definition}

Every $n$-central element $c\in A$ induces a decomposition of $\mathbf{A}$ as a direct product of the 
algebras $\mathbf{A}/\theta (c,\e_{i})$, for $i\leq n$.

The set of all $n$-central elements of an $n\mathrm{CH}$ $\mathbf{A}$ is a subalgebra of the pure reduct of $\mathbf{A}$.
We denote by $\mathbf{Ce}_{n}(\mathbf{A})$ the algebra 
$(\mathrm{Ce}_{n}(\mathbf{A}),q_n,\e_{1},\dots ,\e_{n})$ 
of all $n$-central elements of a 
$n\mathrm{CH}$ $\mathbf{A}$.

%

\subsection{Boolean-like algebras}\label{sec:nba}

Boolean algebras are $2$-CHs all of whose elements are 
$2$-central. It turns out that, among the $n$-CHs, those
algebras all of whose elements are $n$-central inherit many of the
remarkable properties that distinguish Boolean algebras. 

\begin{definition}
\label{mezzucci} \cite{BLPS18,Bucciasali} An $n\mathrm{CH}$ $\mathbf{A}$ of type $\tau$ is called an \emph{$n$-Boolean-like algebra} ($n\mathrm{BA}$, for short) if every element of $A$ is $n$-central.  An  $n\mathrm{BA}$ of empty type is called a pure $n\mathrm{BA}$.
\end{definition}

We denote by $n\mathsf{BA}_\tau$  the class of all $n\mathrm{BA}$s of type $\tau$.
If $\tau$ is empty, then $n\mathsf{BA}$ denotes the class of all pure $n\mathrm{BA}$s.

In an $n\mathrm{BA}$ $q_n(x,-,\dots,-)$ is an $n$-ary decomposition operator for every element $x$ of the universe of the algebra. Then, by Definition \ref{def:decomposition}
the class $n\mathsf{BA}_\tau$ is the variety of all $n\mathrm{CH}$s of type $\tau$ that satisfy the identities defining $q_n(x,-,\dots,-)$ as an $n$-ary decomposition operator.
%
%
%
%

$2\mathrm{BA}$s were introduced in \cite{first} with the
name \textquotedblleft Boolean-like algebras\textquotedblright . \emph{Inter
alia}, it was shown in that paper that the variety of $2\mathrm{BA}$s is term-equivalent to the variety of Boolean algebras.

\begin{example}\label{exa:canonical}
The algebra $\mathbf{Ce}_{n}(\mathbf{A})$ of all $n$-central
elements of an $n\mathrm{CH}$ $\mathbf{A}$ of type $\tau$ is a canonical example of pure $n\mathrm{BA}$.
\end{example}

\begin{example}\label{exa:n}
The algebra 
$\mathbf{n}=( \{ \mathsf \e_{1},\dots,\mathsf \e_{n}\} ,q_n^{\mathbf{n}},\mathsf e^\mathbf{n}_{1},\dots,\mathsf e^\mathbf{n}_{n})$,
where $q_n^{\mathbf{n}}( \mathsf \e_{i}^{\mathbf{n}},x_{1},\dots,x_{n}) =x_{i}$ for every $i\leq n$, is a pure $n\mathrm{BA}$.
\end{example}


\begin{example}\label{exa:parapa} (\emph{$n$-Partitions}) Let $X$ be a set. An \emph{$n$-partition} of $X$ is a sequence $(Y^{1},\ldots ,Y^{n})$ of subsets of $X$ such that $\bigcup_{i=1}^{n}Y^{i}=X$ and $Y^{i}\cap Y^{j}=\emptyset $ for all $i\neq j$.
The set of $n$-partitions of $X$ 
becomes a pure $n\mathrm{BA}$ if we define an operator $q_n$ and $n$ constants $\e_1,\dots,\e_n$ as follows, for all $n$-partitions $\mathbf{y}^{i}= (Y^i_{1},\dots,Y^i_{n})$:
$$
q_n( \mathbf{y}^0,\mathbf{y}^1,\dots ,\mathbf{y}^n) =(\bigcup\limits_{i=1}^{n}Y^0_{i}\cap Y_{1}^{i},\dots,\bigcup\limits_{i=1}^{n}Y^0_{i}\cap Y_{n}^{i});\quad 
\e_1=(X,\emptyset,\dots,\emptyset),\dots, \e_n=(\emptyset,\dots,\emptyset,X).
$$
Notice that the algebra  of $n$-partitions of $X$ can be proved isomorphic to the $n\mathrm{BA}$ $\mathbf{n}^X$  (the Cartesian product of $|X|$ copies of the algebra $\mathbf{n}$).
\end{example}


\begin{remark} 
It is  known from \cite{MS08} that the set of all $2$-central
elements of a $2\mathrm{CH}$ $\mathbf{A}$ is a Boolean algebra with respect to the following operations:
$$x\land y=q_2(x,\e_1,y);\quad x\lor y=q_2(x,y,\e_2);\quad \neg x=q_2(x,\e_2,\e_1).$$
The correspondence $a\in\mathrm{Ce}_2(\mathbf{A}) \mapsto \theta(a,\e_1)$
determines an isomorphism between the Boolean algebra of $2$-central elements and the Boolean algebra of factor congruences of $\mathbf A$. Notice that the factor congruence $\theta(a,\e_2)$ is the complement of the factor congruence $\theta(a,\e_1)$.
\end{remark}


The variety $\mathsf{BA}$ of Boolean algebras is semisimple as every $\mathbf{A}\in \mathsf{BA}$ is
subdirectly embeddable into a power of the $2$-element Boolean algebra, which
is the only subdirectly irreducible (in fact, simple) member of $\mathsf{BA}$. This property finds an analogue in the structure theory of $n\mathsf{BA}$s.

\begin{theorem}
\label{lem:subirr} \cite{BLPS18,Bucciasali} 
\begin{enumerate}
\item[(i)] The algebra $\mathbf n$ is the unique simple pure $n\mathrm{BA}$ and it generates the variety $n\mathsf{BA}$.
\item[(ii)] the variety $n\mathsf{BA}_\tau$ of $n\mathrm{BA}$s of type $\tau$ is generated by its finite members of cardinality $n$.
\end{enumerate}

\end{theorem}

The next corollary shows that, for any $n\geq 2$, the $n\mathrm{BA}$ $\mathbf{n}$ plays a role analogous to the Boolean algebra $\mathbf{2}$ of truth values.

\begin{corollary}
\label{cor:stn} Every pure $n\mathrm{BA}$ is isomorphic to a subdirect power of $\mathbf{n}^{X}$, for some set $X$.
\end{corollary}

By Example \ref{exa:parapa} and Corollary \ref{cor:stn} every pure $n\mathrm{BA}$ is isomorphic to an  $n\mathrm{BA}$ of $n$-partitions of some set $X$.

One of the most remarkable properties of the $2$-element Boolean algebra, called \emph{primality} in universal algebra \cite[ Sec. 7 in Chap. IV]{BS}, is the definability of all finite Boolean functions in terms of the connectives {\sc and, or, not}. This property is inherited by $n$BAs.

\begin{theorem}\label{prop:nbaprim} \cite{BLPS18}
The variety $n\mathsf{BA}=\mathrm{Var}(\mathbf n)$ is primal.
\end{theorem}

\subsection{Lattices of equational theories}\label{sec:pre:leq}
We say that $L$ is a \emph{lattice of equational theories} iff $L$ is isomorphic to the lattice $L(T )$ of all equational theories containing some equational theory $T$ (or dually, to the lattice of all subvarieties of some variety of algebras). 
Thus, if $T$ were the equational theory of all groups, then $L(T)$ would be the lattice of all equational theories of groups and one of the members of $L(T)$ would be the equational theory of Abelian groups.

The lattice $L(T )$ is ordered by set-inclusion, the meet in this lattice is just intersection and the join of a collection $E$ of equational theories is just the equational theory based on $\bigcup E$.
A lattice of equational theories is algebraic and coatomic, possessing a compact top element; but no stronger property was known before Lampe's discovery that any lattice of equational theories obeys a weakening of semidistributivity called the Zipper condition, which is a nontrivial implication in the language of bounded lattices  (see Lampe \cite{lampe}).
Lampe's Theorem suggests that the class of lattices of the form $L(T)$ might have interesting structural properties. 

In 1946 Birkhoff \cite{bir2} stated the lattice of equational theories problem: Find an algebraic characterisation of those lattices which can be isomorphic to $L(T)$ for some equational theory $T$. Maltsev \cite{mal} was instrumental in attracting attention to this problem, which is sometimes referred to as Maltsev's Problem, and this led to many interesting results summarised in \cite[Section 4]{nulty}.
 
Trying to characterise the lattices of equational theories as the congruence lattices of a class of algebras is a natural, though difficult, way of approaching the problem.
In \cite{newrly} Newrly shows that a lattice of equational theories is the congruence lattice of an algebra whose fundamental operations consist of one monoid operation with right zero and one unary operation. In \cite{nur} Nurakunov describes a class of monoids enriched by two unary operations, the so-called Et-monoids, and proves that a lattice $L$ is a lattice of equational theories if and only if $L$ is isomorphic to the congruence lattice of some Et-monoid. Nevertheless, the varieties of algebras generated by Newrly's monoids and by  Nurakunov's Et-monoids have not been thoroughly investigated, and in particular they do not admit a known equational axiomatisation. Hence the problem of characterising the lattices of equational theories is still open.

\section{Clones of operations}\label{sec:clo}
Given an algebra $\mathbf A$ of type $\tau$, one is often interested in the term operations of the algebra rather than in the basic operations $\sigma^\mathbf A$ itself ($\sigma\in\tau$). In particular, if two algebras have the same set of term operations, then one might consider their difference as a mere question of representation. This motivates a notion that describes precisely those sets of operations that can arise as sets of term operations of an algebra and that is exactly what a clone is.

\bigskip

A \emph{$k$-ary projection} is a function $p^{(k)}_i: A^k\to A$ ($k\geq i$) defined by $p^{(k)}_i(a_1,\dots,a_k)=a_i$. 
A \emph{basic projection} is a projection $p^{(i)}_i$ ($i\geq 1$). We denote a basic projection 
by $p_i=p^{(i)}_i$. We denote by $\mathcal J_A$ the set of all projections.

A \emph{$k$-ary constant operation} is a function $c^{(k)}_a: A^k\to A$ ($k\geq 0$ and $a\in A$) such that  
$c^{(k)}_a(a_1,\dots,a_k)=a$, for all $a_1,\dots,a_k\in A$.

\bigskip

One may consider various natural operations on  $\mathcal{O}_A$, the set of all operations on $A$, and among them
the composition operation is of paramount importance. 
In the following definition we formally define the composition.

\begin{definition}\label{def:comp}
 The \emph{composition}  of $f\in  \mathcal{O}_A^{(n)}$ with  $g_1,\dots,g_n\in   \mathcal{O}_A^{(k)}$ is the operation $f(g_1,\dots,g_n)_k\in  \mathcal{O}_A^{(k)}$ defined as follows: 
$$f(g_1,\dots,g_n)_k(\mathbf a)= f(g_1(\mathbf a),\dots, g_n(\mathbf a))\quad\text{for all $\mathbf a\in A^k$}.$$
\end{definition}
In particular, if $f\in  \mathcal{O}_A^{(0)}$ then 
$f()_k(\mathbf a)= f$ for all $\mathbf a\in A^k$.

When there is no danger of confusion, we write $f(g_1,\dots,g_n)$ for $f(g_1,\dots,g_n)_k$.

\begin{definition}\label{def:fic} Let $A$ be a set and $n >0$. An $n$-ary operation $f:A^n\to A$
\begin{itemize}
\item[(i)]   \emph{depends on its $i$-th argument} ($1\leq i\leq n$) if there are $a_1,\dots,a_n,b,c\in A$ such that 
$$f(a_1,\dots,a_{i-1},b,a_{i+1},\dots,a_n)\neq f(a_1,\dots,a_{i-1},c,a_{i+1},\dots,a_n).$$
 \item[(ii)] is \emph{fictitious in the $i$-th argument} if it does not depend on its $i$-th argument.
 \item[(iii)] is \emph{fictitious} if $f$ is fictitious in its $n$-th (i.e., last) argument.
\end{itemize}
\end{definition}

\begin{definition} Let $f$ be a fictitious $n$-ary operation on $A$ and $g$ be an $(n-1)$-ary operation on $A$.
We say that $g$ is the \emph{restriction} of $f$ and $f$ is the \emph{fictitious expansion} of $g$
 if  
$$g(a_1,\dots,a_{n-1})=f(a_1,\dots,a_{n-1},b),\quad\text{for all $a_1,\dots,a_{n-1},b\in A$}.$$
\end{definition}

We say that a set $X$ of operations is \emph{closed under restriction} if 
$X$ contains the restriction of every fictitious operation of $X$.

\begin{definition}\label{def:clo}  A \emph{clone on a set $A$} is a subset $F$ of $\mathcal{O}_A$ containing all projections
$p^{(n)}_i$ and closed under composition and restriction. 
\end{definition}

A \emph{clone on a $\tau$-algebra $\mathbf A$} is a clone on $A$  containing 
 the basic operations $\sigma^\mathbf A$ ($\sigma\in\tau$) of $\mathbf A$.

\begin{remark} The classical approach to clones, as evidenced by the standard monograph \cite{SZ86}, considers clones only containing operations that are at least unary. However, with only minor modifications, most of the usual theory can be lifted to clones allowing nullary operations (see \cite{B14}). 
Typically, clones as abstract clones (see below) and  Lawvere's algebraic theories \cite{Law63} include nullary operations.
The full generality of some results in this paper requires clones allowing nullary operators.
\end{remark}

\begin{remark} Clones without nullary operations do not require the closure under restriction, because using projections  and composition it is possible to define the at least unary restriction of every fictitious operation.
Clones allowing nullary operations do require the closure under restriction. Using projections and composition it is not possible to define the nullary operator $c^{(0)}_a$ that is restriction of the constant unary operation $c^{(1)}_a$.
\end{remark}
 
 Clones on a set $A$ are closed under arbitrary intersection, so that they constitute a complete lattice denoted by $\mathrm{Lat}(\mathcal{O}_A)$. The clone generated by a set $F$ of operations will be denoted by $[F]$.
 If $F=\{f\}$ is a singleton, then we will write $[f]$ for $[\{f\}]$.
 
The set $\mathcal O_A$ of all operations and the set $\mathcal J_A$ of all projections are clones on $A$. 
$\mathcal O_A$ and $\mathcal J_A$ are respectively the top element and the bottom element of the lattice $\mathrm{Lat}(\mathcal{O}_A)$.

\subsection{Clone of the term operations}\label{sec:to} The \emph{clone of the term operations of a $\tau$-algebra $\mathbf A$}, denoted by  $\mathrm{Clo} \mathbf A$, is the smallest clone on $\mathbf A$. It is constituted by the set of all term operations of $\mathbf A$. The  definition of term operation must be carefully given.  
Every term $t$ determines an infinite set $T_t^\mathbf A$ of \emph{term operations} $t^{\mathbf A,k}:A^k\to A$, where $k$ is $\geq r$ for a suitable $r$ depending on $t$. We define $T_t^\mathbf A$ by induction as follows.
\begin{itemize}
\item $T_{v_i}^\mathbf A=\{ v_i^{\mathbf A,k}: k\geq i\}$, where $v_i^{\mathbf A,k}(a_1,\dots,a_k)=a_i$ for every $a_1,\dots,a_k\in A$ and variable $v_i$.
\item If $t=\sigma(t_1,\dots,t_m)$ and $t_1^{\mathbf A,k}  \in T_{t_1}^\mathbf A,\dots,t_m^{\mathbf A,k}\in T_{t_m}^\mathbf A$, then $t^{\mathbf A,k}\in T_{t}^\mathbf A$ is defined as 
$t^{\mathbf A,k}(a_1,\dots,a_k)=\sigma^\mathbf A(t_1^{\mathbf A,k}(a_1,\dots,a_k),\dots,t_m^{\mathbf A,k}(a_1,\dots,a_k))$ for every $a_1,\dots,a_k\in A$.
\item If $t^{\mathbf A,k+1}\in T_{t}^\mathbf A$ is fictitious, then $t^{\mathbf A,k}\in T_{t}^\mathbf A$ is  the restriction of $t^{\mathbf A,k+1}$.
\end{itemize}

\begin{proposition}
  $\mathrm{Clo} \mathbf A = \bigcup_{t\in T_\tau(\omega)} T_{t}^\mathbf A$.
\end{proposition}

\subsection{Abstract clones}\label{sec:attempt}
We describe an attempt (among others) aiming to encode clones into algebras  (see \cite{T93} and \cite[p.\,239]{evans}).
An \emph{abstract clone} is a many-sorted algebra composed of  disjoint sets $B_n$ ($n\geq 0$),
elements $\pi_i^{(n)}\in B_n$ ($n\geq 1$) for all $i\leq n$, and
a family of operations $C^n_k: B_n\times (B_k)^n \to B_k$ for all $k$ and $n$
such that
\begin{enumerate}
\item $C^n_k(C^m_n(x,y_1,\dots,y_m),\mathbf z)= C^m_n(x,C^n_k(y_1,\mathbf z),\dots,C^n_k(y_m,\mathbf z))$, where $x$ is a variable of sort $m$, $y_1,\dots,y_m$ are variables of sort $n$ and $\mathbf z=z_1,\dots,z_n$ are variables of sort $k$;
\item $C^n_n(x,\pi_1^{(n)},\dots,\pi_n^{(n)})=x$, where $x$ is a variable of sort $n$;
\item $C^n_k(\pi_i^{(n)},y_1,\dots,y_n)=y_i$, where $y_1,\dots,y_n$ are variables of sort $k$.
\end{enumerate}
Any clone $F$ on a set $A$ determines an abstract clone 
$\mathbf F=( F^{(n)},C^n_k,\pi_i^{(k)})_{n\geq 0,k\geq 1}$, where the nullary operators $\pi_i^{(k)}=p_i^{(k)}\in F^{(k)}$ are the projections  and
$C^n_k$ is the operator of composition introduced in Definition \ref{def:comp}:
$C^n_k(f,g_1,\dots,g_n)=f(g_1,\dots,g_n)_k$ is the composition of $f\in \mathcal F^{(n)}$ with $g_1,\dots,g_n\in \mathcal F^{(k)}$.


\subsection{Neumann's abstract $\aleph_0$-clones \cite{neu70,T93}}\label{sec:neu} The idea here is to regard an $n$-ary operation $f$ as an infinitary operation
that only depends on the first $n$ arguments (see Section \ref{sec:blocktop}). The corresponding abstract definition is as follows.
An \emph{abstract $\aleph_0$-clone} is an infinitary algebra $(A,\e_i,q_\infty)_{1\leq i<\omega}$, where the $\e_i$ are nullary operators and $q_\infty$ is an $\omega$-ary operation on $A$, satisfying the following axioms: 
\begin{itemize}
\item[(i)] $q_\infty(\e_i,x_1,\dots,x_n,\dots)=x_i$;
\item[(ii)] $q_\infty(x,\e_1,\dots,\e_n,\dots)=x$;
\item[(iii)] $q_\infty(q_\infty(x,\mathbf  y),\mathbf z)= q_\infty(x,q_\infty(y_1,\mathbf z),\dots,q_\infty(y_n,\mathbf z),\dots)$,
where  $\mathbf y=y_1,\dots,y_n,\dots$ and\\ $\mathbf  z=z_1,\dots,z_n,\dots$ are countable infinite sequences of variables.  
\end{itemize}

The most natural abstract $\aleph_0$-clones, the ones the axioms are intended to characterise, are algebras of infinitary operations, called \emph{functional $\aleph_0$-clones}, containing the projections and closed under infinitary composition. 
More precisely, a functional $\aleph_0$-clone is an infinitary algebra $(F,\e_i^\omega,q_\infty^\omega)_{1\leq i<\omega}$ defined as follows, for all $\varphi,\psi_i\in F$ and $s\in A^\omega$: 
\begin{itemize}
\item[(a)]  $F\subseteq \mathcal O_A^{(\omega)}$;
\item[(b)] $\e_i^\omega(s)=s_i$;
\item[(c)] $q_\infty^\omega(\varphi,\psi_1,\dots,\psi_n,\dots)(s)=\varphi(\psi_1(s),\psi_2(s),\dots,\psi_n(s),\dots).$
\end{itemize}
There is a faithful functor from the category of clones to the category of abstract $\aleph_0$-clone, but this functor is not onto. The problem is that these infinitary algebras may contain elements that correspond to operations essentially of infinite rank. Technical difficulties have caused this approach to be largely abandoned.

\section{Clone algebras}\label{sec:clonAlg}
 We have described in Section \ref{sec:attempt} an attempt that has been made to encode clones into algebras using  many-sorted algebras, and in Section \ref{sec:neu}  an attempt based on infinitary algebras.  
In this section we introduce the variety of \emph{clone algebras} as a more canonical algebraic account of clones using standard one-sorted algebras. 
In Sections \ref{sec:dueuno} and \ref{sec:clones} we will show how to encode clones inside clone algebras. 
The algebraic type of clone algebras contains a countable infinite family of nullary operators $\e_i$ and, for each $n\geq 0$, an operator $q_n$ of arity $n+1$. Informally, the constant $\e_i$ represents the $i$-th projection and the operator $q_n$ represents the $n$-ary functional composition. In the relevant example of free algebras the constants $\e_i$  represent the variables and the operators $q_n$  the term-for-variable substitutions.
Each operator $q_n$ embodies the countable infinite family of operators $C^n_k$ ($k\geq 1$) of abstract clones described in Section \ref{sec:clo}. 


In the remaining part of this paper 
when we write $q_n(x,\mathbf y)$ it will be implicitly stated that  $\mathbf y=y_1,\dots,y_n$ is a sequence of length $n$.

The algebraic type of clone $\tau$-algebras is $\tau \cup\{ q_n: n\geq 0\} \cup\{ \e_i : i\geq 1\}$. 

%

\begin{definition} \label{def:clonealg}
 A \emph{clone $\tau$-algebra}   is an algebra 
 $\mathbf C = (\mathbf C_\tau, q^\mathbf C_n,\e^\mathbf C_i)_{n\geq 0,i\geq 1}$ satisfying the following conditions:
\begin{enumerate}
\item[(C0)] $\mathbf C_\tau = (C,\sigma^\mathbf C)_{\sigma\in\tau}$ is a $\tau$-algebra;
\item[(C1)] $q_n(\e_i,x_1,\dots,x_n)=x_i$ $(1\leq i\leq n)$;
\item[(C2)] $q_n(\e_j,x_1,\dots,x_n)=\e_j$ $(j>n)$;
\item[(C3)] $q_n(x,\e_1,\dots,\e_n)=x$ $(n\geq 0)$;
\item[(C4)] $q_k(x, y_1,\dots,y_k)= q_n(x,y_1,\dots,y_k,\e_{k+1},\dots,\e_n)$ ($n> k$);
 \item[(C5)] $q_n(q_n(x,y_1,\dots,y_n),z_1,\dots,z_n)=q_n(x,q_n(y_1,z_1,\dots,z_n),\dots,q_n(y_n,z_1,\dots,z_n))$;
\item[(C6)]  $q_n(\sigma(x_1,\dots,x_k),y_1,\dots,y_n) = \sigma(q_n(x_1,y_1,\dots,y_n),\dots,q_n(x_k,y_1,\dots,y_n))$ for every $\sigma\in\tau$ of arity $k$ and every $n\geq 0$.
\end{enumerate}
If $\tau$ is empty, an algebra satisfying (C1)-(C5) is called a \emph{pure clone algebra}.
\end{definition}

In the following, when there is no danger of confusion,
we will write $\mathbf C = (\mathbf C_\tau, q^\mathbf C_n,\e^\mathbf C_i)$  for $\mathbf C = (\mathbf C_\tau, q^\mathbf C_n,\e^\mathbf C_i)_{n\geq 0,i\geq 1}$.

\begin{definition}
  If $\mathbf C$ is a clone $\tau$-algebra, then $\mathbf C_0=( C, q^\mathbf C_n,\e^\mathbf C_i)$ is called the \emph{pure reduct of $\mathbf C$}.
\end{definition}

The class of clone $\tau$-algebras  is denoted by $\mathsf{CA}_\tau$ and the class of all clone  algebras of any type by $\mathsf{CA}$. $\mathsf{CA}_0$ denotes the class of all pure clone algebras.
We also use $\mathsf{CA}_\tau$ as shorthand for the phrase ``clone $\tau$-algebra'', and similarly for $\mathsf{CA}$. 

By (C1) every $\mathsf{CA}_\tau$ is an $n\mathrm{CH}$, for every $n$ (see Section \ref{dobbiaco}).

We start the study of clone algebras with two simple lemmas.

%

\begin{lemma}\label{allungo} Let $\mathbf y=y_1,\dots,y_n$ and $\mathbf z=z_1,\dots,z_k$. Then the following identities follow from (C1)-(C5):
\begin{enumerate}
\item[(i)] If $n< k$, then
$q_k(q_n(x,\mathbf y),\mathbf z) = q_k(x,q_k(y_1,\mathbf z),\dots,q_k(y_n,\mathbf z),z_{n+1},\dots,z_k)$.
\item[(ii)] If $n\geq k$,  then
$q_k(q_n(x,\mathbf y),\mathbf z) = q_n(x,q_k(y_1,\mathbf z),\dots,q_k(y_n,\mathbf z))$.
\end{enumerate}
\end{lemma}

\begin{proof} 
 \[
\begin{array}{lll}
q_k(q_n(x,\mathbf y),\mathbf z)  & =_{(C4)}  &  q_k(q_k(x,\mathbf y, \e_{n+1},\dots,\e_k),\mathbf z) \\
  & =_{(C5)}  & q_k(x,q_k(y_1,\mathbf z),\dots,q_k(y_n,\mathbf z),q_k(\e_{n+1},\mathbf z),\dots,q_k(\e_k,\mathbf z))  \\
  & =_{(C1)}  &  q_k(x,q_k(y_1,\mathbf z),\dots,q_k(y_n,\mathbf z), z_{n+1},\dots,z_k)\\  
\end{array}
\]
\[
\begin{array}{lll}
q_k(q_n(x,\mathbf y),\mathbf z)  & =_{(C4)}  &  q_n(q_n(x,\mathbf y),\mathbf z, \e_{k+1},\dots,\e_n) \\
  & =_{(C5)}  & q_n(x,q_n(y_1,\mathbf z, \e_{k+1},\dots,\e_n),\dots,q_n(y_n,\mathbf z, \e_{k+1},\dots,\e_n))  \\
  & =_{(C4)}  &  q_n(x,q_k(y_1,\mathbf z),\dots,q_k(y_n,\mathbf z))
\end{array}
\]
\end{proof}

\begin{lemma}\label{lem:24}
 Let $\mathbf C= (\mathbf C_\tau,q_n^\mathbf C,\e_i^\mathbf C)$ be a clone $\tau$-algebra and $\mathbf b= b_1,\dots,b_n\in C$. Then the map $s_\mathbf b:C\to C$, defined by
 $$s_\mathbf b(a)=q_n(a,\mathbf b)\quad\text{for every $a\in C$},$$
 is an endomorphism of the $\tau$-algebra $\mathbf C_\tau$, satisfying $s_\mathbf b(\e_i)=b_i$ for $1\leq i\leq n$, and $s_\mathbf b(\e_i)=\e_i$ for  $i > n$.
\end{lemma}

\begin{proof} By (C6).
\end{proof}

In the remaining part of this section we define the notions of independence and dimension in clone algebras.

\begin{definition}
  An element $a$ of a clone algebra $\mathbf  C$ \emph{is independent of} $\e_{n}$ if 
$q_{n}(a,\e_1,\dots,\e_{n-1},\e_{n+1})=a$. If $a$ is not independent of $\e_{n}$, then we say that 
$a$ \emph{is dependent on} $\e_{n}$.
\end{definition}

\begin{lemma}\label{lem:ind1} Let $\mathbf C$ be a clone algebra and $\mathbf b=b_1,\dots, b_{n-1}\in C$. 
If $k \geq n$ and $a\in C$ is independent of $\e_n,\e_{n+1}\dots,\e_k$, then
$$q_k(a,\mathbf b,b_n,\dots,b_k)=q_{n-1}(a,\mathbf b),\quad\text{for all $b_n,\dots,b_k\in C$}.$$
\end{lemma}

\begin{proof} Let $\mathbf e = \e_1,\dots,\e_{n-1}$. First we analyse the case $k=n$.
Since
\begin{equation}\label{mah}
q_n(a,\mathbf b,b_n)=_{(hyp)} q_n(q_n(a,\mathbf e,\e_{n+1}),\mathbf b,b_n)=_{(C5,C1,C2)} q_n(a,\mathbf b,\e_{n+1}),
\end{equation}
then 
\begin{equation}\label{mah2} q_{n-1}(a,\mathbf b)=_{(C4)} q_n(a,\mathbf b,\e_n)=_{(\ref{mah})}q_n(a,\mathbf b,\e_{n+1}) =_{(\ref{mah})} q_n(a,\mathbf b,b_n).
\end{equation}
The general case is obtained by applying several times (\ref{mah2}) to $q_{n-1}(a,\mathbf b)$.
\end{proof}

Let $a$ be an element of a clone algebra $\mathbf C$.  We define
$$\Gamma(a)= \{i:\ \text{$a$ is dependent on $\e_i$} \};\qquad
\gamma(a)=\begin{cases}
\omega&\text{if $\Gamma(a)$ is infinite}\\ 
0&\text{if $\Gamma(a)$ is empty}\\ 
\mathrm{max}\,\Gamma(a) &\text{otherwise}\\
\end{cases}$$
An element $a\in C$ is said to be: (i) \emph{$k$-dimensional} if $\gamma(a)=k$; (ii) \emph{finite dimensional} if it is $k$-dimensional for some $k<\omega$; (iii) \emph{zero-dimensional}  if $\gamma(a)=0$.

\begin{example}
  The nullary operator
$\e_i$ is $i$-dimensional, because $\Gamma(\e_i)=\{i\}$ and $\gamma(\e_i)=i$.
\end{example}

We consider the following subsets of  $\mathbf C$:
\begin{itemize}
\item The set $\mathrm{Fi}_k\,\mathbf C$ of all elements of $\mathbf C$ whose dimension is $\leq k$;
\item  The set $\mathrm{Fi}\,\mathbf C=\bigcup \mathrm{Fi}_k\,\mathbf C$ of all finite dimensional elements of $\mathbf C$.
\end{itemize}
We say that $\mathbf C$ is \emph{finite dimensional} if $C=\mathrm{Fi}\,\mathbf C$.

\begin{lemma}\label{lem:preservare}
\begin{itemize}
\item[(i)]  If $a,\mathbf b$ have dimension $\leq k$, then $\sigma(\mathbf b)$ and $q_n(a,\mathbf b)$ have dimension $\leq k$.
\item[(ii)] If $h: \mathbf C\to \mathbf D$ is a homomorphism of clone algebras and $a\in C$ has dimension $\leq k$, then $h(a)$ has dimension $\leq k$ in $\mathbf D$.
\end{itemize}
\end{lemma}

\begin{proof} (i) Let $m > k$,  $\mathbf e =\e_1,\dots,\e_{m-1},\e_{m+1}$ and $\mathbf d =\e_1,\dots,\e_{m-1}$. Then
 $$q_m(\sigma(\mathbf b),\mathbf e)   =_{(C6)}   \sigma(q_m(b_1,\mathbf  e),\dots, q_m(b_n,\mathbf e))
  =_{b_i\ \text{ind.}\ \e_m} 
     \sigma(\mathbf b).$$
 If $m > n$, then we have:
    \[
\begin{array}{lll}
q_m(q_n(a,\mathbf b),\mathbf e)  & =_{\mathrm{Lem}\, \ref{allungo}(i)}  &  q_m(a,q_m(b_1,\mathbf e),\dots,q_m(b_n, \mathbf e),\e_{n+1}\dots,\e_{m-1},\e_{m+1})  \\
  &  =_{b_i\ \text{ind.}\ \e_m} &  q_m(a,\mathbf b,\e_{n+1}\dots,\e_{m-1},\e_{m+1}) \\
  &  =_{a\ \text{ind.}\ \e_m,\ \mathrm{Lem}\, \ref{lem:ind1}} &  q_{m-1}(a,\mathbf b,\e_{n+1}\dots,\e_{m-1}) \\
  & =_{(C4)}  &   q_n(a,\mathbf b)
\end{array}
\]
Similarly, if $m\leq n$.

(ii) Trivial.
\end{proof}

\begin{proposition} Let $\mathbf C$ be a clone $\tau$-algebra. Then we have:
\begin{itemize}
\item[(i)]    $\mathrm{Fi}\,\mathbf C$ is a subalgebra of $\mathbf C$.
\item[(ii)]  $\mathrm{Fi}_k\,\mathbf C$ is a subalgebra of $\mathbf C_\tau$ closed under all $q$-operators and containing $\e_1,\dots,\e_k$.
\item[(iii)] $a\in \mathrm{Fi}_0\,\mathbf C\ \text{and}\ \mathbf b\in C^n\ \Longrightarrow\ q_n(a,\mathbf b)=a$.
\item[(iv)] $a\in \mathrm{Fi}\,\mathbf C,  n\geq \gamma(a)\ \text{and}\ \mathbf b\in (\mathrm{Fi}_0\,\mathbf C)^n\ \Longrightarrow\ q_n(a,\mathbf b)\in \mathrm{Fi}_0\,\mathbf C$.
\end{itemize}
\end{proposition}

\begin{proof}
 (i)-(ii) By Lemma \ref{lem:preservare}.
 
 (iii) The proof is by induction on $n$. By (C3) $q_0(a)=a$.
By Lemma \ref{lem:ind1} and by applying the induction hypothesis we get $q_n(a,b_1,\dots,b_{n-1},b_n)=q_{n-1}(a,b_1,\dots,b_{n-1})=a$.

(iv) Let $\mathbf e = \e_1,\dots,\e_{k-1}$. If $k\leq n$,  then
\[
\begin{array}{lll}
 q_k(q_n(a,\mathbf b),\mathbf e,\e_{k+1}) & =_{\text{Lem}\, \ref{allungo}(ii)}  & q_n(a,q_k(b_1,\mathbf e,\e_{k+1}),\dots,q_k(b_n,\mathbf e,\e_{k+1}))  \\
  & =_{b_i\in \mathrm{Fi}_0\,\mathbf C}  & q_n(a,\mathbf b)  \\
\end{array}
\]
 If $k >n$, then
 \[
\begin{array}{llll}
 q_k(q_n(a,\mathbf b),\mathbf e,\e_{k+1}) & =_{Lem\, \ref{allungo}(i)}  & q_k(a,q_k(b_1,\mathbf e,\e_{k+1}),\dots,q_k(b_n,\mathbf e,\e_{k+1}),\e_{n+1},\dots,\e_{k-1},\e_{k+1}) & \\
  & =_{b_i\in \mathrm{Fi}_0\,\mathbf C}  & q_k(a,\mathbf b,\e_{n+1},\dots,\e_{k-1},\e_{k+1}) & \\
  & =_{\gamma(a)<k}  & q_{k-1}(a,\mathbf b,\e_{n+1},\dots,\e_{k-1}) & \\
    & =_{(C4)}  & q_{n}(a,\mathbf b). & \\
\end{array}
\]

\end{proof}

We conclude this section with an example. Other examples of clone algebras will be presented in  Section \ref{sec:dueuno} and Section \ref{sec:clones}.

\begin{example} \label{exa:free} Let $\mathcal V$ be a variety of algebras of type $\tau$ and $\mathbf F_\mathcal V$ be its free algebra over a countable infinite set $I=\{v_1,\dots,v_n,\dots\}$ of generators. 
 We define  an $n+1$-ary operation $q_n^\mathbf F$ on $\mathbf{F}_\mathcal{V}$ as follows (see \cite[Definition 3.2]{mac2}). For
$a,b_1,\dots,b_n\in F_\mathcal{V}$ we put 
$$ q_n^\mathbf F(a,b_1,\dots,b_n)=s(a),$$
where $s$ is the unique endomorphism
of $\mathbf{F}_\mathcal{V}$ which sends the generator $ v_i\in I$ to $b_i$ ($1\leq i\leq n$). 
More suggestively: $q_n^\mathbf F(a,b_1,\dots,b_n)$ is the equivalence class of the term 
$t[w_1/v_1,\dots,w_n/v_n]$, where $t\in a$ and $w_i\in b_i$. 
If we put $\e_i^\mathbf F=v_i$, then the algebra $(\mathbf F_\mathcal V,q_n^\mathbf F, \e_i^\mathbf F)$ is a clone $\tau$-algebra.
We remark that Lampe's proof of the Zipper condition described in Section \ref{sec:pre:leq} uses the operator $q_2^\mathbf F$  (see the proof of McKenzie Lemma in \cite{lampe}).
\end{example}



\section{Functional Clone Algebras}\label{sec:dueuno}
 The most natural $\mathsf{CA}$s, the ones the axioms are intended to characterise, are algebras of functions, called \emph{functional clone algebras}. They will be introduced in this section.
 The elements of a functional clone algebra are infinitary functions from $A^\omega$ into $A$, for a given set $A$. In this framework  $q_n(\varphi,\psi_1,\dots,\psi_n)$ represents the $n$-ary composition of $\varphi$ with $\psi_1,\dots,\psi_n$, acting on the first $n$ coordinates:
 $$q_n(\varphi,\psi_1,\dots,\psi_n)(s)= \varphi(\psi_1(s),\dots,\psi_n(s),s_{n+1},s_{n+2},\dots),\ \text{for every $s\in A^\omega$}$$
and the nullary operators are the projections $p_i$ defined by $p_i(s)=s_i$ for every $s\in A^\omega$.
Hence, the universe of a functional clone algebra is a set of infinitary operations containing the projection $p_i$ and closed under finitary compositions, called hereafter \emph{$\omega$-clone}.
 We will see in Section \ref{sec:clones}  that there exists a bijective correspondence between clones (of  operations) and a suitable class of functional clone algebras, called \emph{block algebras}. 
 Given a clone, the corresponding block algebra is obtained by extending the operations of the clone by countably many dummy arguments.

  A clone algebra is functionally representable if it is isomorphic to a functional clone algebra. One of the main results of this paper is the general representation theorem of Section \ref{sec:GRT}, where it is shown that every $\mathsf{CA}$ is functionally representable. 
Therefore, the clone algebras are the algebraic counterpart of $\omega$-clones, while the block algebras are the algebraic counterpart of clones. By Corollary \ref{cor:cablk}  the variety of clone algebras is generated by the class of block algebras. Then every $\omega$-clone is algebraically generated by a suitable family of clones by using direct products, subalgebras and homomorphic images.

\bigskip

Let $A$ be a set and   $\mathcal O_A^{(\omega)}$ be
 the set of all infinitary operations from $A^\omega$ into $A$. If $r\in A^\omega$ and $a_1,\dots,a_n\in A$ then $r[a_1,\dots,a_n]\in A^\omega$ is defined by
 $$r[a_1,\dots,a_n](i)=\begin{cases}a_i&\text{if $i\leq n$}\\ r_i&\text{if $i > n$}\end{cases}$$

\begin{definition}\label{def:fca}
 Let $\mathbf A$ be a $\tau$-algebra.
The algebra ${\mathbf O}_\mathbf A^{(\omega)}=(\mathcal O_A^{(\omega)},\sigma^\omega, q^\omega_n,\e^\omega_i)$, where, for every $s\in A^\omega$ and $\varphi,\psi_1,\dots,\psi_n\in \mathcal O_A^{(\omega)}$, 
\begin{itemize}
\item $\e^\omega_i(s)=s_i$;
\item $q^\omega_n(\varphi,\psi_1,\dots,\psi_n)(s)=\varphi(s[\psi_1(s),\dots, \psi_n(s)])$;
\item $\sigma^\omega(\psi_1,\dots,\psi_n)(s)=\sigma^\mathbf A(\psi_1(s),\dots, \psi_n(s))$ for every $\sigma\in\tau$ of arity $n$;
\end{itemize}
 is called the \emph{full functional clone $\tau$-algebra with value domain $\mathbf A$}.
\end{definition}

\begin{lemma}\label{lem:fun} 
The algebra ${\mathbf O}_\mathbf A^{(\omega)}$ is a clone $\tau$-algebra.
\end{lemma}


\begin{definition} 
 A subalgebra of ${\mathbf O}_\mathbf A^{(\omega)}$  is called a \emph{functional clone algebra with value domain} $\mathbf A$.
\end{definition}
 
 The class of  functional clone algebras is denoted by $\mathsf{FCA}$.  $\mathsf{FCA}_\tau$  is the class of $\mathsf{FCA}$s   whose value domain is a $\tau$-algebra.
 We also use $\mathsf{FCA}_\tau$ as shorthand for the phrase ``functional clone algebra of type $\tau$'', and similarly for $\mathsf{FCA}$.

 
In the following lemma the algebraic and functional notions of independence  are shown to be equivalent.

\begin{lemma}\label{lem:independent}
  An infinitary operation  $\varphi\in \mathcal O_A^{(\omega)}$  is \emph{independent of $\e_n$} iff, for all $s,u\in A^\omega$, $u_i=s_i$ for all $i\neq n$ implies $\varphi(u) = \varphi(s)$.
\end{lemma}

\begin{proof} Let $\mathbf e=\e_1,\dots,\e_{n-1}$ and $s,u\in A^\omega$ such that $u_i=s_i$ for all $i\neq n$. Let 
$\mathbf u=u_1,\dots,u_{n-1}$ and $\mathbf s=s_1,\dots,s_{n-1}$.

($\Rightarrow$) If $\varphi=q^\omega_n(\varphi,\mathbf e,\e_{n+1})$ then 
$\varphi(u)= q^\omega_n(\varphi,\mathbf e,\e_{n+1})(u)= \varphi(u[\mathbf u,u_{n+1}])= 
 \varphi(s[\mathbf s,s_{n+1}])=\dots=\varphi(s)$, because $u_i=s_i$ for all $i\neq n$.
 
($\Leftarrow$) $\varphi(s)=  \varphi(s[\mathbf s,s_n])= \varphi(s[\mathbf s,s_{n+1}])= q^\omega_n(\varphi,\mathbf e,\e_{n+1})(s)$, because
$s_i=s[\mathbf s,s_n]_i= s[\mathbf s,s_{n+1}]_i$ for all $i\neq n$.
\end{proof}



\begin{example}\label{exa:semi}
  We now provide an example of a zero-dimensional element of a $\mathsf{FCA}$ that is not a constant function.
A function $\varphi:A^\omega\to A$ is \emph{semiconstant} if it is not constant and, for every $r,s\in A^\omega$, $|\{i: r_i\neq s_i\}|<\omega$ implies $\varphi(r)=\varphi(s)$. Let $2=\{0,1\}$.
The function $\varphi:2^\omega\to 2$, defined by 
$$\varphi(s)=\begin{cases}0&\text{if $|\{i:s_i=0\}|<\omega$}\\
1&\text{otherwise}\end{cases}$$
is an example of semiconstant function.
It is easy to see that every  semiconstant function  is zero-dimensional in the full $\mathsf{FCA}$ $\mathcal O_A^{(\omega)}$.

\end{example}

\begin{example} Let $2=\{0,1\}$. The function $\psi:2^\omega\to 2$, defined by 
$$\psi(s)=\begin{cases}0&\text{if $|\{i:s_i=0\}|$ is finite and even }\\
1&\text{otherwise}\end{cases}$$
is infinite dimensional. 
\end{example}

In the following proposition we show that the notions of functional clone algebra and Neumann's functional $\aleph_0$-clone (see Section \ref{sec:neu}) are distinct.

\begin{proposition} Every  abstract $\aleph_0$-clone is a clone algebra, but there are functional clone algebras that are not functional $\aleph_0$-clones.
\end{proposition}

\begin{proof} Every  abstract $\aleph_0$-clone is a clone algebra because 
$$q_n(x,y_1,\dots,y_n)=q_\infty(x,y_1,\dots,y_n,\e_{n+1},\e_{n+2},\dots).$$
We now show that the subalgebra $\{\psi:2^\omega \to 2\ |\ \text{$\psi$ is semiconstant}\}\cup\{\e_i^\omega\ |\  i\geq 1\}$  of the full $\mathsf{FCA}$ $\mathcal O_2^{(\omega)}$ is not an abstract $\aleph_0$-clone.
  Let $s,r\in 2^\omega$ such that $s_i=0$ and $r_i=1$ for all $i$. 
 If $\varphi:2^\omega\to 2$ is any semiconstant function such that $\varphi(s)=0$ and  $\varphi(r)=1$, then $q_\infty^\omega(\varphi,\e_1^\omega,\e_1^\omega,\dots,\e_1^\omega,\dots)$ is not a semiconstant function. 
\end{proof}

 \section{Clones of operations and block algebras}\label{sec:clones}
 In this section we introduce an equivalence relation over the set $\mathcal O_A$ of operations of a given set $A$, in order to turn $\mathcal O_A$ into a functional clone algebra. Roughly speaking, two operations are equivalent if the one having greater arity extends the other one by a bunch of dummy arguments.
Each \emph{block} (equivalence class) of this equivalence relation determines univocally an infinitary function that we call the\emph{ top extension} of the block. The set of these top extensions is a $\omega$-clone and it is exactly the functional clone algebra associated to $\mathcal O_A$, called the \emph{full block algebra on $A$}. A block algebra on $A$ is a subalgebra of the full block algebra on $A$. In the last result of this section we prove that there exists a bijective correspondence between clones and block algebras.
 
 \bigskip
 
We  define a partial order $\preceq$ on the set $\mathcal O_A$ of all operations (see \cite{sangalli}). For all $f\in \mathcal O^{(k)}_A$ and $g\in \mathcal O^{(n)}_A$ we put
$$f\preceq g \Leftrightarrow k\leq n\ \text{and}\ \forall \mathbf a\in A^k\ \forall \mathbf b\in A^{n-k}:\ f(\mathbf a)=g(\mathbf a,\mathbf b).$$ 
Using the terminology of Section \ref{sec:clo}, the operation $g$ is fictitious in the last $n-k$ arguments.

\begin{definition}\label{def:similar}
 We say that two operations $f,g\in \mathcal O_A$ are \emph{similar}, and we write $f\approx_{\mathcal O_A} g$, if  either $f\preceq g$ or $g\preceq f$.
\end{definition}

\begin{lemma} \cite[Lemma 1]{sangalli}
 \begin{enumerate}
\item The relation $\approx_{\mathcal O_A}$ is an equivalence relation on $\mathcal O_A$. 
\item Each block of the relation $\approx_{\mathcal O_A}$ is totally ordered by $\preceq$ and has a minimal element. 
\end{enumerate}
\end{lemma}



We denote by $\mathcal{B}_A$ the set of all blocks of the relation $\approx_{\mathcal O_A}$.

If $f\in\mathcal O_A$ then $\langle f\rangle$ denotes the unique block containing $f$.

\begin{definition}
\begin{enumerate}
\item  An operation $f$ is said to be \emph{a generator} if $f$ is the minimal element w.r.t. $\preceq$ of the unique block $\langle f\rangle$ containing $f$.
\item A block $B$ \emph{has arity} $k$ if the generator of the block $B$ has arity $k$.
\end{enumerate}
\end{definition}


If $B$ is a block of arity $k$, then $|B \cap \mathcal O^{(n)}_A|= 1$ for every $n\geq k$, and $|B \cap \mathcal O^{(n)}_A|= \emptyset$ for every $n< k$.

As a matter of notation, if $B$ is a block of arity $k$, then we denote by $B^{(n)}$ ($n\geq k$) the unique function in $B\cap \mathcal O^{(n)}_A$. Therefore, $B = \{B^{(n)} : n\geq k\}$ and $B^{(k)}$ is the generator of the block $B$.

\begin{lemma}\label{lem:generator}
 An operation $f:A^k\to A$ is a generator if and only if either $k=0$ or there are $a_1,\dots,a_{k-1}$, $b,c\in A$ such that
 $f(a_1,\dots,a_{k-1},b) \neq f(a_1,\dots,a_{k-1},c)$.
\end{lemma}

In other words, $f$ is a generator iff it is not fictitious according to Definition \ref{def:fic}(iii).

\begin{example} The constant operations of value $a$ are all equivalent: $c^{(n)}_a\approx_{\mathcal O_A} c^{(k)}_a$ for all $n$ and $k$. The block containing all $c^{(n)}_a$ has arity $0$, because it is generated by $c^{(0)}_a$. This block  will be denoted by $C_a$ and will be called \emph{constant block (of value $a$)}.
\end{example}

\begin{example} We have $p^{(n)}_i\approx_{\mathcal O_A} p_i$ for every $n\geq i$, where $p_i=p^{(i)}_i$.
The block generated by the basic projection $p_i$ contains all the projections $p^{(n)}_i$ ($n\geq i$) and has arity $i$. This block will be denoted by $P_i$ and will be called \emph{projection block}.
\end{example}

\begin{lemma}\label{lemclone} Every clone  is union of blocks.
\end{lemma}

\begin{proof} Let $F$ be a clone on $A$, $f\in  F$ be an operation of arity $n$ and $g:A^k\to A$ be the generator of the block $\langle f\rangle$. If $f= c_a^{(n)}$ is a constant function, then by Definition \ref{def:clo} the element $g=c_a^{(0)}$ of $A$ belongs to $F$ and $\langle f\rangle^{(m)}=g()_m\in  F$ for all $m\geq 0$.
  If $f$ is not constant, then $g=f(p^{(k)}_1,\dots,p^{(k)}_k, p^{(k)}_k,\dots,p^{(k)}_k)_k\in F$ and $\langle f\rangle^{(m)}=g(p^{(m)}_1,\dots,p^{(m)}_k)_m\in  F$ for all $m\geq k$. In conclusion, $\langle f\rangle \subseteq  F$. 
\end{proof}

\begin{example}\label{exa:to} Let $\mathbf A$ be a $\tau$-algebra and $\mathrm{Clo}\mathbf A$ be the clone of its term operations.
 If $t$ is a $\tau$-term, then the set $T_t^\mathbf A$ (defined in Section \ref{sec:to}) of the term operations determined by $t$ is a block.
Moreover,  $B\subseteq \mathrm{Clo}\mathbf A$ is a block if and only if $B=T_t^\mathbf A$ for some term $t$.
\end{example}

\subsection{Block algebras and top extensions of blocks}\label{sec:blocktop}
In this section we study the properties that a family $G$ of blocks of the relation $\approx_{\mathcal O_A}$ must have 
for  $\bigcup G$ to be a clone on $A$.
To simplify the approach it is convenient to work with the set $\mathcal O^{(\omega)}_A$ of infinitary operations from $A^\omega$ to $A$.

\begin{definition} The \emph{top operator} is a map $(-)^\top:\mathcal O_A\to\mathcal O^{(\omega)}_A$
 defined as follows, for every $f\in \mathcal O_A^{(n)}$:
 $$f^\top(s)=f(s_1,\dots,s_n),\quad \text{for all $s\in A^\omega$}.$$ 
\end{definition}
The infinitary operation  $f^\top$, defined by Neumann \cite{neu70} to formalise $\aleph_0$-clones (see Section \ref{sec:neu}), will be called
 \emph{the top extension} of the operation $f\in \mathcal O_A$.
 
The proof of the following lemma is trivial.

\begin{lemma}\label{lem:similar} Let $f,g\in\mathcal O_A$. Then
  $f\approx_{\mathcal O_A} g$ iff $f^\top = g^\top$.
\end{lemma}

In other words, the kernel of the top operator coincides with the relation of similarity among operations. This means that the set $\mathcal B_A$ of blocks of the relation $\approx_{\mathcal O_A}$ coincides with the set $\mathcal O_A$ modulo the kernel of the top operator.

By Lemma \ref{lem:similar} the \emph{top extension $B^\top$ of a block} $B$ can be well defined as  
$B^\top=f^\top$ for some (and then all) $f\in B$.
Then the map $B\mapsto B^\top$ embeds the set $\mathcal B_A$ of blocks into $\mathcal O^{(\omega)}_A$. Its
 image $\{B^\top: B\in \mathcal B_A\}$ will be denoted by $\mathcal B_A^\top$. $\mathcal B_A$ and $\mathcal B_A^\top$ are equipotent sets.

If $\varphi$ is the top extension of a block, then we denote by $\varphi_\bot$ the unique block such that 
$(\varphi_\bot)^\top=\varphi$. By Lemma \ref{lem:similar} the block $\varphi_\bot$ is well defined.

Notice that the notion of dimension is an intrinsic property of a function $\varphi\in \mathcal O^{(\omega)}_A$: if $\varphi$ has dimension $k$ in a $\mathsf{FCA}$, then by Lemma \ref{lem:independent} $\varphi$ has dimension $k$ in every $\mathsf{FCA}$ containing $\varphi$.


\begin{lemma}\label{lem:ardim}
A block $B\in\mathcal B_A$  has arity $r$  if and only if $B^\top$ has dimension $r$.
\end{lemma}

\begin{proof} ($\Rightarrow$) First we prove that, if $r$ is the arity of a block $B$, then $B^\top$ is dependent on $\e_r$.
Let $a_1,\dots, a_{r-1},b,c\in A$ such that $B^{(r)}(a_1,\dots,a_{r-1},b)\neq B^{(r)}(a_1,\dots,a_{r-1},c)$. Let $s,u\in A^\omega$ such that $s_r=b$,  $u_r=c$, $s_i=u_i=a_i$ for every $i=1,\dots, r-1$  and $s_j=u_j$ for every $j > r$.
Then $B^\top(s)= B^{(r)}(s_1,\dots,s_r)\neq B^{(r)}(u_1,\dots,u_r)=B^\top(u)$. By Lemma \ref{lem:independent}  $B^\top$ is dependent on $\e_r$, where $r$ is the arity of the block $B$.

We now show that $B^\top$ is independent of $\e_k$ for every $k > r$, the arity of $B$. 
For every $s,u\in A^\omega$ such that $s_i=u_i$ for every $i\neq k$ we have:
$$B^\top(s)= B^{(r)}(s_1,\dots,s_r)=_{(\text{$s_i=u_i$ for $i\leq r$})}B^{(r)}(u_1,\dots,u_r) =B^\top(u).$$ 
Then by Lemma \ref{lem:independent} $B^\top$ is independent of $\e_k$.

($\Leftarrow$) Let $k$ be the arity of the block $B$. Since $B^\top(s)= B^{(k)}(s_1,\dots,s_k)$ for all $s\in A^\omega$, then it is easy to verify that $k=r$.
\end{proof}

There exist finite dimensional elements of $\mathcal O^{(\omega)}_A$ that are not top extension of a block.
The semiconstant functions are defined in Example \ref{exa:semi}.

\begin{lemma}\label{lem:semiconstant}
Every semiconstant function $\varphi\in \mathcal O^{(\omega)}_A$  is zero-dimensional, but it is not the top extension of any  operation. 
\end{lemma}

\begin{proof} The function $\varphi:\{0,1\}^\omega\to \{0,1\}$ defined in Example \ref{exa:semi} is zero-dimensional but it is not the top extension of a constant.
 \end{proof}
 
 We now are ready to define a structure of clone algebra on $\mathcal B_A^\top$.
 
 Recall that the full $\mathsf{FCA}$ $\mathbf O^{(\omega)}_A$ with value domain $A$ was introduced in Lemma \ref{lem:fun}.

\begin{lemma}
 $\mathcal B_A^\top$ is a finite dimensional subalgebra of the full $\mathsf{FCA}$ $\mathbf O^{(\omega)}_A$ with value domain $A$.
\end{lemma}

\begin{proof}
 First $\e_i^\omega = P_i^\top$, where $P_i$ is the block of all projections $p^{(k)}_i$. We now show that $\mathcal B_A^\top$ is closed under the operations $q_n^\omega$. Let $B,G_1,\dots,G_n$ be blocks and let $k\geq n$ be greater than the arities of $B,G_1,\dots,G_n$. We now show that $q_n^\omega(B^\top,G_1^\top,\dots,G_n^\top)$ is the top extension of a suitable function of arity $k$. Let $s\in A^\omega$.
 \[
\begin{array}{lll}
  &   & q_n^\omega(B^\top,G_1^\top,\dots,G_n^\top)(s)  \\
  &  = & B^\top(s[G_1^\top(s),\dots,G_n^\top(s)])  \\
  & =  &   B^\top(s[G_1^\top(s),\dots,G_n^\top(s),s_{n+1},\dots,s_k])\\
   & =  &B^{(k)}(G_1^\top(s),\dots,G_n^\top(s),s_{n+1},\dots,s_k)\\
 & =  &  B^{(k)}(G_1^{(k)}(s_{1},\dots,s_k),\dots, G_n^{(k)}(s_{1},\dots,s_k),s_{n+1},\dots,s_k)\\
 & =  &  B^{(k)}(G_1^{(k)}(s_{1},\dots,s_k),\dots, G_n^{(k)}(s_{1},\dots,s_k),P_{n+1}^{(k)}(s_{1},\dots,s_k),\dots,P_{k}^{(k)}(s_{1},\dots,s_k))\\
 & =  &  [B^{(k)}(G_1^{(k)},\dots, G_n^{(k)},P_{n+1}^{(k)},\dots,P_{k}^{(k)})]^\top(s).
\end{array}
\]
\end{proof}

The $\mathsf{FCA}$ $\mathcal B_A^\top$ will be called \emph{the full block algebra on $A$}.

\begin{definition}
 A \emph{block algebra on $A$} is a subalgebra of the full block algebra $\mathcal B_A^\top$.
\end{definition}

 A \emph{block algebra on a $\tau$-algebra $\mathbf A$} is a block algebra on $A$ containing $\langle\sigma^\mathbf A\rangle^\top$ for every $\sigma\in\tau$.

By Lemma \ref{lem:semiconstant} it follows the following corollary.

\begin{corollary}
 The finite dimensional $\mathsf{FCA}$  $\mathrm{Fi}\, \mathbf O^{(\omega)}_A$ with value domain $A$ is not a block algebra on $A$.
\end{corollary}
 
The above corollary does not contradict  Theorem \ref{thm:firstrepresentation} below, where it is shown that every finite dimensional clone algebra is isomorphic to a block algebra.


 If $F\subseteq \mathcal O_A$, then we define $F^\top=\{f^\top: f\in F\}$. 
 If $G\subseteq \mathcal B_A^\top$ then we define $G_\bot=\{\varphi_\bot: \varphi\in G\}$, where $\varphi_\bot$ is a block for every $\varphi\in G$.
 
 
 \begin{proposition}\label{prop:cloneblock}
 Let $F\subseteq \mathcal O_A$ and $\mathbf A$ be an algebra. Then the following conditions are equivalent:
\begin{description}
\item[(i)] $F$ is a clone on $\mathbf A$;
\item[(ii)] $F^\top$ is the universe of a block algebra on $\mathbf A$.
\end{description}
Moreover,  $\bigcup\, (F^\top)_\bot=F$.
\end{proposition}

\begin{proof} (i) $\Rightarrow$ (ii)
 First we have $(p^n_i)^\top =\e_i^\omega$. We now check the closure under $q_n^\omega$ by showing that
$q_n^\omega(f^\top,g_1^\top,\dots,g_n^\top)\in F^\top$ for all $f,g_1,\dots,g_n\in F$.
 Let $k\geq n$ be greater than the arities of $f,g_1,\dots,g_n$. 
For every $s\in A^\omega$, we have:
\begin{equation}\label{bbbbb}
\begin{array}{rll}
q_n^\omega(f^\top,g_1^\top,\dots,g_n^\top)(s)&  = &f^\top(s[g_1^\top(s),\dots,g_n^\top(s)])\\
&  = &\langle f\rangle^{(k)}(g_1^\top(s),\dots,g_n^\top(s),s_{n+1},\dots,s_k)\\
&  = &\langle f\rangle^{(k)}(\langle g_1\rangle^{(k)}(s_1,\dots,s_k),\dots, \langle g_n\rangle^{(k)}(s_1,\dots,s_k),s_{n+1},\dots,s_k)\\
&  = &\langle f\rangle^{(k)}(\langle g_1\rangle^{(k)},\dots, \langle g_n\rangle^{(k)},P^{(k)}_{n+1},\dots,P_k^{(k)})_k(s_1,\dots,s_k)\\
\end{array}
\end{equation}
where $h=\langle f\rangle^{(k)}(\langle g_1\rangle^{(k)},\dots, \langle g_n\rangle^{(k)},P^{(k)}_{n+1},\dots,P_k^{(k)})_k\in F$ because $F$ contains the blocks $\langle f\rangle$, $\langle g_i\rangle$ and $P_i$. Then $q_n^\omega(f^\top,g_1^\top,\dots,g_n^\top)$ is the top expansion of the above function $h\in F$.

(ii) $\Rightarrow$ (i) If $f\in F^{(n)}$ and $g_1,\dots,g_n\in F^{(k)}$, then $ f( g_1,\dots, g_n)_k\in q_n^\omega(f^\top,g_1^\top,\dots,g_n^\top)_\bot.$
\end{proof}

As a consequence of the above proposition, there exists a bijection between the set of clones on  $A$ and the set of block algebras on $A$.


 \begin{corollary}\label{cor:cloneblock} Let $A$ be a set. Then the following lattices are isomorphic:
\begin{enumerate}
\item The lattice $\mathrm{Lat}(\mathcal O_A)$ of all clones on $A$;
\item The lattice of all subalgebras of the full block algebra $\mathcal B_A^\top$.
\end{enumerate}
\end{corollary}

\section{The block algebra of representable functions}\label{sec:rep}
In this section we introduce the notion of \emph{representable function} in a clone algebra. 
Roughly speaking, every $k$-dimensional element $a$ of a clone algebra $\mathbf C$ determines a block of representable functions $f_n$ ($n\geq k$) through the operators $q_n$: $f_n(x_1,\dots,x_n)=q_n^\mathbf C(a,x_1,\dots,x_n)$.
The set of representable functions includes the basic operations of $\mathbf C$. The representable functions turn out to be a clone and  the top extension of this clone is isomorphic to  the subalgebra $\mathrm{Fi}\,\mathbf C$ of all finite dimensional elements of $\mathbf C$. $\mathrm{Fi}\,\mathbf C$ coincides with $\mathbf C$ whenever $\mathbf C$ is finite dimensional. It follows that every finite dimensional clone algebra is isomorphic to a block algebra. 
\bigskip

Let $\mathbf C$ be a clone $\tau$-algebra and $\sigma\in\tau$ be an operator of arity $k$.
By Lemma \ref{lem:preservare} the element  $\sigma(\e_1,\dots,\e_k)$ has dimension $\leq k$ and  it univocally determines the values $\sigma(\mathbf a)$,  for all $\mathbf a=a_1,\dots,a_k\in C$:
 \begin{equation}\label{sigma-eq}q_k(\sigma(\e_1,\dots,\e_k),\mathbf a) =_{(C6)} \sigma(q_k(\e_1,\mathbf a),\dots,q_k(\e_k,\mathbf a))=_{(C1)}\sigma(\mathbf a).\end{equation}

In the following definition we characterise the operations on $C$ that have a  behaviour similar to the basic operations.

\begin{definition}\label{def:representable} Let $\mathbf C$ be a clone algebra and $f:C^k\to C$ be a function.
 We say that $f$ is \emph{$\mathbf C$-representable}  if  $f(\e_1,\dots,\e_k)$ has dimension $\leq k$ and 
  $$ f(\mathbf a)=q_k(f(\e_1,\dots,\e_k),\mathbf a),\quad \text{for all $\mathbf a$}.$$
 We denote by $R_\mathbf C$ the set of all $\mathbf C$-representable functions.
\end{definition}

As usual, $R_\mathbf C^{(n)}$ denotes the set of all $\mathbf C$-representable functions of arity $n$.

\begin{lemma}\label{cor:sigma}
 Let $\mathbf C$ be a clone $\tau$-algebra. Then every basic operation $\sigma^\mathbf C$ ($\sigma\in\tau$) is $\mathbf C$-representable.
\end{lemma}

\begin{proof}
 By (\ref{sigma-eq}) and Definition \ref{def:representable}.
\end{proof}

In the following lemma we show that a function is $\mathbf C$-representable if and only if it satisfies an analogue of  identity (C6) in Definition \ref{def:clonealg}.

\begin{lemma}\label{12}
 Let $\mathbf C$ be a clone algebra and $f:C^k\to C$ be a function. Then the following conditions are equivalent:
 \begin{itemize}
\item[(i)] $f$ is $\mathbf C$-representable;
 \item[(ii)]  $q_n(f(\mathbf a),\mathbf c) = f(q_n(a_1,\mathbf c),\dots, q_n(a_k,\mathbf c))$ for every $n\geq 0$ and every $\mathbf a\in C^k$, $\mathbf c\in C^n$.
\end{itemize}
\end{lemma}

\begin{proof}
 (ii) $\Rightarrow$ (i) Let $\mathbf e=\e_1,\dots,\e_k$. Then
$q_k(f(\mathbf e),\mathbf a)    =  
f(q_k(\e_1,\mathbf a),\dots, q_k(\e_k,\mathbf a)) 
   =   f(\mathbf a)$. We now prove that $f(\mathbf e)$ has dimension $\leq k$.
Let $n>k$ and $\mathbf d=\e_1,\dots,\e_{n-1},\e_{n+1}$. Then $q_n(f(\mathbf e),\mathbf d)=
f(q_n(\e_1,\mathbf d),\dots, q_n(\e_k,\mathbf d))=
f(\mathbf e)$. It follows that $f(\mathbf e)$ has dimension $\leq k$.

 (i) $\Rightarrow$ (ii) 
 \begin{itemize}
\item If $k\geq n$,  then 
$$f(q_n(a_1,\mathbf b),\dots, q_n(a_k,\mathbf b))=_{(i)} q_k(f(\mathbf e),q_n(a_1,\mathbf b),\dots, q_n(a_k,\mathbf b))=_{\text{Lem}\, \ref{allungo}(ii)} q_n(q_k(f(\mathbf e),\mathbf a),\mathbf b).$$
\item If $k< n$, then  by Lemma \ref{lem:ind1} we obtain:
\[
\begin{array}{lll}
f(q_n(a_1,\mathbf b),\dots, q_n(a_k,\mathbf b))  & =_{(i)}  & q_k(f(\mathbf e),q_n(a_1,\mathbf b),\dots, q_n(a_k,\mathbf b))  \\
  & =_{\text{Lem}\, \ref{lem:ind1}}  &  q_n(f(\mathbf e),q_n(a_1,\mathbf b),\dots, q_n(a_k,\mathbf b),b_{k+1},\dots,b_n) \\
  & =_{\text{Lem}\, \ref{allungo}(i)}  &   q_n(q_k(f(\mathbf e),\mathbf a),\mathbf b).
\end{array}
\]
\end{itemize}
\end{proof}

Let $\mathbf C$ be a clone algebra. For every $a\in C$ of finite dimension, we consider
the family $$R(a)=\bigcup_{n\in \omega} \{f\in R_\mathbf C^{(n)}: a=f(\e_1,\dots,\e_n)\}$$ of the $\mathbf C$-representable functions determined by $a$.

\begin{proposition}\label{prop:rr} Let $\mathbf C$ be a clone algebra and and $a,b$ be finite-dimensional elements of C. Then the following conditions hold: 
\begin{enumerate}
\item For every $f\in R_\mathbf C^{(n)}$  and $g\in R_\mathbf C^{(k)}$, $f\approx_{\mathcal O_C} g$ iff $f(\e_1,\dots,\e_n)=g(\e_1,\dots,\e_k)$.
\item $R(a)$ is a block.
\item $R(a)= R(b) \Rightarrow a=b$.
\item $R_\mathbf C=\bigcup_{a\in \mathrm{Fi}\,C} R(a)$ is a clone on $\mathbf C$.
\item The block $R(a)$ has arity $k$ iff the element $a$ has dimension $k$ in $\mathbf C$.
\end{enumerate}
\end{proposition}

\begin{proof}  (1) ($\Rightarrow$) If $n\leq k$, then $f(\e_1,\dots,\e_n)=g(\e_1,\dots,\e_n,\mathbf b)$ for every $\mathbf b$. In particular for $\mathbf b=\e_{n+1},\dots,\e_k$ we get the conclusion. ($\Leftarrow$) It is trivial by the hypotheses.
 
 (2) By (1).

 (3) If $R(a)=R(b)$ and $f\in R(a)$ has arity $n$, then $a=f(\e_1,\dots,\e_n)=b$.
 
 (4) The projection $p^{(n)}_i$ is $\mathbf C$-representable:
$$a_i=p^{(n)}_i(a_1,\dots,a_n)=q_n(\e_i,a_1,\dots,a_n).$$
If $f$ of arity $n$ and $g_1,\dots,g_n$ of arity $k$ are $\mathbf C$-representable, then the function $h=f(g_1,\dots,g_n)_k$ is also
 $\mathbf C$-representable. Let $\mathbf e=\e_1,\dots,\e_k$ and $\mathbf a=a_1,\dots,a_k$. Then we have:
 \[
\begin{array}{llll}
q_k(h(\e_1,\dots,\e_k),\mathbf a)& =  &q_k(f(g_1(\mathbf e),\dots,g_n(\mathbf e)),\mathbf a) &\\
 & =  &f(q_k(g_1(\mathbf e),\mathbf a),\dots,q_k(g_n(\mathbf e),\mathbf a))  &\text{by Lemma \ref{12}} \\
  &  = &  f(g_1(\mathbf a)),\dots,g_n(\mathbf a))   &\text{by Lemma \ref{12} and (C1)} \\
   &  = & h(\mathbf a)&
\end{array}
\]
The basic operations $\sigma^\mathbf C$ are also $\mathbf C$-representable.

 (5) If $f$ is $\mathbf C$-representable and $a=f(\e_1,\dots,\e_k)$, then $a$ has dimension $\leq k$.   The element $a$ is independent of $\e_k$ iff (by Lemma \ref{lem:ind1}), for every  $b_1,\dots,b_{k-1},c\in C$,
 $f(b_1,\dots,b_{k-1},c)=q_k(a,b_1,\dots,b_{k-1},c)=q_{k-1}(a,b_1,\dots,b_{k-1})$ iff, for every  $b_1,\dots,b_{k-1}$, $c,d\in C$, $f(b_1,\dots,b_{k-1},c)=f(b_1,\dots,b_{k-1},d)$ iff (by Lemma \ref{lem:generator}) 
 $f$ is not a generator. 
\end{proof}

%


\begin{lemma} \label{lem:dim} Let $\mathbf C=(\mathbf C_\tau,q_n,\e_i)$ be a  clone $\tau$-algebra. 
Then $R_\mathbf C^\top=\{R(a)^\top: a\in \mathrm{Fi}\,\mathbf C\}$ is a block algebra on $\mathbf C_\tau$.
\end{lemma}

\begin{proof}  By Propositions \ref{prop:cloneblock} and \ref{prop:rr}.
 \end{proof}


\begin{theorem}\label{thm:firstrepresentation} Let $\mathbf C$ be a  finite dimensional clone $\tau$-algebra.
The function $(-)^\top \circ R$ mapping
$$a\in C \mapsto\ R(a)^\top$$
is an isomorphism from $\mathbf C$ onto the block algebra $R_\mathbf C^\top$. 
\end{theorem}

\begin{proof} $(-)^\top \circ R$ is trivially bijective and $\e_i^\mathbf C\mapsto R(\e_i^\mathbf C)^\top=\e_i^\omega$. The map $(-)^\top \circ R$ preserves the operators $q_n$:
$$R(q_n^\mathbf C(a,b_1,\dots,b_n))^\top= q_n^\omega(R(a)^\top,R(b_1)^\top,\dots,R(b_n)^\top).$$
Let $k\geq n$ be greater than the arities of $R(a), R(b_1),\dots,R(b_n)$ and the dimension of $q_n^\mathbf C(a,b_1,\dots,b_n)$. Let $s\in C^\omega$, $\mathbf s=s_1,\dots,s_k$, $A=R(a)$ and $B_i=R(b_i)$.
\[
\begin{array}{rll}
q_n^\omega(A^\top,B_1^\top,\dots,B_n^\top)(s)&  = &A^\top(s[B_1^\top(s),\dots,B_n^\top(s)])\\
&  = &A^{(k)}(B_1^\top(s),\dots,B_n^\top(s),s_{n+1},\dots,s_k)\\
&  = &A^{(k)}(B_1^{(k)}(\mathbf s),\dots,B_n^{(k)}(\mathbf s),s_{n+1},\dots,s_k)\\
&  = &A^{(k)}(B_1^{(k)},\dots,B_n^{(k)},p^{(k)}_{n+1},\dots,p^{(k)}_k)(\mathbf s)\\
\end{array}
\]
Let $\mathbf b= b_1,\dots,b_n$ and  $f\in R(q_n^\mathbf C(a,\mathbf b))$ of arity $k$. Then, we have
  $$ f(\mathbf a)=q_k (f(\e_1,\dots,\e_k),\mathbf s)=q_k (q_n (a,\mathbf b),\mathbf s)=q_k(a,q_k(b_1,\mathbf s),\dots,q_k(b_n,\mathbf s),s_{n+1},\dots,s_k)$$
  $$= A^{(k)}(B_1^{(k)},\dots,B_n^{(k)},p^{(k)}_{n+1},\dots,p^{(k)}_k)(\mathbf s).$$

Moreover, for every $\sigma\in \tau$ of arity $n$, it is not difficult to show that 
$$R(\sigma^\mathbf C(\e_1,\dots,\e_n))^\top= \sigma^\omega(\e_1^\omega,\dots,\e_n^\omega).$$
%
\end{proof}

We denote by $\mathsf{BLK}$ the class of all block algebras and by $\mathsf{FiCA}$ the class of all finite dimensional clone algebras.
 
\begin{theorem}\label{thm:fica}
 $\mathsf{FiCA}_\tau = \mathbb I\, \mathsf{BLK}_\tau$.
\end{theorem}

\begin{proof} By Theorem \ref{thm:firstrepresentation} $\mathsf{FiCA}_\tau \subseteq \mathbb I\, \mathsf{BLK}_\tau$.
The  inequality $\mathsf{BLK}_\tau\subseteq \mathsf{FiCA}_\tau$ is trivial, because every block algebra is a finite dimensional clone algebra. 
\end{proof}

\section{The operators of an algebraic type as nullary operators}\label{sec:nullary}
Each $n$-ary basic operation $\sigma^\mathbf C$ of a clone $\tau$-algebra $\mathbf C$ is represented by the element $\sigma^\mathbf C(e_1^\mathbf C,\ldots,e_n^\mathbf C)$.
Taking these elements as nullary operators and discharging the $\sigma$'s, we get \emph{pure clone algebras with constants}. In this section 
we show that the variety of clone algebras of type $\tau$ is equivalent to the variety of pure clone algebras with
$\tau$-constants.

\begin{definition}
 If $\tau$ is a similarity
type, denote by $\tau^*$ the expansion of the type of pure clone algebras by a new
constant $c_\sigma$ for every $\sigma\in\tau$. A \emph{pure clone algebra with $\tau$-constants} is an
algebra $\mathbf A = (A,q_n^\mathbf A,\e_i^\mathbf A, \{c_\sigma^\mathbf A\}_{\sigma\in\tau})$ of type $\tau^*$, where $(A,q_n^\mathbf A,\e_i^\mathbf A)$ is a pure clone algebra.
\end{definition}

The variety of clone $\tau$-algebras  and the variety of pure clone algebra with $\tau$-constants are term equivalent.
Consider the following correspondence.
\begin{itemize}
\item Beginning on the clone algebra side,  if $\mathbf C=(\mathbf C_\tau, q^\mathbf C, \e_i^\mathbf C)$ is a clone $\tau$-algebra, then $\mathbf C^\bullet = (C; q^\mathbf C, \e_i^\mathbf C,c^\bullet_\sigma)$, where $c^\bullet_\sigma = \sigma^\mathbf C(\e_1^\mathbf C,\dots,\e_k^\mathbf C)$ for every $\sigma\in\tau$ of arity $k$, denotes the corresponding algebra in the similarity type of pure clone algebras with $\tau$-constants.
\item Beginning on the other  side, if $\mathbf A= (A, q^\mathbf A_n,\e^\mathbf A_i, c^\mathbf A_\sigma)$ is a pure clone algebra with $\tau$-constants, then 
$\mathbf A^* = (\mathbf A_\tau, q_n^\mathbf A,\e_i^\mathbf A)$, where $\mathbf A_\tau= (A,\sigma^*)_{\sigma\in\tau}$ and $\sigma^*(a_1,\dots,a_k)= q_k^\mathbf A(c^\mathbf A_\sigma,a_1,\dots,a_k)$ for all $a_i\in A$ and every $\sigma\in\tau$ of arity $k$, denotes the corresponding algebra in the similarity type of clone $\tau$-algebras.
\end{itemize}

It is not difficult to prove the following proposition.

\begin{proposition}\label{prop:equiv}
  The above correspondences define a term equivalence between the variety $\mathsf{CA}_\tau$ of clone $\tau$-algebras  and the variety of pure clone algebras with $\tau$-constants. More precisely,
  \begin{itemize}
\item[(i)] If $\mathbf A$ is a pure clone algebra with $\tau$-constants, then $\mathbf A^*$ is a clone $\tau$-algebra;
\item[(ii)] If $\mathbf C$ is a clone $\tau$-algebra, then $\mathbf C^\bullet$ is a pure clone algebra with $\tau$-constants;
\item[(iii)] $(\mathbf A^*)^\bullet = \mathbf A$ and $(\mathbf C^\bullet)^* = \mathbf C$.
\end{itemize}
\end{proposition}

\begin{proof} (i) By Lemma \ref{12} applied to $\sigma^*$ we get (C6).

(ii) By Lemma \ref{cor:sigma} and Definition \ref{def:representable}.

(iii) First we have $\sigma^*(\e_1,\dots,\e_k)= q_k^\mathbf A(c^\mathbf A_\sigma,\e_1,\dots,\e_k)=c^\mathbf A_\sigma$.\\
Since $c^\bullet_\sigma = \sigma^\mathbf C(\e_1^\mathbf C,\dots,\e_k^\mathbf C)$, then we have:
 $\sigma^*(a_1,\dots,a_k)= q_k^\mathbf C(c^\bullet_\sigma,a_1,\dots,a_k)=$\\
 $q_k^\mathbf C(\sigma^\mathbf C(\e_1^\mathbf C,\dots,\e_k^\mathbf C),a_1,\dots,a_k)=\sigma^\mathbf C(a_1,\dots,a_k)$.
\end{proof}

In view of this proposition, we will denote a clone $\tau$-algebra  either as $\mathbf C = (\mathbf C_\tau, q^\mathbf C_n,\e^\mathbf C_i)$ or as $\mathbf C = (C, q^\mathbf C_n,\e^\mathbf C_i, c^\mathbf C_\sigma)_{\sigma\in\tau}$, whichever seems more convenient.

\bigskip

Proposition \ref{prop:equiv} has two important consequences.

\begin{corollary}
 Let $\mathbf C$ be a clone $\tau$-algebra, $\mathbf C_0$ be the pure reduct of $\mathbf C$ and $\theta$ be an equivalence relation on $C$. Then,
 $\theta$ is a congruence on $\mathbf C$ if and only if $\theta$ is a congruence on $\mathbf C_0$; hence, 
 $$\mathrm{Con}\, \mathbf C=\mathrm{Con}\, \mathbf C_0.$$
\end{corollary}

\begin{proof}
 If $\mathbf a\theta \mathbf b$ and $\theta$ preserves the operators $q_n$, then
$\sigma^\mathbf C(\mathbf a)=_{(C6),(C1)}q_k^\mathbf C(\sigma^\mathbf C(\mathbf e),\mathbf a)\ \theta\ q_k^\mathbf C(\sigma^\mathbf C(\mathbf e),\mathbf b)=\sigma^\mathbf C(\mathbf b).$
\end{proof}

\begin{corollary}\label{cor:taunu} Let $\mathbf C$ and $\mathbf D$ be clone algebras of type $\tau$ and $\nu$, respectively.   If $\mathbf C$ and $\mathbf D$ have the same pure reduct, then  $\mathrm{Con}\, \mathbf C=\mathrm{Con}\, \mathbf D$.
\end{corollary}

\section{The General Representation Theorem}\label{sec:GRT}

This section is devoted to the proof of the main representation theorem.
Firstly we introduce the class $\mathsf{RCA}$ of \emph{point-relativized functional clone algebras}, which is instrumental in the proof of  the representation theorem.
The following diagram provides the outline of the proof that $\mathsf{CA}=\mathbb I\,\mathsf{FCA}$:
\[
\begin{array}{llll}
\mathsf{CA}  & =  & \mathbb I\,\mathsf{RCA}  &\text{Lemma \ref{lem:ca=rca}}\\
  & \subseteq  &\mathbb I\,\mathbb S\,\mathbb U_p\,  \mathsf{FCA} &\text{Lemma \ref{lem:ultrapower}, Lemma \ref{lem:red}} \\
  & \subseteq  &   \mathbb I\,\mathbb S\,\mathbb P\,  \mathsf{FCA} &\text{Lemma \ref{lem:ultrafca}} \\
  & \subseteq  &\mathbb I\, \mathsf{FCA}&\text{Lemma \ref{lem:subprod}}\\
  & \subseteq  & \mathsf{CA}&\text{Lemma \ref{lem:fun}}\\
\end{array}
\]
In other words, the proof is structured as follows:
\begin{itemize}
\item Each clone algebra is isomorphic to a point relativized functional clone algebra.
\item Each point relativized functional clone algebra embeds into an ultrapower of a functional clone algebra.
\item Each ultrapower of a functional clone algebra is isomorphic to a subdirect product of a family of functional clone algebras.
\item Functional clone algebras are closed under subalgebras and direct products.
\end{itemize}
Moreover, we prove that the variety of  clone algebras is generated by its 
finite dimensional members (or by the class of block algebras):
$$\mathsf{CA} = \mathbb H\,\mathbb S\,\mathbb P (\mathsf{FiCA})=  \mathbb H\,\mathbb S\,\mathbb P (\mathsf{BLK}).$$
Then, the variety of clone algebras is the algebraic counterpart of $\omega$-clones, 
the class of block algebras is the algebraic counterpart of clones, and the $\omega$-clones
are algebraically generated by clones through direct products, subalgebras and 
homomorphic images.

\subsection{Point-relativized functional clone algebras}

Let $A$ be a set.  
We define an equivalence relation on $A^\omega$ as follows 
$$r\equiv s\ \text{iff}\ |\{i: r_i\neq s_i\}|<\omega.$$
Let $A^\omega_r=\{s\in A^\omega : s\equiv r\}$ be the equivalence class of $r$ and $\mathcal O_{A,r}^{(\omega)}$ be the set of all functions from $A^\omega_r$ to $A$. 


\begin{definition} Let $\mathbf A$ be a $\tau$-algebra and $r\in A^\omega$.
The algebra  $\mathbf O_{\mathbf A,r}^{(\omega)}=(\mathcal O_{A,r}^{(\omega)},\sigma^r, q^r_{n},\e_{i}^r)$, where, for every $s\in A^\omega_r$ and $\varphi,\psi_1,\dots,\psi_n\in \mathcal O_{A,r}^{(\omega)}$, 
\begin{itemize}
\item $\e_i^r(s)=s_i$;
\item $q^r_n(\varphi,\psi_1,\dots,\psi_n)(s)=\varphi(s[\psi_1(s),\dots, \psi_n(s)])$;
\item $\sigma^r(\psi_1,\dots,\psi_n)(s)=\sigma^\mathbf A(\psi_1(s),\dots, \psi_n(s))$ for every $\sigma\in\tau$ of arity $n$,
\end{itemize}
is called the \emph{full point-relativized functional clone algebra with value domain} $\mathbf A$ \emph{and thread} $r$.
\end{definition}

Notice that, if $r\equiv s$, then $\mathbf O_{A,r}^{(\omega)}=\mathbf O_{A,s}^{(\omega)}$.

\begin{lemma}\label{lem:fun2}
  The algebra  $\mathbf O_{\mathbf A,r}^{(\omega)}$ is a clone $\tau$-algebra.
\end{lemma}

\begin{definition} 
 A subalgebra of  $\mathbf O_{\mathbf A,r}^{(\omega)}$ is called a   \emph{point-relativized functional clone algebra} with value domain $\mathbf A$ and thread $r$.
\end{definition}
 
 The class of point-relativized functional clone algebras is denoted by $\mathsf{RCA}$.  $\mathsf{RCA}_\tau$ is the class of $\mathsf{RCA}$s  whose value domain is a $\tau$-algebra.

We introduce the following maps.
 \begin{itemize}
\item If $B\in \mathcal B_A$ is a block, then \emph{the $r$-relativized top extension $B^\top_r:A^\omega_r \to A$ of $B$} is defined by $B^\top_r(s)=B^{(n)}(s_1,\dots,s_n)$, for every $s\in A^\omega_r$ and $n$ greater than the arity of $B$.
\item \emph{The $r$-relativized $n$-ary restriction $\psi_{n,r}:A^n\to A$ of $\psi\in \mathcal O_{ A,r}^{(\omega)}$} is defined as follows: 
$$\psi_{n,r}(a_1,\dots,a_n)= \psi(r[a_1,\dots,a_n])\ \text{for all $a_1,\dots,a_n\in A$}.$$
\end{itemize}

An analogous of Lemma \ref{lem:independent}, relating the algebraic and functional notions of independence, holds for $\mathsf{RCA}$s. 

The following lemma, which is true in $\mathcal O_{ A,r}^{(\omega)}$ and false in $\mathcal O_{ A}^{(\omega)}$ (see Lemma \ref{lem:semiconstant}), explains well the difference between $\mathsf{RCA}$s and $\mathsf{FCA}$s.

\begin{lemma} \label{lem:toptop}
Let $\varphi\in \mathcal O_{ A,r}^{(\omega)}$. Then the following conditions are equivalent:
\begin{itemize}
\item[(i)] $\varphi=B^\top_r$ for some block $B$;
\item[(ii)] $\varphi$ is finite dimensional in the clone algebra $\mathbf O_{ A,r}^{(\omega)}$.
\end{itemize} 
\end{lemma}

\begin{proof} (i) $\Rightarrow$ (ii) If $B$ has arity $k$ then, for every $s\in A^\omega_r$ and $n\geq k$, we have
$$\varphi(s)= B^\top_r(s)=B^{(n)}(s_1,\dots,s_n)=B^\top_r(r[s_1,\dots,s_n]) =\varphi(r[s_1,\dots,s_n]).$$ 
We now prove that 
$\varphi$ is independent of $\e_n$ for every $n> k$. Let
$s,u\in A^\omega_r$ such that $s_i=u_i$ for every $i\neq n$.
Then  $\varphi(s)=\varphi(r[s_1,\dots,s_k]) = \varphi(r[u_1,\dots,u_k]) = \varphi(u)$. Then by Lemma \ref{lem:independent} $\varphi$ is independent of $\e_n$ for every $n>k$. In other words, $\varphi$ is finite dimensional.
 
 (ii) $\Rightarrow$ (i) Let $n$ be the dimension of $\varphi$ and $s\in A^\omega_r$. We have to show that  $\varphi(s)=(\varphi_{n,r})^\top_r(s)$. 
 Let $k$ be  minimal such that $s=r[s_1,\dots,s_k]$. If $k \leq n$ then $s_i=r_i$ for all $k+1\leq i\leq n$ and $s=r[s_1,\dots,s_k]= r[s_1,\dots,s_n]$. Hence, 
 $(\varphi_{n,r})^\top_r(s)= \varphi_{n,r}(s_1,\dots,s_n) = \varphi(r[s_1,\dots,s_n]) =\varphi(s)$.
 If $k> n$, then $(\varphi_{n,r})^\top_r(s)= \varphi_{n,r}(s_1,\dots,s_n) = \varphi(r[s_1,\dots,s_n]) =\varphi(r[s_1,\dots,s_n,s_{n+1}])=\dots=\varphi(r[s_1,\dots,s_n,s_{n+1},\dots,s_k])= \varphi(s)$, because $\varphi$ is independent of $\e_{n+1},\dots,\e_k$.
\end{proof}


\subsection{The main theorem}
We recall that  $\mathsf{CA}$ is the class of all clone algebras, $\mathsf{RCA}$ is the class of all point-relativized functional clone algebras, $ \mathsf{FCA}$ is the class of all functional clone algebras,  $\mathsf{FiCA}$ is the class of all finite dimensional clone algebras, and $\mathsf{BLK}$ is the class of all block algebras.

\begin{theorem}\label{thm:main}
 $\mathsf{CA} = \mathbb I\, \mathsf{RCA}=\mathbb I\, \mathsf{FCA}=\mathbb{HSP}(\mathsf{FiCA})=\mathbb{HSP}(\mathsf{BLK})$.
\end{theorem}

 
 The proof of the main theorem is divided into lemmas.
 
 \begin{lemma}\label{lem:ca=rca}
 $\mathsf{CA}_\tau = \mathbb I\, \mathsf{RCA}_\tau$.
\end{lemma}

\begin{proof} Let $\mathbf C=(\mathbf C_\tau,q^\mathbf C_n,\e^\mathbf C_i)$ 
be a clone $\tau$-algebra.
Let $\mathbf O_{ \mathbf C_\tau,\epsilon}^{(\omega)}$    be the full $\mathsf{RCA}$ with value domain $\mathbf C_\tau$ and thread $\epsilon$, where  $\epsilon_i=\e^\mathbf C_i$ for every $i$. We define a map $F:C\to  \mathcal O_{C,\epsilon}^{(\omega)}$ as follows:
$$F(c)(s)=q_k^\mathbf C(c, s_1,\dots,s_k),\ \text{for every $s\in C^\omega_\epsilon$ such that $s= \epsilon[s_1,\dots,s_k]$ and $s_k\neq \e_k^\mathbf C$.}$$
Notice that by (C4) $F(c)(s)=q_n^\mathbf C(c, s_1,\dots,s_k,\e_{k+1}^\mathbf C,\dots,\e_n^\mathbf C)$ for every $n\geq k$.
$F$ is injective because $F(c)(\epsilon)=q_0^\mathbf C(c)=c$ for every $c\in C$. 
We prove that $F$ embeds $\mathbf C$ into $\mathbf O_{\mathbf C_\tau,\epsilon}^{(\omega)}$.  Let $\mathbf a=s_1,\dots,s_k$. 

If $n\geq k$ then we have:
$$
\begin{array}{rlll}
  F(q^\mathbf C_n(b,\mathbf c))(s) 
 & =  &  q_k^\mathbf C(q^\mathbf C_n(b,\mathbf c), \mathbf a) &\text{Def. $F$}\\
  & =  & q_n^\mathbf C(b,q_k^\mathbf C(c_1,\mathbf a),\dots,q_k^\mathbf C(c_n,\mathbf a))&\text{Lemma \ref{allungo}(ii)} \\
&=& q_n^\mathbf C(b,F(c_1)(s),\dots,F(c_n)(s))&\text{Def. $F$}\\
  &  = & F(b)(\epsilon[F(c_1)(s),\dots,F(c_n)(s)])  &\text{Def. $F$}\\
    &  = & F(b)(s[F(c_1)(s),\dots,F(c_n)(s)])  &\text{by $n\geq k$}\\
  &  = & q^\epsilon_n(F(b),F(c_1),\dots,F(c_n))(s).&
 \end{array}
$$
A similar proof works for $n< k$. We  conclude the proof as follows:
$
F(\sigma^\mathbf C(\mathbf b))(s) 
 =   q^\mathbf C_k(\sigma^\mathbf C(\mathbf b),\mathbf a)
  =  \sigma^\mathbf C(q^\mathbf C_k(b_1,\mathbf a),\dots,q^\mathbf C_k(b_n,\mathbf a))  
 =  \sigma^\mathbf C(F(b_1)(s),\dots,F(b_n)(s))  
  =   \sigma^\epsilon(F(b_1),\dots,F(b_n))(s). 
$
\end{proof}

By the $n$-reduct of a clone $\tau$-algebra $\mathbf C=(\mathbf C_\tau,q^\mathbf C_n,\e^\mathbf C_i)_{n\geq 0,i\geq 1}$ we mean the algebra 
$$\mathrm{Rd}_n \mathbf C:=(\mathbf C_\tau,q^\mathbf C_0,\dots,q^\mathbf C_n,\e^\mathbf C_1,\dots,\e^\mathbf C_n).$$

\begin{lemma}\label{lem:red} Let $\mathbf A$ be a $\tau$-algebra and
 $\mathbf D$ be a $\mathsf{RCA}_\tau$ with value domain $\mathbf A$ and thread $r$. For every $n> 0$ 
 the map $F_{r,n}: D \to \mathcal B_{A}^\top$, defined by
 $$F_{r,n}(\varphi)(s)=\varphi(r[s_1,\dots,s_n])\qquad\text{for every $\varphi\in D$ and $s\in A^\omega$},$$
 is a homomorphism of $\mathrm{Rd}_n \mathbf D$ into the $n$-reduct of the full block algebra
 $\mathbf B_{\mathbf A}^\top$.
\end{lemma}

\begin{proof}
 Let $k\leq n$,  $s\in A^\omega$ and $u=r[s_1,\dots,s_n]$. Let $F= F_{r,n}$ in this proof. 
 \[
\begin{array}{rll}
F(q^r_k(\varphi,\psi_1,\dots,\psi_k))(s)   & =  & q^r_k(\varphi,\psi_1,\dots,\psi_k)(u)\\
  &= & \varphi(u[\psi_1(u),\dots,\psi_k(u)])  \\
  &  = &   \varphi(r[\psi_1(u),\dots,\psi_k(u),s_{k+1},\dots,s_n])\\
    &=&F(\varphi)(s[\psi_1(u),\dots,\psi_k(u),s_{k+1},\dots,s_n])\\
   &=&F(\varphi)(s[F(\psi_1)(s),\dots,F(\psi_k)(s),s_{k+1},\dots,s_n]) \\
   &=&F(\varphi)(s[F(\psi_1)(s),\dots,F(\psi_k)(s)]) \\
 &=& q_k^\omega (F(\varphi),F(\psi_1),\dots,F(\psi_k))(s)\\
 &&\\
F(\sigma^r(\psi_1,\dots,\psi_k))(s)   & =  & \sigma^r(\psi_1,\dots,\psi_k)(u)\\
  &= & \sigma^\mathbf A(\psi_1(u),\dots,\psi_k(u))  \\
  &= & \sigma^\mathbf A(F(\psi_1)(s),\dots,F(\psi_k)(s))  \\  
  &  = &  \sigma^\omega(F(\psi_1),\dots,F(\psi_k))(s). \\
  \end{array}
\]
\end{proof}

\begin{lemma}\label{lem:cafica} $\mathsf{CA}_\tau =\mathbb{HSP}(\mathsf{FiCA}_\tau)$.
\end{lemma}

\begin{proof}
 Let $t(v_1,\dots,v_k)=u(v_1,\dots,v_k)$ be an identity (in the language of clone $\tau$-algebras) satisfied by every finite dimensional clone $\tau$-algebra.
 We now show that the identity $t=u$ holds in every clone $\tau$-algebra. Since $\mathsf{CA}_\tau =\mathbb I\,\mathsf{RCA}_\tau$
 it is sufficient to prove that the identity $t=u$ holds in the full $\mathsf{RCA}_\tau$ $\mathbf O_{ \mathbf A,r}^{(\omega)}$ with value domain  $\mathbf A$ and thread $r$. If $t^r$ and $u^r$ are the interpretation of $t$ and $u$ in $\mathbf O_{ \mathbf A,r}^{(\omega)}$, we have to show that
  $$t^r(\varphi_1,\dots, \varphi_k)(s)=u^r(\varphi_1,\dots, \varphi_k)(s)$$
for all $\varphi_1,\dots,\varphi_k\in \mathcal O_{ A,r}^{(\omega)}$ and all $s \in A^\omega_r$.
Let $n>k$ such that $q_m$ and $\e_m$ do not occur in $t,u$ for every $m> n$. Since $\mathbf O_{ \mathbf A,r}^{(\omega)}$ satisfies the equation $t=u$ iff the $n$-reduct $\mathrm{Rd}_n\, \mathbf O_{ \mathbf A,r}^{(\omega)}$ satisfies it, then
we can use the function $F_{s,n}$  defined in Lemma \ref{lem:red}. Let $t^\omega$ and $u^\omega$ be the interpretation of $t$ and $u$ in the full block algebra $\mathbf B_{\mathbf A}^\top$. Recall that  $\mathbf B_{\mathbf A}^\top$ is a finite-dimensional subalgebra of the full $\mathsf{FCA}$ $\mathbf O_{ \mathbf A}^{(\omega)}$.
\[
\begin{array}{lll}
t^r(\varphi_1,\dots, \varphi_k)(s)  & =  & t^r(\varphi_1,\dots, \varphi_k)(s[s_1,\dots,s_n])  \\
  & =  &F_{s,n}(t^r(\varphi_1,\dots, \varphi_k))(s)   \\
  &  = &  t^\omega(F_{s,n}(\varphi_1),\dots, F_{s,n}(\varphi_k))(s) \\ 
  &  = &  u^\omega(F_{s,n}(\varphi_1),\dots, F_{s,n}(\varphi_k))(s) \\ 
    & =  &F_{s,n}(u^r(\varphi_1,\dots, \varphi_k))(s)   \\
 & =  & u^r(\varphi_1,\dots, \varphi_k)(s)
\end{array}
\]
because the image of $F_{s,n}$ is a finite dimensional $\mathsf{FCA}$.
\end{proof}

Then, the following corollary is a consequence of  Theorem \ref{thm:fica}.

\begin{corollary}\label{cor:cablk}
 $\mathsf{CA}_\tau =\mathbb{HSP}(\mathsf{BLK}_\tau)$.
\end{corollary}

The remaining part of the proof of Theorem \ref{thm:main} is technical and it is postponed in Appendix.

\section{A characterisation of the lattices of equational theories}\label{sec:eqth}
In this section we propose a possible answer to the lattice of equational theories problem described in Section \ref{sec:pre:leq}. We prove that a lattice is isomorphic to a lattice of equational theories if and only if it is isomorphic to the lattice of all congruences of a finite dimensional clone algebra. Unlike Newrly's and Nurakunov's approaches \cite{newrly,nur}, we have an equational axiomatisation of the variety generated by  the class of finite dimensional clone algebras (see Theorem \ref{thm:main} and  Section \ref{sec:pre:leq}).

The main steps of the proof are the following:
\begin{itemize}
\item Given a variety of $\tau$-algebras axiomatised by the equational theory $T$, we turn its free algebra into a finite dimensional clone $\tau$-algebra, whose lattice of congruences is isomorphic to the lattice of equational theories extending $T$.
\item Given a finite dimensional clone $\tau$-algebra $\mathbf C$, we build a new algebraic type $\rho_\mathsf C$ and  a variety $\mathcal V$ of $\rho_\mathsf C$-algebras such that the congruence lattice of $\mathbf C$ is isomorphic to the lattice of equational theories extending the equational theory of   $\mathcal V$.
\end{itemize}
We conclude the section by showing that a lattice is isomorphic to a lattice of subclones if and only if it is isomorphic to the lattice of all subalgebras of a finite dimensional clone algebra.

\bigskip

It is well known that any lattice of equational theories is isomorphic to a congruence lattice. 
Let $T$ be an equational theory and $\mathcal V$ be the variety axiomatised by $T$. The lattice $L(T)$ of all equational theories extending $T$ is isomorphic 
to the lattice of all fully invariant congruences of the free algebra $\mathbf{F}_\mathcal{V}$ over a countable set $I=\{v_1,v_2,\ldots,v_n,\ldots\}$ of generators.

We say that an endomorphism $f$ of the free algebra $\mathbf{F}_\mathcal{V}$ is \emph{$n$-finite} if $f(v_i)=v_i$ for every $i>n$. An endomorphism is finite if it is $n$-finite for some $n$. 

\begin{lemma}\label{lem:end} \cite{BS,mac} Let $\mathcal V$ be a variety axiomatised by $T$. Then the  lattice $L(T)$ of all equational theories extending $T$ is isomorphic to
the congruence lattice of the algebra $(\mathbf{F}_\mathcal{V},f)_{f\in \mathrm{End}}$, which is an expansion of the free algebra $\mathbf{F}_\mathcal{V}$ by the set End of all its finite endomorphisms.
\end{lemma}

The set of all $n$-finite endomorphisms can be collectively expressed by an $(n+1)$-ary operation $q_n^\mathbf F$ on $\mathbf{F}_\mathcal{V}$ (see Example \ref{exa:free}):
 \begin{equation}\label{eq:qqq} q_n^\mathbf F(a,b_1,\dots,b_n)=s(a),\quad\text{for every
$a,b_1,\dots,b_n\in F_\mathcal{V}$,}\end{equation}
where $s$ is the unique $n$-finite endomorphism
of $\mathbf{F}_\mathcal{V}$ which sends the generator $ v_i$ to $b_i$ ($1\leq i\leq n$). 


\begin{definition}\label{def:cva}
 Let $\mathcal V$ be a variety and $\mathbf{F}_\mathcal{V}$ be the free $\mathcal V$-algebra over a countable set $I$ of generators.
 Then the algebra  $\mathbf{Cl}(\mathcal{V})=(\mathbf{F}_\mathcal{V}, q_n^\mathbf{F}, \e_i^\mathbf{F})$,
where $\e_i^\mathbf{F}= v_i\in I$ and $q_n^\mathbf{F}$ is defined in (\ref{eq:qqq}),
is  called \emph{the clone  $\mathcal V$-algebra}.
\end{definition}
 
\begin{proposition}\label{prop:free} Let $\mathcal V$ be a variety of $\tau$-algebras axiomatised by the equational theory $T$. Then we have:
\begin{enumerate}
\item   The clone  $\mathcal V$-algebra  $\mathbf{Cl}(\mathcal{V})$
is a finite dimensional clone $\tau$-algebra.
\item  The congruences lattice $\mathrm{Con}\,\mathbf{Cl}(\mathcal{V})$   is isomorphic to the lattice of equational theories $L(T)$.
\item If $w\in F_\mathcal{V}$ has dimension $n>0$ in $\mathbf{Cl}(\mathcal{V})$, then there exists a $\tau$-term $t(v_1,\dots,v_n)$ belonging to $w$. 
\item If $w\in F_\mathcal{V}$ has dimension $0$ in $\mathbf{Cl}(\mathcal{V})$, then there exists a $\tau$-term $t(v_1)\in w$ such that $\mathcal V \models t(v_1)=t(v_2)$.
\item If $\mathrm{Clo}\,\mathbf{F}_\mathcal{V}$ is the clone of term operations of $\mathbf{F}_\mathcal{V}$, then the clone  $\mathcal V$-algebra  $\mathbf{Cl}(\mathcal{V})$ is isomorphic to the block algebra $(\mathrm{Clo}\,\mathbf{F}_\mathcal{V})^\top$ (see Section \ref{sec:to} and Lemma \ref{prop:cloneblock}). 
\end{enumerate}
\end{proposition}

\begin{proof} (1) is straightforward.

(2) By Lemma \ref{lem:end} and the definition of $q_n^\mathbf{F}$.

 (3) Let $w\in F_\mathcal{V}$ of dimension $n>0$ and let $u\in w$ be an arbitrary term. Let $v_k$ be the last variable occurring in $u$ (i.e., $v_i$ does not occur in $u$ for every $i>k$). If $k \leq n$, then $u$ satisfies the required properties. Let $k> n$. Since $q_k^\mathbf{F}(w,\e_1^\mathbf{F},\dots,\e_n^\mathbf{F},\e_{1}^\mathbf{F},\dots, \e_1^\mathbf{F})=w$ is the equivalence class of the term $u[v_1/v_{n+1},v_1/v_{n+2}\dots,v_1/v_k ]$, then this last term belongs to $w$ and satisfies the required properties.
 
 (4) Let $w\in F_\mathcal{V}$ be zero-dimensional and $u\in w$ be an arbitrary term. If $u$ is ground, then $u=u(v_1)$ and we are done. Otherwise, we follow the reasoning in item (3).
 
(5) If $t(v_1,\dots,v_n)$ is a $\tau$-term and $\mathbf A\in\mathcal V$ is a $\tau$-algebra, then the set $T_t^\mathbf A$ (defined in Section \ref{sec:to}) of the term operations determined by $t$ is a block.
Notice that the arity of the block $T_t^\mathbf A$ may be less than $n$.  If $\mathcal V\models t_1=t_2$, then $T_{t_1}^\mathbf A = T_{t_2}^\mathbf A$ for every $\mathbf A\in\mathcal V$.
For the free algebra $\mathbf{F}_\mathcal{V}$, we have that $\mathcal V\models t_1=t_2$ iff $T_{t_1}^\mathbf A = T_{t_2}^\mathbf A$ iff $(T_{t_1}^\mathbf A)^\top = (T_{t_2}^\mathbf A)^\top$. It easily follows that $\mathbf{Cl}(\mathcal{V})$ is isomorphic to the block algebra $(\mathrm{Clo}\,\mathbf{F}_\mathcal{V})^\top$.
\end{proof}

\begin{definition}\label{def:ctype}
 Let $\mathbf C$ be a  clone algebra and  $R_\mathbf C$ be the clone of all $\mathbf C$-representable functions described in Definition \ref{def:representable}. 
\begin{enumerate}
\item The \emph{$\mathbf C$-type} is the algebraic type
$\rho_\mathbf C =  \{\overline f: f \in R_\mathbf C\}$, 
where  the operation symbol $\overline f$ has  arity $k$ if  $f\in R_\mathbf C$ is a $k$-ary function. 
\item The $\rho_\mathbf C$-algebra $\mathbf R_\mathbf C=(C,  f)_{f\in R_\mathbf C}$ is called  the \emph{algebra of $\mathbf C$-representable functions};
\item The algebra $\overline{\mathbf R}_\mathbf C=(\mathbf R_\mathbf C,q_n^\mathbf C,\e_i^\mathbf C)$ is called the \emph{clone $\rho_\mathbf C$-algebra of $\mathbf C$-representable functions}.
\end{enumerate}

\end{definition}

%
%
%

\begin{theorem}\label{thm:ch} Let $\mathbf C$ be a finite dimensional clone algebra.
Then we have:
\begin{itemize}
\item[(i)]  $\mathbf R_\mathbf C$ is isomorphic to the free $\rho_\mathsf C$-algebra over a countable set of generators in the variety $\mathrm{Var}(\mathbf R_\mathbf C)$;
\item[(ii)] $\overline{\mathbf R}_\mathbf C$ is isomorphic to the clone $\mathrm{Var}(\mathbf R_\mathbf C)$-algebra.
\end{itemize}
\end{theorem}

\begin{proof}  
We show that $\mathbf R_\mathbf C$ is the free algebra  over a countable set $\{\e_1^\mathbf C,\dots,\e_n^\mathbf C,\dots\}$ of generators in the variety $\mathrm{Var}(\mathbf R_\mathbf C)$. 
 Let $\mathbf A\in  \mathrm{Var}(\mathbf R_\mathbf C)$, $g: \{\e_1^\mathbf C,\dots,\e_n^\mathbf C,\dots\}\to A$ be an arbitrary map, and $d_i=g(\e_i^\mathbf C)$.
 We extend $g$ to a map $g^*: C\to A$ as follows. Let $b\in C$ of dimension $k$. 
 By Proposition \ref{prop:rr} the set $R(b)=\bigcup_{n\in \omega} \{f\in R_\mathbf C^{(n)}: b=f(\e_1^\mathbf C,\dots,\e_n^\mathbf C)\}$ is a block of arity $k$. For every $m\geq k$, we denote by $f^{m}_b:C^m\to C$  the unique function of arity $m$ belonging to $R(b)$. The function $f^{m}_b$ is defined as follows: $f^{m}_b(c_1,\dots,c_m)=q_m^\mathbf C(b,c_1,\dots,c_m)$ for every $c_1,\dots,c_m\in C$. Since $f^{m}_b$ is $\mathbf C$-representable, then $\overline{f^{m}_b}\in \rho_\mathbf C$ for every $m\geq k$ and we define 
 $$g^*(b)=\overline{f^{k}_b}^\mathbf A(d_1,\dots,d_{k}).$$
Since $\mathbf C\models f^m_b(x_1,\dots,x_k,x_{k+1},\dots, x_m)= f^k_b(x_1,\dots,x_k)$ for every $m\geq k$ and $\mathbf A\in \mathrm{Var}(\mathbf R_\mathbf C)$, then we have   
$$g^*(b) =\overline{f^m_b}^\mathbf A (d_1,\dots,d_m)\quad \text{for every $m\geq k$}.$$ 
We now show that $g^*$ is a homomorphism of $\rho_\mathbf C$-algebras. 
Let $\overline h\in \rho_\mathbf C$ of arity $n$, $\mathbf b=b_1,\dots,b_n\in C$ and $\mathbf e=\e_1,\dots,\e_n$. 
Let $m\geq n$ be a natural number greater than the maximal number among the dimensions of the elements $b_1,\dots,b_n,q^\mathbf C_n(h(\mathbf e),\mathbf b)$.
Let $\mathbf d=d_1,\dots,d_m$ and $\mathbf o=\e_1,\dots,\e_m$. 
We now show that 
$$g^*( h(\mathbf b))=\overline h^\mathbf A(g^*(b_1),\dots,g^*(b_n)).$$
Recalling that $h(\mathbf b)=q^\mathbf C_n(h(\mathbf e),\mathbf b)$ (see Definition \ref{def:representable}), then we have:
\begin{itemize}
\item $\overline h^\mathbf A(g^*(b_1),\dots,g^*(b_n))=
\overline h^\mathbf A(\overline{f^m_{b_1}}^\mathbf A(\mathbf d),\dots,\overline{f^m_{b_n}}^\mathbf A(\mathbf d))$;
\item $ g^*(h(\mathbf b))
=g^*(q^\mathbf C_n(h(\mathbf e),\mathbf b))= \overline{f^m_r}^\mathbf A(\mathbf d)$, where $r= q^\mathbf C_n(h(\mathbf e),\mathbf b)$.

\end{itemize}

We get that $g^*$ is a homomorphism if  the algebra $\mathbf A$ satisfies the identity 
$$\overline h(\overline{f^m_{b_1}} (x_1,\dots,x_m),\dots,\overline{f^m_{b_n}} (x_1,\dots,x_m))=
\overline{f^m_r}(x_1,\dots,x_m),\quad
\text{where $r= q^\mathbf C_n(h(\mathbf e),\mathbf b)$}.$$
Since $\mathbf A\in\mathrm{Var}(\mathbf R_\mathbf C)$, then it is sufficient to prove that the algebra $\mathbf R_\mathbf C$ satisfies the above identity. By putting $\mathbf x= x_1,\dots,x_m$ the conclusion follows from Lemma \ref{12}(ii):
$$f^m_r(\mathbf x)=q_m^\mathbf C(q^\mathbf C_n(h(\mathbf e),\mathbf b),\mathbf x)= q_m^\mathbf C(h(\mathbf b),\mathbf x)=_{Lem. \ref{12}}
 h(q_m^\mathbf C(b_1,\mathbf x),\dots,q_m^\mathbf C(b_n,\mathbf x))=h(f^m_{b_1}(\mathbf x),\dots,f^m_{b_n} (\mathbf x)).$$ 
It remains to show that the operation $p(x)=q_n^\mathbf C(x,\mathbf b)$ is the unique $n$-finite endomorphism
of the free algebra $\mathbf R_\mathbf C$ which sends $\e_i$ to $b_i$ ($1\leq i\leq n$). This again follows from Lemma \ref{12}(ii) because
$p(h(\mathbf a))=q_n^\mathbf C(h(\mathbf a),\mathbf b) = h(q_n^\mathbf C(a_1,\mathbf b),\dots,q_n^\mathbf C(a_k,\mathbf b))=h(p(a_1),\dots,p(a_k))$ for every $\overline h\in\rho_\mathbf C$ of arity $k$.
\end{proof}



\begin{theorem}\label{thm:BM} A lattice $L$ is isomorphic to a lattice of equational theories if and only if $L$ is isomorphic to the congruence lattice of a finite-dimensional $\mathsf{CA}$. 
  \end{theorem}

\begin{proof} ($\Rightarrow$) It follows from Proposition \ref{prop:free}. 

($\Leftarrow$) Let $\mathbf C$  be a finite dimensional clone algebra and $\overline{\mathbf R}_\mathbf C$ be the clone $\rho_\mathbf C$-algebra of $\mathbf C$-representable functions. Since $\mathbf C$ and $\overline{\mathbf R}_\mathbf C$ have the same pure reduct, then by Corollary \ref{cor:taunu} we have $\mathrm{Con}\,\mathbf C=\mathrm{Con}\,\overline{\mathbf R}_\mathbf C$. The conclusion of the theorem follows from Theorem \ref{thm:ch}(ii) and Proposition \ref{prop:free}(2), because $\overline{\mathbf R}_\mathbf C$ is isomorphic to the clone $\mathrm{Var}(\mathbf R_\mathbf C)$-algebra. 
%
\end{proof}

Recalling  that every finite dimensional $\mathsf{CA}$ is isomorphic to a block algebra (see Theorem \ref{thm:fica}), in this corollary we relate lattices of equational theories and clones. 

\begin{corollary}
  A lattice $L$ is isomorphic to a lattice of equational theories if and only if $L$ is isomorphic to the lattice of all congruences of a block algebra. 
\end{corollary}

We conclude this section by characterising the lattices of subclones.

Let $A$ be a set and $F$ be a clone on $A$. A subset $G\subseteq F$ is called a \emph{subclone} of $F$ if $G$ is a clone on $A$. For example, every clone on $A$ is a subclone of $\mathcal O_A$.

We denote by $\mathrm{Sb}(F)$ the lattice of all subclones of a clone $F$. We say that a lattice $L$ is  isomorphic to a lattice of subclones if there exists a set $A$ and a clone $F$ on $A$ such that $L$ is isomorphic to the lattice $\mathrm{Sb}(F)$.

\begin{proposition}
   A lattice $L$ is isomorphic to a lattice of subclones if and only if $L$ is isomorphic to the lattice of subalgebras of a block algebra. 
\end{proposition}

\begin{proof}
 By Proposition \ref{prop:cloneblock} and Corollary \ref{cor:cloneblock}.
\end{proof}

\section{The category of  varieties}\label{sec:allvar}
Important properties of a variety $\mathcal V$ depend on the pure reduct of the clone $\mathcal V$-algebra $\mathbf{Cl}(\mathcal V)$ associated with its free algebra. However, not every clone $\tau$-algebra is the clone $\mathcal V$-algebra associated with the free algebra of a variety $\mathcal V$ of type $\tau$.
In this section, after characterising central elements in clone algebras, we introduce minimal clone algebras and prove that 
a clone $\tau$-algebra $\mathbf C$ is minimal if and only if $\mathbf C\cong \mathbf{Cl}(\mathcal V)$ for some variety $\mathcal V$ of type $\tau$. 
We also introduce the category $\mathcal{CA}$ of all clone algebras (of arbitrary similarity types) with pure homomorphisms (i.e., preserving only the nullary operators $\e_i$ and the operators $q_n$) as arrows and we show that the category $\mathcal{CA}$ is equivalent both to the full subcategory $\mathcal{MCA}$ of minimal clone algebras and, more to the point, to the variety $\mathsf{CA}_0$ of pure clone algebras.
We prove that $\mathcal{MCA}$ is categorically isomorphic to the category $\mathcal{VAR}$ of all varieties, so that we can use the more manageable category $\mathcal{MCA}$ of minimal clone algebras and pure homomorphisms to study the category $\mathcal{VAR}$. 
We conclude the section by showing  that the category $\mathcal{MCA}$ is closed under categorical product and use this result and central elements to provide a generalisation of the theorem on independent varieties presented by Gr\"atzer et al. in \cite{GLP}.

\subsection{Central elements in clone algebras}\label{sec:central}
Every clone algebra is an $n$CH, for every $n$. Therefore, there exists a bijection between the set of $n$-central elements of a clone algebra and the set of its $n$-tuples of complementary factor congruences.
In this section we show the following results characterising the central elements of a clone algebra $\mathbf C$:
\begin{itemize}
\item[(i)]  Every $n$-central element of $\mathbf C$ is finite dimensional;
\item[(ii)] An element $c\in C$ is $n$-central if and only if $n\geq \gamma(c)$ and $c$ is $m$-central for every $m\geq \gamma(c)$;
\item[(iii)] The set $\{c: \text{$c$ is $n$-central for some $n$}\}$ of all central elements of ${\bf C}$ is a pure clone subalgebra of the pure reduct $\mathbf C_0$ of $\mathbf C$.
\item[(iv)] An element is $n$-central in $\mathbf C$ iff it is $n$-central in the pure reduct $\mathbf C_0$ of $\mathbf C$.
\end{itemize}
As a consequence of (iv),  the decomposability of a variety $\mathcal V$ as product of other varieties only depends on the pure reduct  $\mathbf{Cl}(\mathcal V)_0$ of the clone $\mathcal V$-algebra $\mathbf{Cl}(\mathcal V)$ (see Section \ref{sec:categ}).

\begin{lemma}\label{central1}
Let ${\bf C}$ be a clone algebra and $c\in C$ be  $n$-central, for some $n$. 
Then $c$ is finite dimensional and $\gamma(c)\leq n$.
\end{lemma}
\begin{proof}
By the way of contradiction, let us suppose that either $c$ is finite dimensional and $\gamma(c)>n$  or $\gamma(c)=\omega$. 
In both cases there exists $m>n$ such that $c$ is dependent on $\e_m$, meaning that $c\neq q_m(c,\e_1,\ldots,\e_{m-1},\e_{m+1})$.

Since $c$ is  $n$-central, the equation
\[\begin{array}{lllr}
&&q_n(c,q_m(g_1,h_1^1,\ldots,h_m^1),\ldots, q_m(g_n,h_1^n,\ldots,h_m^n)) &\\
&=& q_m(q_n(c,g_1,\ldots,g_n),q_n(c,h_1^1,\ldots,h_1^n),\ldots,q_n(c,h_m^1,\ldots,h_m^n))&\\
\end{array}
\]
holds for all $g_1,\ldots,g_n,h_1^1,\ldots h^1_m,\ldots, h^n_1,\ldots,  h_m^n$ in $C$.
By letting $g_i=\e_i$ for $1\leq i\leq n$ , $h_i^j=\e_i$ for $1\leq i\leq m-1$ and $1\leq j\leq n$, 
and $h_m^j=\e_{m+1}$ for $1\leq j\leq n$ in the equation above, and by exploiting again the fact that $c$ is $n$-central, we get
$q_n(c,\e_1,\ldots,\e_n)=
q_m(c,\e_1,\dots,\e_{m-1},\e_{m+1})$.

The left-hand side of the equation above being equal to $c$, we get a contradiction using our initial assumption that $c\neq q_m(c,\e_1,\ldots,\e_{m-1},\e_{m+1})$.
\end{proof}

\begin{lemma}\label{central2}
Let ${\bf C}$ be a clone $\tau$-algebra, $c\in C$ be  $n$-central  for some $n$,
and let $m\geq n$. Then $c$ is $m$-central.
\end{lemma}
\begin{proof}
By Lemma \ref{central1}, $\gamma(c)\leq n$ so that $c$ is independent of $\e_n$, $\e_{n+1},\ldots,\e_m$. The first equation characterising $m$-centrality, namely $q_m(c,\e_1,\ldots,\e_m)=c$ is
verified {\em a priori}. As for the second one: given $x\in C$, we have 
$x=q_n(c,x,\ldots,x)=q_m(c,x,\ldots,x)$, the second equality following from Lemma \ref{lem:ind1}, 
It remains to verify the third equation of $m$-centrality. We have:

\[\begin{array}{lllr}
 q_m(c,\sigma(x^1_1,\ldots,x^1_k),\ldots,\sigma(x^m_1,\ldots,x^m_k))
&=& q_n(c,\sigma(x^1_1,\ldots,x^1_k),\ldots,\sigma(x^n_1,\ldots,x^n_k))&\\ 
&=& \sigma(q_n(c,x^1_1,\ldots,x^n_1),\ldots,q_n(c,x^1_k,\ldots,x^n_k)) &\\ 
&=& \sigma(q_m(c,x^1_1,\ldots,x^m_1),\ldots,q_m(c,x^1_k,\ldots,x^m_k)) &\\ 
\end{array}
\]
and we are done.
\end{proof}

\begin{proposition}\label{prop:central}
Let ${\bf C}$ be a clone $\tau$-algebra and $c\in C$. If there exists $n$ such that $c$ is   $n$-central,
then, for all $m$,  $c$ is $m$-central if and only if $m\geq \gamma(c)$.
\end{proposition}
\begin{proof}
By Lemma \ref{central1} and Lemma \ref{central2}, it is enough to show that
$c$ is $\gamma(c)$-central. For the sake of readability, let $\gamma(c)=l$. 
For all $x,x^1_1,\ldots,x^1_k,\ldots,x^l_1,\ldots, x^l_k\in C$, and for all $k$-ary $\sigma$,
using repeatedly Lemma \ref{lem:ind1}, 
we have $q_l(c,x,\ldots,x)=q_n(c,x,\ldots,x)=x$, and 

\[\begin{array}{lllr}
  q_l(c,\sigma(x^1_1,\ldots,x^1_k),\ldots,\sigma(x^l_1,\ldots,x^l_k))
&=& q_n(c,\sigma(x^1_1,\ldots,x^1_k),\ldots,\sigma(x^n_1,\ldots,x^n_k))&\\ 
&=& \sigma(q_n(c,x^1_1,\ldots,x^n_1),\ldots,q_n(c,x^1_k,\ldots,x^n_k)) &\\ 
&=& \sigma(q_l(c,x^1_1,\ldots,x^l_1),\ldots,q_l(c,x^1_k,\ldots,x^l_k)) &\\ 
\end{array}
\]
\end{proof}

We denote by $\mathrm{Ce}_n(\mathbf C)$ the set of all $n$-central elements of $\mathbf C$, and by
$\mathrm{Ce}(\mathbf C)$ the set $\bigcup_{n\geq 1} \mathrm{Ce}_n(\mathbf C)$.
The algebra $(\mathrm{Ce}_n(\mathbf C),q_n^\mathbf C,\e_1^\mathbf C,\dots,\e_n^\mathbf C)$ is an $n\mathrm{BA}$ (see Section \ref{sec:nba}).

\begin{proposition} Let $\mathbf C$ be a clone $\tau$-algebra. Then $\mathrm{Ce}(\mathbf C)$ is a finite dimensional subalgebra of the pure reduct of $\mathbf C$.
\end{proposition}

\begin{proof} Let $a$ and $\mathbf b=b_1,\dots,b_n$ be elements of  $\mathrm{Ce}(\mathbf C)$. We show that $q_n(a,\mathbf b)$ is also central. By Lemma \ref{central1} $a,b_1,\dots,b_n$ are finite dimensional. Let $m\geq n$ be greater than the dimensions of $a,b_1,\dots,b_n$. Since by (C4) $q_n(a,\mathbf b)= q_m(a,\mathbf b,\e_{n+1},\dots,\e_m)$ and by Lemma \ref{central2} the elements $a,b_1,\dots,b_n, \e_{n+1},\dots,\e_m$ are $m$-central, 
then $q_m(a,\mathbf b,\e_{n+1},\dots,\e_m)$ is also $m$-central, because $\mathrm{Ce}_m(\mathbf C)$ is an $m\mathrm{BA}$.
\end{proof}

The variety generated by the class $\{ \mathrm{Ce}(\mathbf C) : \mathbf C\in\mathsf{CA}\}$ will be called the variety of pure $\omega$-Boolean-like algebras ($\omega\mathrm{BA}$, for short). We propose the problem of finding an equational axiomatisation for the variety of $\omega\mathrm{BA}$s.  

\bigskip

We conclude the section with the following useful result. 

\begin{proposition}\label{prop:de} Let $\mathbf C$ be a clone $\tau$-algebra and $c\in C$ of dimension $n$. Then  $c$ is $n$-central in $\mathbf C$ iff it is $n$-central in the pure reduct of $\mathbf C$.
\end{proposition}

\begin{proof}
 The conclusion follows because, for every $\sigma\in\tau$ of arity $k$, $\sigma^\mathbf C$ is defined in terms of $q_k^\mathbf C$ and the element $\sigma^\mathbf C(\e_1,\dots,\e_k)$:
 $\sigma^\mathbf C(a_1,\dots,a_k)=q_k^\mathbf C(\sigma^\mathbf C(\e_1,\dots,\e_k),a_1,\dots,a_k)$.
\end{proof}

 The meaning of the above proposition is that the decomposability of a variety as product of other varieties only depends on the pure clone algebraic structure of its free algebra.

\subsection{Minimal clone algebras}\label{sec:mca} 
Not every clone $\tau$-algebra is the clone $\mathcal V$-algebra associated with the free algebra of some variety $\mathcal V$ of type $\tau$ (see Theorem \ref{thm:ch}).
In this section we introduce minimal clone algebras and prove that 
a clone $\tau$-algebra $\mathbf C$ is minimal if and only if $\mathbf C\cong \mathbf{Cl}(\mathcal V)$ for some variety $\mathcal V$ of type $\tau$. 

Let $\mathbf C=(\mathbf C_\tau,q_n^\mathbf C,\e_i^\mathbf C)$ be a clone $\tau$-algebra. We denote by $M(\mathbf C)$ the minimal subalgebra of the algebra $(\mathbf C_\tau,\e_i^\mathbf C)_{i\geq 1}$. 
The algebra $M(\mathbf C)$ is an algebra in the type $\tau(\e)=\tau\cup\{\e_1,\dots,\e_n,\dots\}$. 
In Lemma \ref{lem:min} we show  that $M(\mathbf C)$ is also closed under the operators $q_n^\mathbf C$ and, as algebra in the type of clone $\tau$-algebras, it is the minimal subalgebra of $\mathbf C$.

A ground $\tau(\e)$-term is a term defined by the following grammar: $t,t_i::= \e_i\ |\ \sigma(t_1,\dots,t_k)$, where  
$\sigma\in \tau$. 

\begin{lemma}\label{lem:min}
 Let $\mathbf C$ be a clone $\tau$-algebra. Then the following conditions hold: 
 \begin{itemize}
 \item[(i)]   $b\in M(\mathbf C)$ if and only if $b=t^\mathbf C$ for some ground $\tau(\e)$-term $t$.
\item[(ii)]  $M(\mathbf C)$ is closed under the operators $q_n^\mathbf C$. 
\item[(iii)] The clone $\tau$-algebra $M(\mathbf C)$ is finite dimensional and it is the minimal subalgebra of $\mathbf C$.
\end{itemize}
\end{lemma}

\begin{proof} (i) Trivial.

 (ii) The proof is by induction over the complexity of the ground $\tau(\e)$-terms in the first argument of $q_n$.
 
(iii) By induction on the complexity of a ground $\tau(\e)$-term $t$, if $\e_k$ does not occur in $t$, then $t^\mathbf C$ is independent of $\e_k$. It follows that, if $t=t(\e_1,\dots,\e_n)$, then $t^\mathbf C$ has dimension $\leq n$.
\end{proof}

\begin{definition}
  We say that a clone $\tau$-algebra $\mathbf C$ is \emph{minimal} if $\mathbf C= M(\mathbf C)$.
\end{definition}

We remark that,  if $h:\mathbf C\to\mathbf D$ is an onto homomorphism of clone $\tau$-algebras and $\mathbf C$ is minimal, then $\mathbf D$ is minimal.

%
%
%

 The translation of the ground $\tau(\e)$-terms into $\tau$-terms in the variables $v_1,v_2,\dots,v_n,\dots$ is defined by
$\e_i^*=v_i$;  $\sigma(t_1,\dots,t_n)^*=\sigma(t_1^*,\dots,t_n^*)$.

\begin{theorem}\label{thm:ch2} Let $\mathbf C=(\mathbf C_\tau,q_n^\mathbf C,\e_i^\mathbf C)$ be a minimal  clone $\tau$-algebra,  $\mathrm{Var}(\mathbf C_\tau)$ be the variety of $\tau$-algebras generated by $\mathbf C_\tau$, and $ \mathrm{Var}(\mathbf C)$ be the variety of clone $\tau$-algebras generated by $\mathbf C$. Then,  
\begin{itemize}
\item[(i)] $\mathbf C_\tau$ is  the free algebra over a countable set of generators in the variety $ \mathrm{Var}(\mathbf C_\tau)$;  
\item[(ii)] $\mathbf C$ is the clone $ \mathrm{Var}(\mathbf C_\tau)$-algebra;
\item[(iii)] $\mathbf C$ is  the free algebra over an empty set of generators in the variety $ \mathrm{Var}(\mathbf C)$.
\end{itemize} 
\end{theorem}

\begin{proof} (i) We show that $\mathbf C_\tau$ is isomorphic to the free algebra  over a countable set $\{\e_1^\mathbf C,\dots,\e_n^\mathbf C,\dots\}$ of generators in the variety $\mathrm{Var}(\mathbf C_\tau)$. 
 Let $\mathbf A\in  \mathrm{Var}(\mathbf C_\tau)$, $g: \{\e_1^\mathbf C,\dots,\e_n^\mathbf C,\dots\}\to A$ be an arbitrary map, and $d_i=g(\e_i^\mathbf C)$.
 We extend $g$ to a map $g^*: C\to A$ as follows. Let $b\in C$ of dimension $k$ and let $m\geq k$.
 Since $\mathbf C$ is minimal, there exists a ground $\tau(\e)$-term $t$ such that $t^\mathbf C=b$.
 We define 
 $$g^*(b)=(t^*)^{\mathbf A,m}(d_1,\dots,d_m),$$
 where $t^*$ is a $\tau$-term and  $(t^*)^{\mathbf A,m}$ is the term operation defined in Section \ref{sec:to}. 
 The definition of $g^*(b)$ is independent of $m\geq k$.
We now show that $g^*$ is a homomorphism of $\tau$-algebras. 
Let $\sigma\in\tau$ of arity $n$, $\mathbf b=b_1,\dots,b_n\in C$ and $\mathbf e=\e_1,\dots,\e_n$. 
Let $m\geq n$ be a natural number greater than the maximal number among the dimensions of the elements $b_1,\dots,b_n,\sigma^\mathbf C(\mathbf b)$.
Let $\mathbf d=d_1,\dots,d_m$. 
We now show that 
$$g^*( \sigma^\mathbf C(\mathbf b))=\sigma^\mathbf A(g^*(b_1),\dots,g^*(b_n)).$$
If $b_i=t_i^{\mathbf C}$ for some ground $\tau(\e)$-term $t_i$ ($i=1,\dots,n$), then $\sigma^\mathbf C(\mathbf b)=\sigma(t_1,\dots,t_n)^\mathbf C$.
Recalling that $\sigma^\mathbf C(\mathbf b)=q^\mathbf C_n(\sigma^\mathbf C(\mathbf e),\mathbf b)$, then we have:
$$\sigma^\mathbf A(g^*(b_1),\dots,g^*(b_n))=
\sigma^\mathbf A((t_1^*)^{\mathbf A,m}(\mathbf d),\dots,(t_n^*)^{\mathbf A,m}(\mathbf d))=
(\sigma(t_1,\dots,t_n)^*)^{\mathbf A,m}(\mathbf d)=g^*(\sigma^\mathbf C(\mathbf b)).$$

(ii) By Lemma \ref{lem:24} the map $x \mapsto q_n^\mathbf C(x,\mathbf b)$ is the unique endomorphism
of the free algebra $\mathbf C_\tau$ which sends $\e_i$ to $b_i$ ($1\leq i\leq n$). 

(iii) Let $\mathbf A\in  \mathrm{Var}(\mathbf C)$. Then $\mathbf A_\tau\in  \mathrm{Var}(\mathbf C_\tau)$. By (i) there exists a unique homomorphism $f$ from   $\mathbf C_\tau$ into $\mathbf A_\tau$ such that $f(\e_i^\mathbf C)= \e_i^\mathbf A$.
Since $\mathbf C$ is minimal, then $f$ is onto $M(\mathbf A)$ and, for every ground $\tau(\e)$-term $t$, we have $f(t^\mathbf C)=t^\mathbf A$.
The proof that $f$ preserves $q_n$ is by induction over the complexity of the first argument of $q_n$.
\end{proof}

\begin{corollary} A clone $\tau$-algebra $\mathbf C$ is minimal if and only if $\mathbf C\cong \mathbf{Cl}(\mathcal V)$ for some variety $\mathcal V$ of type $\tau$. 
\end{corollary}

Let  $\mathcal V$ be a variety of $\tau$-algebras axiomatised by the equational theory $T$ and
 $\mathcal V^{cl}$ be the variety of clone $\tau$-algebras satisfying $T$. Since $\mathbf{Cl}(\mathcal{V})$ satisfies $T$, then $\mathrm{Var}(\mathbf{Cl}(\mathcal{V}))\subseteq \mathcal V^{cl}$. In the following proposition we show that the inclusion is sometimes strict.


\begin{proposition} 
\begin{itemize}
\item[(i)]  The clone $\mathcal V$-algebra $\mathbf{Cl}(\mathcal{V})$ is the free algebra over an empty set of generators in the variety $\mathcal V^{cl}$;
\item[(ii)] $\mathcal V^{cl}$ is not in general generated by $\mathbf{Cl}(\mathcal{V})$.
\end{itemize}
\end{proposition}

\begin{proof} (i) Let $\mathbf A=(\mathbf A_\tau,q_n^\mathbf A,\e_i^\mathbf A)\in \mathcal V^{cl}$.
Since $\mathbf A_\tau\in \mathcal V$, then there exists a unique homomorphism $f$ of $\tau$-algebras from $\mathbf{F}_\mathcal{V}$ into $\mathbf A_\tau$ such that $f(\e_i^\mathbf F)=f(v_i)= \e_i^\mathbf A$. The proof that $f$ preserves the operators $q_n^\mathbf F$ is similar to that of Theorem \ref{thm:ch2}(ii). 

(ii) If $\mathcal S$ is the class of all sets (i.e., the variety of all algebras in the empty type), then $\mathcal S^{cl}$ is the variety of all pure clone algebras. We show that $\mathbf{Cl}(\mathcal{S})$ does not genetate $\mathcal S^{cl}$.
 $\mathbf{Cl}(\mathcal{S})=(I,q_n^\mathbf I,\e_i^\mathbf I)$ has the set  $I=\{v_1,v_2,\dots,v_n,\dots\}$ as universe and $\e_i^\mathbf I=v_i$. The algebra $\mathbf I$ satisfies the identity
\begin{equation}\label{eq:identity}
q_n(y, q_n(y, x_{11},x_{12},\dots,x_{1n}),\dots,q_n(y,x_{n1},x_{n2},\dots,x_{nn}))=q_n(y,x_{11},\dots,x_{nn})
\end{equation}
 but $\mathcal S^{cl}$ does not satisfy it. Here is a counterexample.
 Let $2=\{0,1\}$ and $f:2^2\to 2$ be a function such that $f(0,0)=0$ and $f(0,1)=f(1,0)=f(1,1)=1$.
 Then $1=f(f(0,1),f(1,0))\neq f(0,0)$. Then the pure functional clone algebra of universe $\mathcal O_2^{(\omega)}$ does not satisfies identity (\ref{eq:identity}): 
  $$q_2^\omega(f^\top,q_2^\omega(f^\top,\e_1^\omega,\e_2^\omega), q_2^\omega(f,\e_2^\omega,\e_1^\omega))\neq q_2^\omega(f^\top,\e_1^\omega,\e_1^\omega).$$
Let $r\in 2^\omega$ such that $r_1=0$ and $r_2=1$. Then,
 \[
\begin{array}{lll}
 q_2^\omega(f^\top,q_2^\omega(f^\top,\e_1^\omega,\e_2^\omega), q_2^\omega(f^\top,\e_2^\omega,\e_1^\omega))(r)  
  & =  & f^\top(r[q_2^\omega(f^\top,\e_1^\omega,\e_2^\omega)(r),q_2^\omega(f^\top,\e_2^\omega,\e_1^\omega)(r)])  \\
  &  = & f^\top(r[f^\top(r[\e_1^\omega(r),\e_2^\omega(r)]),f^\top(r[\e_2^\omega(r),\e_1^\omega(r)])])  \\
  &  = & f^\top(r[f^\top(r[0,1]),f^\top(r[1,0])])  \\
   &  = & f^\top(r[f(0,1),f(1,0)])  \\ 
   &  = & f(f(0,1),f(1,0)) =1 \\ 
\end{array}
\]
while $q_2^\omega(f^\top,\e_1^\omega,\e_1^\omega)(r) = f(0,0)=0$.
\end{proof}

In Proposition \ref{exa:rc} below we compare on minimality a clone $\tau$-algebra $\mathbf C$ and the clone $\rho_\mathbf C$-algebra of its $\mathbf C$-representable functions. 

Let $\mathbf C=(\mathbf C_\tau,q_n^\mathbf C,\e_i^\mathbf C)$ be a finite-dimensional clone $\tau$-algebra. We recall that (i) For every $\tau$-term $t$,  $T_t^{\mathbf C_\tau}$ is the block of term operations determined by $t$ (see Section \ref{sec:to}  and  Example \ref{exa:to}); (ii) For every $a\in C$, $R(a)$ is the block of $\mathbf C$-representable functions determined by $a$; (iii) The set $R_\mathbf C$ of $\mathbf C$-representable functions is a clone (see Proposition \ref{prop:rr}).

\begin{lemma}\label{lem:ratt} Let $\mathbf C=(\mathbf C_\tau,q_n^\mathbf C,\e_i^\mathbf C)$ be a finite-dimensional  clone $\tau$-algebra. Then $a\in M(\mathbf C)$ if and only if
$R(a)=T_t^{\mathbf C_\tau}$ for some $\tau$-term $t$.
\end{lemma}

\begin{proof} ($\Rightarrow$) If $a\in M(\mathbf C)$ has dimension $n$, then by Lemma \ref{lem:min}   
$a=t^\mathbf C(\e_1^\mathbf C,\dots,\e_n^\mathbf C)$ for some $\tau(\e)$-term $t= t(\e_1,\dots,\e_n)$. Let $t^*= t(v_1,\dots,v_n)$ be the $\tau$-term translation of $t$. Since
  $$q_k^\mathbf C(a,b_1,\dots,b_k)= q_k^\mathbf C(t^\mathbf C(\e_1^\mathbf C,\dots,\e_n^\mathbf C),b_1,\dots,b_k)=_{(C6)}\dots=_{(C6,C1)} (t^*)^{\mathbf C,k}(b_1,\dots,b_k)$$
  for every $k\geq n$ and $b_1,\dots,b_k\in C$, then $R(a)=T_{t^*}^{\mathbf C_\tau}$.
  
  ($\Leftarrow$) If $R(a)=T_u^{\mathbf C_\tau}$ for some $\tau$-term $u=u(v_1,\dots,v_n)$, then $a=q_n^\mathbf C(a,\e_1^\mathbf C,\dots,\e_n^\mathbf C)= u^{\mathbf C_\tau}(\e_1^\mathbf C,\dots,e_n^\mathbf C)$. By Lemma \ref{lem:min} $a\in M(\mathbf C)$. 
\end{proof}

\begin{proposition}\label{exa:rc}
Let $\mathbf C=(\mathbf C_\tau,q_n^\mathbf C,\e_i^\mathbf C)$ be a finite-dimensional clone $\tau$-algebra and $\overline{\mathbf R}_\mathbf C=(\mathbf R_\mathbf C,q_n^\mathbf C,\e_i^\mathbf C)$ be the clone $\rho_\mathbf C$-algebra of $\mathbf C$-representable functions. Then the following conditions hold: 
\begin{itemize}
\item[(i)]  $\overline{\mathbf R}_\mathbf C$ is minimal.
\item[(ii)] $\mathbf C$ is minimal if and only if $R_\mathbf C = \mathrm{Clo}\,\mathbf C_\tau$.
\end{itemize}
\end{proposition}

\begin{proof}
(i) If $a\in C$ has dimension $n$, then the $n$-ary operation $f$ defined by $f(x_1,\dots,x_n)=q_n^\mathbf C(a,x_1,\dots,x_n)$ is $\mathbf C$-representable. Then $f$ is a basic operation of $\mathbf R_\mathbf C$ and $a=f(\e_1^\mathbf C,\dots,\e_n^\mathbf C)$.

(ii) First of all, we have $\mathrm{Clo}\,\mathbf C_\tau\subseteq R_\mathbf C$, because by Lemma \ref{cor:sigma} every basic operation of type $\tau$ is $\mathbf C$-representable. For the opposite inclusion it is sufficient to apply Lemma \ref{lem:ratt}. 
\end{proof}


\subsection{The category of clone algebras and pure homomorphisms}\label{sec:categ}
A map between clone algebras preserving the operators $q_n$ and the nullary operators $\e_i$ will be called a \emph{pure homomorphism}. This means that, given a clone $\tau$-algebra $\mathbf C$ and a clone $\nu$-algebra $\mathbf D$, a map $f:C\to D$ is a pure homomorphism from $\mathbf C$ into $\mathbf D$ if and only if $f$ is a  homomorphism from the pure reduct $\mathbf C_0$ of $\mathbf C$ into the pure reduct $\mathbf D_0$ of  $\mathbf D$.


In this section every variety $\mathcal V$ of $\tau$-algebras will be considered as a category whose objects are the algebras of $\mathcal V$ and whose arrows are the homomorphisms of $\tau$-algebras.

Let Type be the class of all algebraic types. The category $\mathcal{CA}$ has the class $\bigcup_{\tau\in \mathrm{Type}} \mathsf{CA}_\tau$ as class of objects and pure homomorphisms as arrows. We denote by $\mathcal{MCA}$ the full subcategory of  $\mathcal{CA}$ whose objects are the minimal clone algebras.  The variety $\mathsf{CA}_0$ of pure clone algebras is a  full subcategory of  $\mathcal{CA}$.

\begin{proposition}\label{prop:equiva}
 The category $\mathcal{CA}$ is equivalent to both $\mathcal{MCA}$ and $\mathsf{CA}_0$.
\end{proposition}

\begin{proof} Two categories are equivalent if and only if they have isomorphic skeletons. The categories $\mathcal{CA}$, $\mathcal{MCA}$ and $\mathsf{CA}_0$ have the same skeleton.
%
%
%
\end{proof}

We denote by $\mathcal{VAR}$  the category whose objects are varieties of algebras and whose arrows are interpretations of varieties (see Section \ref{sec:alg}).

\begin{theorem}\label{thm:iso} The categories $\mathcal{VAR}$ and  $\mathcal{MCA}$ are categorically isomorphic.
 There is a bijection between the class of all varieties of algebras and the class of all minimal clone algebras:
 \[
\begin{array}{rcl}
 \text{Variety $\mathcal V$ of $\tau$-algebras} & \mapsto  & \text{clone $\mathcal V$--algebra $\mathbf{Cl}(\mathcal V)$}  \\
\text{Minimal clone $\tau$-algebra $\mathbf C=(\mathbf C_\tau,q_n^\mathbf C,\e_i^\mathbf C)$}  & \mapsto  & \text{Variety $\mathrm{Var}(\mathbf C_\tau)$ generated by $\mathbf C_\tau$}.  \\
\end{array}
\]
We have $$\mathcal V= \mathrm{Var}(\mathbf{Cl}(\mathcal V)_\tau);\qquad \mathbf C \cong \mathbf{Cl}(\mathrm{Var}(\mathbf C_\tau)).$$
Moreover, there is a bijective correspondence between the sets $\mathrm{Hom}_{\mathcal{VAR}}(\mathcal V,\mathcal W)$ of interpretations and the set $\mathrm{Hom}_{\mathcal{CA}}(\mathbf{Cl}(\mathcal V),\mathbf{Cl}(\mathcal W))$ of pure homomorphisms.
\end{theorem}

\begin{proof} The first part of the theorem follows from Theorem \ref{thm:ch2}.
We now prove that interpretations of varieties and
 pure  homomorphisms of minimal clone algebras are in bijective correspondence. 

Let $\mathcal V$ be a variety of type $\tau$,  $\mathcal W$ be a variety of type $\nu$, $\mathbf C= \mathbf{Cl}(\mathcal V)$ and $\mathbf D= \mathbf{Cl}(\mathcal W)$. We recall that $\mathbf C_\tau$ is the free algebra over a countable set of generators in variety $\mathcal V$. Similarly for $\mathbf D_\nu$.
An interpretation $f$ of $\mathcal V$ into  $\mathcal W$ determines a pure homomorphism $F$ of $\mathbf C$ into $\mathbf D$.
If $t$ is a $\tau$-term, then we denote by $t^\mathbf C$ its equivalence class in the free algebra $\mathbf C_\tau$. Then we define:
\begin{itemize}
\item If $t=\sigma(t_1,\dots,t_n)$, then $F(\sigma^\mathbf C(t_1^\mathbf C,\dots,t_n^\mathbf C)) =q_n^\mathbf D(f(\sigma)^\mathbf D,F(t_1^\mathbf C),\dots,F(t_n^\mathbf C))$ for every $c\in \tau$ of arity $n>0$.
\item $F(c^\mathbf C)=f(c)^\mathbf D$ for every $c\in \tau$ of arity $0$. Notice that $f(c)$ is a unary $\nu$-term such that $f(c)^\mathbf D$ is zero-dimensional in $\mathbf D$.
\end{itemize}

For the converse, let $\mathbf C$ be a minimal clone $\tau$-algebra, $\mathbf D$ be a minimal clone $\nu$-algebra and $F$ be a pure homomorphism from $\mathbf C$ into $\mathbf D$. Then, for every $n$-ary operator $\sigma\in \tau$ ($n > 0$), we define $f(\sigma)$ to be any $\nu$-term $t=t(v_1,\dots,v_n)$ belonging to $\sigma^\mathbf D(\e_1^\mathbf D,\dots,\e_n^\mathbf D)$ (see Proposition \ref{prop:free}(3)). If $c\in\tau$ is a nullary operator, then we define $f(c)$ to be any $\nu$-term $t=t(v_1)$ belonging to $c^\mathbf D$ (see Proposition \ref{prop:free}(4)).
$f$ is an interpretation from $\mathrm{Var}(\mathbf C_\tau)$ into $\mathrm{Var}(\mathbf D_\nu)$.
\end{proof}

The following corollary is a reformulation of \cite[Theorem 4.140]{mac}. 

\begin{corollary} Two varieties $\mathcal V$ and $\mathcal W$ are isomorphic in the category $\mathcal{VAR}$ (equivalent in the terminology of  \cite[Theorem 4.140]{mac})  if and only if there is a pure isomorphism from $\mathbf{Cl}(\mathcal V)$ onto $\mathbf{Cl}(\mathcal W)$.
\end{corollary}

The common skeleton of $\mathsf{CA}_0$, $\mathcal{MCA}$ and $\mathcal{CA}$  is  the lattice of interpretability types of the varieties.

Hereafter, we identify the categories $\mathcal{MCA}$ and $\mathcal{VAR}$.

Given a clone algebra $\mathbf C$, recall from Definition \ref{def:ctype} the definition of the type $\rho_\mathbf C$ and of the clone $\rho_\mathbf C$-algebra $\overline{R}_{\mathbf C}$.

\begin{definition} \label{def:catprod}
  The \emph{categorical product} $\mathbf C\odot \mathbf D$  of $\mathbf C,\mathbf D\in\mathcal{MCA}$ is defined as the clone $\rho_{\mathbf C_0\times\mathbf D_0}$-algebra $\overline{R}_{\mathbf C_0\times\mathbf D_0}$ of all $\mathbf C_0\times\mathbf D_0$-representable functions. 
\end{definition}

$\mathbf C\odot \mathbf D$ is minimal by Proposition \ref{exa:rc}(i).  Moreover,  $\mathbf C\odot \mathbf D$ is the product of $\mathbf C$ and $\mathbf D$  in $\mathcal{MCA}$, because the categories $\mathcal{MCA}$ and $\mathsf{CA}_0$ are equivalent and  $(\mathbf C\odot \mathbf D)_0=\mathbf C_0\times\mathbf D_0$ is the product of $\mathbf C_0$ and $\mathbf D_0$ in the variety $\mathsf{CA}_0$ of pure clone algebras.

\bigskip

Let $\mathcal V_1,\dots,\mathcal V_n$ be subvarieties of a variety $\mathcal V$.
We recall from Section \ref{sec:alg} the definition of product $\mathcal V_1\times\dots\times\mathcal V_n$ of similar varieties 
(we advertise the reader that this product is not the categorical product).
The equational theory generated by the union $\bigcup_{i=1}^n Eq(\mathcal V_i)$ axiomatises 
$\mathcal V_1\cap\dots\cap\mathcal V_n$, while the join $\mathcal V_1\lor\dots\lor\mathcal V_n$ (in the lattice of subvarieties of $\mathcal V$) is axiomatised by 
$\bigcap_{i=1}^n Eq(\mathcal V_i)$. The join $\mathcal V_1\lor\dots\lor\mathcal V_n$ contains $\mathcal V_1\times\dots\times\mathcal V_n$.


The following theorem provides necessary and sufficient conditions for the independence of varieties, improving a theorem on independent varieties by Gr\"atzer et al. \cite{GLP} (see also \cite{KP09,KPL13}).

Recall from Proposition \ref{prop:rr} that the set of representable functions of a clone algebra is a clone.

\begin{theorem}\label{prop:varind} Let $\mathbf C$ and $\mathbf D$ be minimal $\mathsf{CA}_\tau$s and let $\mathbf E=\mathbf C\times\mathbf D$ be the clone $\tau$-algebra that is the product of $\mathbf C$ and $\mathbf D$ in the variety $\mathsf{CA}_\tau$. Then the following conditions are equivalent:
\begin{enumerate}
\item $\mathbf E$ is minimal.
\item $\mathrm{Var}(\mathbf C_\tau)$ and $\mathrm{Var}(\mathbf D_\tau)$ are independent.
\item $\mathrm{Clo}\,\mathbf E_\tau= R_{\mathbf E}$, where  $R_{\mathbf E}$ is the clone of the $\mathbf E$-representable functions and $\mathrm{Clo}\,\mathbf E_\tau$ is the clone of term operations of the $\tau$-algebra $\mathbf E_\tau$ (the $\tau$-reduct of the clone $\tau$-algebra $\mathbf E$).
\end{enumerate}
 If one of the above equivalent conditions holds, then $\mathrm{Var}(\mathbf E_\tau)=\mathrm{Var}(\mathbf C_\tau)\times \mathrm{Var}(\mathbf D_\tau)=\mathrm{Var}(\mathbf C_\tau)\lor \mathrm{Var}(\mathbf D_\tau)$, where the join $\lor$ is taken in the lattice of subvarieties of $\mathrm{Var}(\mathbf E_\tau)$. 
\end{theorem}

%

\begin{proof} 

(1 $\Rightarrow$ 2) By Proposition  \ref{prop:central} there exists 
 a $2$-central element $c=(\e_1^{\mathbf C},\e_2^{\mathbf D})\in E$ of dimension $2$ such that $\mathbf C \cong \mathbf E/\theta(c,\e_1^\mathbf E)$ and $\mathbf D \cong \mathbf E/\theta(c,\e_2^\mathbf E)$. 
 By Lemma \ref{lem:min}(i) and  the minimality of $\mathbf E$ there exists a ground $\tau(\e)$-term $t=t(\e_1,\e_2)$ such that $c=t^\mathbf C$. Let $t^*=t^*(v_1,v_2)$ be the $\tau$-term translation of $t$ (see Section \ref{sec:mca}). By $\mathbf C \cong \mathbf E/\theta(c,\e_1^\mathbf E)$ and $\mathbf D \cong \mathbf E/\theta(c,\e_2^\mathbf E)$ we get $\mathrm{Var}(\mathbf C_\tau) \models t^*(v_1,v_2)=v_1$ and $\mathrm{Var}(\mathbf D_\tau) \models t^*(v_1,v_2)=v_2$. Hence, the varieties $\mathrm{Var}(\mathbf C_\tau)$ and $\mathrm{Var}(\mathbf D_\tau)$ are independent.

(2 $\Rightarrow$ 1) Let $t(v_1,v_2)$ be a $\tau$-term such that $\mathrm{Var}(\mathbf C_\tau) \models t(v_1,v_2)=v_1$ and $\mathrm{Var}(\mathbf D_\tau) \models t(v_1,v_2)=v_2$. 
Let $(a,b)\in E$. Then $a\in C$ and $b\in D$. Since $\mathbf C$ and $\mathbf D$ are minimal, then by Lemma \ref{lem:min}(i) there exist two $\tau(\e)$-terms $u_1$ and $u_2$ such that $a=u_1^\mathbf C$ and $b=u_2^\mathbf D$. 
Then $\mathbf E$ is minimal, because the pair $(a,b)\in E$ coincides with the interpretation of the $\tau(\e)$-term $t(u_1,u_2)$. 

(1 $\Leftrightarrow$ 3) By Proposition \ref{exa:rc}(ii).

We now prove the last condition If $\mathbf E$ is minimal, then  by Theorem \ref{thm:ch2} $\mathbf E_\tau$ is the free algebra of the variety $\mathrm{Var}(\mathbf E_\tau)$ and $\mathbf E_\tau = \mathbf C_\tau\times \mathbf D_\tau$. Then the decomposition operator $t(v_1,v_2)^{\mathbf E_\tau}$ giving the decomposition $\mathbf E_\tau = \mathbf C_\tau\times \mathbf D_\tau$ provides the decomposition $\mathrm{Var}(\mathbf E_\tau)=\mathrm{Var}(\mathbf C_\tau)\times \mathrm{Var}(\mathbf D_\tau)$.
%
\end{proof}

We leave  to the reader the interpretation of the above theorem in the category $\mathcal{VAR}$.

\begin{remark}
 If $\tau$ is a type of unary operators, it is well known that there are no independent varieties of type $\tau$. By Theorem \ref{prop:varind} the algebra $\mathbf E=\mathbf C\times \mathbf D$ is never minimal, because every unary term operation $t^\mathbf E$ cannot be a nontrivial decomposition operator on $\mathbf E$. 
 \end{remark}
 
 \begin{remark}
 Gr\"atzer et al.  \cite{GLP}  provide examples of varieties $\mathcal V$ and $\mathcal W$ of the same type $\tau$ such that $\mathcal V \lor \mathcal W = \mathcal V \times \mathcal W$, but $\mathcal V$ and $\mathcal W$ are not independent. Let $\mathbf E=\mathbf{Cl}(\mathcal V)\times \mathbf{Cl}(\mathcal W)$. 
 We wonder whether the stronger condition $\mathrm{Var}(\mathbf E_\tau)= \mathcal V \times \mathcal W=\mathcal V \lor \mathcal W$  implies that $\mathcal V$ and $\mathcal W$ are independent.
  \end{remark}

%
%

We conclude this section with a generalisation of Theorem \ref{prop:varind} to  clone algebras of different type.

\begin{definition}\label{def:expa}
   Let $\mathbf C$ be a clone $\tau$-algebra, $\mathbf D$ be a  clone $\nu$-algebra and $f: \mathbf C\to \mathbf D$ be a pure homomorphism. The \emph{$f$-expansion} of $\mathbf D$ is the clone $\tau$-algebra $\mathbf D^f=(\mathbf D^f_\tau, q_n^\mathbf D, \e_i^\mathbf D)$, where $\mathbf D^f_\tau=(D,\sigma^{\mathbf D^f})_{\sigma\in\tau}$ and  $\sigma^{\mathbf D^f}$ ($\sigma\in\tau$ of arity $n$) is the $n$-ary operation such that
$\sigma^{\mathbf D^f}(\e_1^\mathbf D,\dots,\e_n^\mathbf D)= f(\sigma^\mathbf C(\e_1^\mathbf C,\dots,\e_n^\mathbf C))$. 
\end{definition}

\begin{lemma}\label{lem:expa} In the hypotheses of Definition \ref{def:expa} we have:
\begin{enumerate}
\item   The map $f:C\to D$ is a homomorphism of $\tau$-algebras from $\mathbf C$ into $\mathbf D^f$;
\item If $\mathbf C$ is minimal and $f$ is onto, then $\mathbf D^f$ is also minimal.
\end{enumerate}
\end{lemma}


\begin{theorem}\label{thm:gra}
  Let $\mathbf C_i$ be a minimal clone $\tau_i$-algebra ($i=1,2$), $\mathbf C_1\odot\mathbf C_2$ be the categorical product in $\mathcal{MCA}$ and  $\nu$ be the type of $\mathbf C_1\odot\mathbf C_2$. 
Then the following conditions hold: 
\begin{enumerate}
\item The $\pi_i$-expansion $\mathbf C_i^{\pi_i}$ of $\mathbf C_i$ (see Definition \ref{def:expa}) is a minimal clone $\nu$-algebra, where $\pi_i$  is the projection from  $\mathbf C_1\odot\mathbf C_2$ into $\mathbf C_i$   ($i=1,2$);
\item $\mathbf C_i^{\pi_i}$ is purely isomorphic to  $\mathbf C_i$ ($i=1,2$);
\item $\mathbf C_1^{\pi_1}\times \mathbf C_2^{\pi_2}= \mathbf C_1\odot\mathbf C_2$, where the product $\mathbf C_1^{\pi_1}\times \mathbf C_2^{\pi_2}$ is taken in the variety $\mathsf{CA}_\nu$;
\item The varieties $\mathrm{Var}\,(\mathbf C_1^{\pi_1})_\nu$ and  $\mathrm{Var}\,(\mathbf C_2^{\pi_2})_\nu$ are independent;
\item $\mathrm{Var}\, (\mathbf C_1^{\pi_1}\times \mathbf C_2^{\pi_2})_\nu=\mathrm{Var}\,(\mathbf C_1^{\pi_1})_\nu\times \mathrm{Var}\,(\mathbf C_2^{\pi_2})_\nu= \mathrm{Var}\,(\mathbf C_1)_{\tau_1}\odot \mathrm{Var}\,(\mathbf C_2)_{\tau_2}$.
\end{enumerate}
\end{theorem}

\begin{proof} (1) By Lemma \ref{lem:expa}, because $\pi_i C_1\times C_2\to C_i$ is a pure homomorphism from $\mathbf C_1\odot\mathbf C_2$ onto $\mathbf C_i$.

(2) $\mathbf C_i^{\pi_i}$ and $\mathbf C_i$ have the same pure reduct.

(3) By Definition \ref{def:catprod} the pure reduct of $\mathbf C_1\odot\mathbf C_2$ is $(\mathbf C_1)_0\times (\mathbf C_2)_0$. The conclusion follows from the definition of $\mathbf C_i^{\pi_i}$.

(4) By (3) and Theorem \ref{prop:varind}, because $\mathbf C_i^{\pi_i}$ ($i=1,2$) is a minimal clone $\nu$-algebra. 

(5) By Theorem \ref{prop:varind} and the fact that the categories $\mathcal{MCA}$ and $\mathcal{VAR}$ are isomorphic.
\end{proof}

\section*{Conclusions}

All the original results presented in this paper stem from the definition of clone algebra, that is, therefore, the main contribution of this work.  The results listed below give evidence of the relevance
of the notion of clone algebra, that goes beyond providing a neat algebraic treatment of clones. Indeed, unexpected applications and promising further direction, as those we are going to describe, are often marks of the relevance and versatility of a new mathematical notion.
\begin{itemize}
\item Theorem \ref{thm:main}, the representation theorem, 
ensures that the variety of clone algebras provides 
an algebraic theory of clones.
\item In Theorem \ref{thm:BM}, by endowing free algebras with the structure 
of clone algebras, and clone algebras with the structure of free algebras, we are able to characterise the lattices of 
equational theories, thus providing a possible answer to a classical 
open question. 
\item Theorem  \ref{prop:varind}, and other results presented in Section \ref{sec:allvar},
show that clone algebras may be used to study other classical topics in universal algebras, like the equivalence and the independence of varieties.
\end{itemize}

The focus of the present paper is on the representation theorems and their meaning for the theory of clones and $\omega$-clones, 
 and partly on the categorical aspects of clone algebras illustrated in the last section of the paper. A closer examination of potential implications  to universal algebra is deferred to future work that is currently in progress. 
 We have here space to describe two possible directions of research.

We intend to analyse the relationship between a variety $\mathcal V$ of pure clone algebras and the corresponding subcategory $\mathbb{C}(\mathcal V)$ of
$\mathcal{VAR}$, where a variety $\mathcal W$ of $\tau$-algebras belongs to  $\mathbb{C}(\mathcal V)$ if the pure reduct of the clone $\mathcal W$-algebra $\mathbf{Cl}(\mathcal W)$ is an element of $\mathcal V$. We explain with an example the kind of connection we are looking for. 
Assume that the class of indecomposable members of $\mathcal V$ is a universal class.
  Since $\mathcal V$ is also a variety of $2\mathrm{CH}$s, then by \cite[Theorems 3.8,  3.9]{first} every member of $\mathcal V$ is a weak Boolean product of a family of indecomposable members of $\mathcal V$. We are interested in understanding how this weak Boolean representation influences the structure of the category $\mathbb{C}(\mathcal V)$.

There are some classical concepts of the theory of clones that have a more general and algebraic formulation within the theory of $\omega$-clones. For example, if $F$ is an $\omega$-clone and $f\in F$, then the centralizer $f^*$ (of the infinitary operations commuting with $f$) is a subalgebra of the pure $\mathsf{FCA}$ $\mathbf F=(F,q_n^\omega,\e_i^\omega)$. 
We are interested in understanding whether the centralizer is invariant up to isomorphism of pure $\mathsf{FCA}$s.



\section*{Appendix}
In this Appendix we conclude the proof of Theorem \ref{thm:main}.

In the following lemma we prove that  
a $\mathsf{RCA}_\tau$   $\mathbf B$ with value domain $\mathbf A$ can be embedded into the ultrapower $(\mathbf O_{ \mathbf A}^{(\omega)})^\omega/U$ of the full $\mathsf{FCA}_\tau$ $\mathbf O_{ \mathbf A}^{(\omega)}$,
for every   nonprincipal ultrafilter $U$  on $\omega$. 

\begin{lemma}\label{lem:ultrapower}
 Every $\mathsf{RCA}_\tau$ can be embedded  into an ultrapower of a $\mathsf{FCA}_\tau$. 
\end{lemma}

\begin{proof}
 Let  $U$ be a nonprincipal ultrafilter on $\omega$ that contains the set $\{j: j\geq i\}$ for every $i\in\omega$. 
$U$ does not contain finite sets.
Let $\mathbf O_{ \mathbf A,r}^{(\omega)}$ be the full $\mathsf{RCA}$ with value domain $\mathbf A$ and thread $r$. Let $F_{r, n}$ be the function defined in Lemma \ref{lem:red}. We prove that the map
$$h(\varphi)=\langle F_{r, n}(\varphi): n\in\omega\rangle /U,\quad \text{for all $\varphi\in \mathcal O_{A,r}^{(\omega)}$}$$
is an embedding of the full $\mathsf{RCA}$ $\mathbf O_{ \mathbf A,r}^{(\omega)}$ into the  ultrapower $(\mathbf O_{ \mathbf A}^{(\omega)})^\omega/U$ of the full $\mathsf{FCA}$ $\mathbf O_{ \mathbf A}^{(\omega)}$ with value domain $\mathbf A$.

We prove that $h$ is injective. If $h(\varphi)=h(\psi)$ then $\{ n: F_{r, n}(\varphi)=F_{r, n}(\psi)\}\in U$. Then, for every $i\in\omega$, by the hypothesis on $U$ we have:
$$\{j: j\geq i\}\cap \{ n: F_{r, n}(\varphi)=F_{r, n}(\psi)\}\ \text{is an infinite set}.$$
Then there exists an increasing sequence $k_1 < k_2 <\dots < k_i <\dots$ of natural numbers such that 
$k_i\in \{j: j\geq i\}\cap \{ n: F_{r, n}(\varphi)=F_{r, n}(\psi)\}$ and $k_i > k_{i-1}$.
Let $s\in A^\omega_r$ such that $s=r[s_1,\dots,s_m]$. 
Let $k_n>m$. Then $s=r[s_1,\dots,s_m] = r[s_1,\dots,s_m,r_{m+1},\dots,r_{k_n}]$ and  we have:
\[
\begin{array}{llll}
\varphi(s)  &  = &  \varphi(r[s_1,\dots,s_m,r_{m+1},\dots,r_{k_n}]) &\\
  &  = & F_{r, k_n}(\varphi)(s)  &\text{by def. $F_{r, k_n}$} \\
  &  = &    F_{r, k_n}(\psi)(s)&\text{by $k_n\in \{ n: F_{r, n}(\varphi)=F_{r, n}(\psi)\}$}\\
    &  = & \psi(r[s_1,\dots,s_m,r_{m+1},\dots,r_{k_n}])&\\
&  = & \psi(s).&
\end{array}
\]
By the arbitrariness of $s$ it follows that $\varphi=\psi$.
We now prove that $h$ is a homomorphism. 
\[\begin{array}{llll}
&   &h(q^r_k(\varphi,\psi_1,\dots,\psi_k))& \\
 & =  & \langle F_{r, n}(q^r_k(\varphi,\psi_1,\dots,\psi_k)): n\in\omega\rangle /U & \\
  & =  & \langle q^\omega_k(F_{r, n}(\varphi),F_{r, n}(\psi_1),\dots,F_{r, n}(\psi_k)): n\in\omega\rangle /U& \\
\end{array}
\]  
because by Lemma \ref{lem:red} $\{n:F_{r, n}(q^r_k(\varphi,\psi_1,\dots,\psi_k))=q^\omega_k(F_{r, n}(\varphi),F_{r, n}(\psi_1),\dots,F_{r, n}(\psi_k)) \}\supseteq \{n: n\geq k\}\in U$.
Let $\mathbf B=\mathbf O_{ \mathbf A}^{(\omega)}$. By definition of $q^{\mathbf B^\omega/U}_k$, we obtain
\[\begin{array}{llll}
   &   & q^{\mathbf B^\omega/U}_k(h(\varphi),h(\psi_1),\dots,h(\psi_k))&\\
  &  = &  q^{\mathbf B^\omega/U}_k (\langle F_{r, n}(\varphi): n\in\omega\rangle /U, \langle F_{r, n}(\psi_1): n\in\omega\rangle /U,\dots, \langle F_{r, n}(\psi_k): n\in\omega\rangle /U)&\\
  & =  & \langle q^\omega_k(F_{r, n}(\varphi),F_{r, n}(\psi_1),\dots,F_{r, n}(\psi_k)): n\in\omega\rangle /U.& \\  
\end{array}
\]
Moreover, 
$h(\e_i^r)=\langle F_{r, n}(\e_i^r): n\in\omega\rangle /U = \langle \e_i^\omega: n\in\omega\rangle /U$
because 
$\{n:F_{r, n}(\e_i^r)=\e_i^\omega \}\supseteq \{n: n\geq i\}\in U$.
A similar computation works for $\sigma\in\tau$.
\end{proof}

The product of a family of $\mathsf{FCA}_\tau$s can be embedded into a  $\mathsf{FCA}$ whose value domain is the product of the value domains of the family.

\begin{lemma}\label{lem:subprod}
The class $\mathbb I\,\mathsf{FCA}_\tau$ is  closed under subalgebras and direct products.
\end{lemma}

\begin{proof}
  The class of $\mathsf{FCA}_\tau$'s is trivially closed under subalgebras. It is also closed under products,
because $\prod_{i\in H} \mathbf B_i$, where $\mathbf B_i$ is a $\mathsf{FCA}_\tau$ with value domain $\mathbf A_i$,
can be embedded into the full $\mathsf{FCA}_\tau$ with value domain $\prod_{i\in H}\mathbf A_i$: the  sequence
$\langle \varphi_i: A_i^\omega \to A_i\in B_i\ |\ i\in H\rangle$ maps to $\varphi : (\prod_{i\in H} A_i)^\omega \to \prod_{i\in H} A_i$ defined by $\varphi(r)= \langle \varphi_i(\langle r_j(i):j\in\omega\rangle)\ |\ i\in H\rangle$.
\end{proof}


\begin{lemma}\label{lem:ultrafca}
Ultrapowers of $\mathsf{FCA}_\tau$s are isomorphic to $\mathsf{FCA}_\tau$s.
\end{lemma}

\begin{proof} 
 Let $\mathbf B$ be a $\mathsf{FCA}$ with value domain $\mathbf A$,  $K$ be a set and $U$ be any ultrafilter on $K$.  By Lemma \ref{lem:subprod} we get the conclusion if the ultrapower $\mathbf B^K/U$ is isomorphic to a subdirect product of $\mathsf{FCA}$s.

 A choice function is a function
 $ch: A^K/U\to A^K$  such that $ch(w/U)\in w/U$ for every $w\in A^K$ (see \cite[Section 6]{SG99}). 
 Any choice function $ch$ induces a function $ch^+ : (A^K/U)^\omega \to (A^\omega)^K$:
 $$ch^+(r)_k = \langle ch(r_i)_k: i\in \omega \rangle,\qquad\text{for every $r\in (A^K/U)^\omega$ and $k\in K$.} $$
 We use the choice function $ch$ to define a function $h_{ch}: B^K/U\to \mathcal O_{A^K/U}^{(\omega)}$ as follows:
 $$h_{ch}(u/U)(r)=\langle u_k(ch^+(r)_k): k\in K\rangle/U,\quad\text{for every $u\in B^K$ and $r\in (A^K/U)^\omega$}.$$
 The map $h_{ch}$ is a homomorphism from the ultrapower $\mathbf B^K/U$ into the full $\mathsf{FCA}$  $\mathbf O_{\mathbf A^K/U}^{(\omega)}$ with value domain $\mathbf A^K/U$. 
 Let $\mathbf C:= \mathbf B^K/U$, $\mathbf D:=\mathbf O_{\mathbf A^K/U}^{(\omega)}$, $r\in (A^K/U)^\omega$ and $s_k:= ch^+(r)_k\in A^\omega$.
  \[
\begin{array}{llll}
h_{ch}(\e_i^\mathbf C)(r)&=& h_{ch}(\langle\e_i^\mathbf B:k\in K\rangle/U)(r)&\\
&=&h_{ch}(\langle\e_i^\omega:k\in K\rangle/U)(r)&\text{by $\mathbf B\in \mathsf{FCA}$ and Lemma \ref{lem:fun}}\\
&=&\langle \e_i^\omega(ch^+(r)_k): k\in K\rangle/U&\text{by def. $h_{ch}$}\\
&=&\langle ch(r_i)_k: k\in K\rangle/U &\text{by def. $ch^+$ and $\e_i^\omega$}\\ 
&=& ch(r_i)/U&\\ 
&=& r_i&\text{by $ch(r_i)\in r_i$}\\ 
 \end{array}
\]
Without loss of generality, we prove that $h_{ch}$ preserves $q_2^\mathbf C$.
 \[
\begin{array}{llll}
&&h_{ch}(q_2^\mathbf C(u/U,w^1/U,w^2/U))(r)& \\
 &  = & h_{ch}( \langle q_2^\mathbf B(u_k,w^1_k,w^2_k): k\in K\rangle/U)(r) &\text{by def. $q_2^\mathbf C$}\\
 &  = &  \langle q_2^\mathbf B(u_k,w^1_k,w^2_k)(s_k): k\in K\rangle/U &\text{by def. $h_{ch}$}\\
  & =  & \langle q_2^\omega(u_k,w^1_k,w^2_k)(s_k): k\in K\rangle/U&\text{by $\mathbf B\in \mathsf{FCA}$ and Lemma \ref{lem:fun}}  \\
  & =  &    \langle u_k(s_k[w^1_k(s_k),w^2_k(s_k)]): k\in K\rangle/U&\text{by def.  $q_2^\omega$}\\
  & =  &    \langle u_k(ch^+(r)_k[w^1_k(s_k),w^2_k(s_k)]): k\in K\rangle/U&\text{by def.  $s_k$}\\
  & =  &    \langle u_k( \langle ch(r_i)_k: i\in \omega \rangle[w^1_k(s_k),w^2_k(s_k)]): k\in K\rangle/U&\text{by def.  $ch^+$}\\
    & =  &    \langle u_k(w^1_k(s_k),w^2_k(s_k), ch(r_3)_k,ch(r_4)_k,\dots): k\in K\rangle/U&\\

\end{array}
\]

Let $t=r[\langle w^1_j(s_j): j\in K\rangle/U, \langle w^2_j(s_j): j\in K\rangle/U]$. 
 \[
\begin{array}{llll}
&&q_2^\omega(h_{ch}(u/U),h_{ch}(w^1/U),h_{ch}(w^2/U))(r)&\\
&=& h_{ch}(u/U)(r[h_{ch}(w^1/U)(r),h_{ch}(w^2/U)(r)]) & \text{by def.  $q_2^\omega$}\\
&=&h_{ch}(u/U)(r[\langle w^1_j(s_j): j\in K\rangle/U, \langle w^2_j(s_j): j\in K\rangle/U])&\text{by def.  $h_{ch}$}\\
&=& \langle u_k(ch^+(t)_k): k\in K\rangle/U&  \text{by def.  $h_{ch}$} \\
&=& \langle u_k(\langle ch(t_i)_k : i\in \omega \rangle): k\in K\rangle/U& \text{by def.  $ch^+$}  \\
&=& \langle u_k( ch(t_1)_k, ch(t_2)_k, ch(t_3)_k,ch(t_4)_k,\dots): k\in K\rangle/U&   \\
&=& \langle u_k( ch(t_1)_k, ch(t_2)_k, ch(r_3)_k,ch(r_4)_k,\dots): k\in K\rangle/U& \text{by def.  $t$}  \\
&=& \langle u_k( ch(\langle w^1_j(s_j): j\in K\rangle/U)_k, ch(\langle w^2_j(s_j): j\in K\rangle/U)_k, &\\
&&\qquad\qquad\qquad ch(r_3)_k,ch(r_4)_k,\dots): k\in K\rangle/U& \text{by def.  $t$}  \\
    & =  &    \langle u_k(w^1_k(s_k),w^2_k(s_k), ch(r_3)_k,ch(r_4)_k,\dots): k\in K\rangle/U&\\
\end{array}
\]
because $\{k\in K : ch(\langle w^i_j(s_j): j\in K\rangle/U)_k =w^i_k(s_k) \}\in U$ ($i=1,2$). A similar proof works for $\sigma\in\tau$.
Hence a homomorphic image of the ultrapower $\mathbf B^K/U$ is isomorphic to a $\mathsf{FCA}$.
 
 By \cite[Lemma 8.2]{BS} we have that the ultrapower $\mathbf B^K/U$ is isomorphic to a subdirect product of $\mathsf{FCA}$s if the family of maps $h_{ch}$ (indexed by choice functions) satisfies the following property: for all distinct $w/U, u/U\in B^K/U$ there exists a choice function $ch$ for which $h_{ch}(w/U)\neq h_{ch}(u/U)$. We are going to prove this fact.
 
 Let $w = \langle w_i:A^\omega\to A:i\in K\rangle$ and  $u = \langle u_i:A^\omega\to A:i\in K\rangle$.
 For every $j\in K$, let $\rho_j\in A^\omega$ such that $w_j(\rho_j)\neq u_j(\rho_j)$ whenever $w_j\neq u_j$. For every $i\in\omega$, let $r_i\in A^K$ such that $r_i(j)=\rho_j(i)$ for all $j\in K$. Define $s\in (A^K/U)^\omega$ as $s_i=r_i/U$ and consider any choice function $ch$ such that $ch(s_i)=r_i$. Then we have
$h_{ch}(w/U)(s)=\langle w_j(\rho_j): j\in K)\rangle/U$ and $h_{ch}(u/U)(s)=\langle u_j(\rho_j): j\in K)\rangle/U$, but
$\{j: w_j(\rho_j)=u_j (\rho_j)\}=\{j: w_j= u_j\} \notin U$.
since $w/U\neq u/U$. It follows that $h_{ch}(w/U)(s)\neq h_{ch}(u/U)(s)$ and then $h_{ch}(w/U)\neq h_{ch}(u/U)$.
 \end{proof}

\begin{corollary} Let $\mathbf B$ be a $\mathsf{FCA}$ with value domain $\mathbf A$ and $U$ be an ultrafilter on $\omega$.
Then there exists a set $J$ of the same cardinality as $B$ such that the ultrapower $\mathbf B^\omega/U$ is isomorphic to a $\mathsf{FCA}$ with value domain
$(\mathbf A^\omega/U)^J$.
\end{corollary}

\end{document}